\documentclass[10pt]{article}
\usepackage{paper,macros,graphicx}
\usepackage{epstopdf}
\usepackage{url}
\usepackage{color}
\definecolor{darkgreen}{rgb}{0.0, 0.70, 0.0}

\usepackage{subcaption}

\usepackage[normalem]{ulem}

\def\R{\mathbb{R}}
\newcommand{\Var}[1]{\mathrm{Var}\left \{ #1\right\}}   

\newcommand{\cN}{\mathcal{N}}      

\usepackage{mathtools}

\usepackage[colorlinks=true]{hyperref}
\usepackage{framed}
\usepackage{subcaption}

\newcommand{\ks}[1]{{\color{black}#1}}

\usepackage[numbers, sort&compress ]{natbib}
\numberwithin{equation}{section}
\usepackage{bbm}

\usepackage{booktabs}
\usepackage{array}
\usepackage{arydshln}
\renewcommand{\thefootnote}{\arabic{footnote}}

\title{A Theoretical and Empirical Comparison of Gradient Approximations in Derivative-Free Optimization}

\author{A. S. Berahas\footnotemark[1]
   \and L. Cao\footnotemark[2]
   \and K. Choromanski\footnotemark[3]
   \and K. Scheinberg\footnotemark[4]\ \footnotemark[5]}

\begin{document}
\maketitle

\renewcommand{\thefootnote}{\fnsymbol{footnote}}
\footnotetext[1]{Department of Industrial and Operations Engineering, University of Michigan, Ann Arbor, MI, USA; E-mail: \url{albertberahas@gmail.com}}
\footnotetext[2]{Department of Industrial and Systems Engineering, Lehigh University, Bethlehem, PA, USA; E-mail: \url{lic314@lehigh.edu}}
\footnotetext[3]{Google Brain, New York, NY, USA; Email: \url{kchoro@google.com}}
\footnotetext[4]{Department of Operations Research and Information Engineering, Cornell University, Ithaca, NY, USA; E-mail: \url{katyas@cornell.edu}}
\footnotetext[5]{Corresponding author.}
\renewcommand{\thefootnote}{\arabic{footnote}}

\begin{abstract}

In this paper, we analyze several methods for approximating gradients of noisy functions using only function values. These methods include finite differences, linear interpolation, Gaussian smoothing and smoothing on a sphere. The methods differ in the number of functions sampled, the choice of the sample points, and the way in which the gradient approximations are derived. For each method, we derive bounds on the number of samples and the sampling radius which guarantee favorable convergence properties for a line search or fixed step size descent method. To this end, we use the results in \cite{berahas2019global} and show how each method can satisfy the sufficient conditions, possibly only with some sufficiently large probability at each iteration, as happens to be the case with Gaussian smoothing and smoothing on a sphere. Finally, we present numerical results evaluating the quality of the gradient approximations as well as their performance in conjunction with a line search derivative-free optimization algorithm. 

\end{abstract}

\section{Introduction}

We consider an unconstrained optimization problem of the form
\begin{align*} 
	\min_{x \in \mathbb{R}^n} \phi(x), 
\end{align*}
where $f(x)=\phi(x) + \epsilon(x)$ is computable, while $\phi(x)$ may not be. 
In other words, $f: \mathbb{R}^n\rightarrow \mathbb{R}$ is a possibly noisy approximation of a smooth function $\phi: \mathbb{R}^n\rightarrow \mathbb{R}$, and the goal is to minimize $\phi$.   
The noise in our analysis can be deterministic, stochastic or adversarial, however,  we assume that the noise is bounded uniformly, i.e., there exists a constant $\epsilon_f \geq 0$ such that $|\epsilon(x)|\leq \epsilon_f$ for all $x \in \mathbb{R}^n$. \ks {Thus, even then the noise is stochastic, we replace it with the worst case bound $\epsilon_f$, instead of treating it as a random variable. We assume that $\epsilon_f$  is {\em known}, which is a key assumption in our analysis.}
While this may seem a strong assumption, it is often satisfied in practice when $f(x)$ is the result of a computer code aimed at computing $\phi(x)$, but that has inaccuracies due to internal discretization \cite{more2009benchmarking,more2011estimating}. Another common setting in which the assumption is satisfied is when $f(x)$ is a nonsmooth function and $\phi(x)$ is its smooth approximation; see e.g., \cite{nesterov2017random, maggiar2018derivative}. \ks{In practice $\epsilon_f$ can be obtained with the cost of several function evaluations \cite{berahas2019derivative,more2011estimating}. 
It is important to note that while we assume $|\epsilon(x)|\leq \epsilon_f$ for all $x \in \mathbb{R}^n$, for simplicity, we, in fact, only use the bound on the noise at the points which are used as sample points to estimate  $\nabla \phi(x)$ for a specific $x$. Thus when the sample points are known to lie in a ball of a given radius around a fixed $x$ (as is the case for several gradient estimate methods we consider here), then our analysis can be applied if $\epsilon_f$ bounds the noise only in that given ball.}

In this paper, we do not assume that $\nabla \phi(x)$ is computable or available, but we do assume that $\nabla \phi(x)$ is Lipschitz continuous and that knowledge of an upper bound on the Lipschitz constant is available.
Such problems arise in many fields such as Derivative-Free Optimization (DFO) \cite{ConnScheVice08c,LarsMeniWild2019,brent2013algorithms,ghadimi2013stochastic,kiefer1952stochastic,spall2005introduction,berahas2019derivative}, Simulation Optimization \cite{pasupathy2018sampling,shashaani2018astro} and Machine Learning \cite{choro,fazel2018global,liu2018zeroth,TRPO,shamir2017optimal,duchi2015optimal,jamieson2012query,bollapragada2019adaptive,bogolubsky2016learning}. There have been a number of works  analyzing the case when $\epsilon(x)$ is a random function with zero mean (not necessarily bounded). The results obtained for stochastic noise, and the corresponding optimization methods, are different than those for bounded arbitrary noise. 

One common approach to optimizing functions without derivatives is to compute an estimate of the gradient $\nabla \phi(x)$ at the point $x$, denoted by $g(x)$, using (noisy) function values and then apply a gradient based method with $g(x)$. The most straightforward way to estimate $\nabla \phi(x)$  is to use forward finite differences by sampling one point near $x$ along each of the $n$ coordinates. Alternatively, one can estimate $\nabla \phi(x)$ via central finite differences where two points are sampled along each coordinate in both directions. As a generalization of the finite difference approach, $g(x)$ can be computed via linear interpolation. This approach also requires $n$ sample points near $x$, however, the location of the sample points can be chosen arbitrarily, as long as they form a set of $n$ linearly independent directions from $x$. Linear interpolation is very useful when coupled with an optimization algorithm that (potentially) reuses some of the sample function values computed at prior iterations, thus avoiding the need to compute $n+1$ new function values at each iteration.  The accuracy of the resulting gradient approximation depends on the conditioning of the matrix $Q_{\mathcal{X}}$, which is the matrix whose rows are the linearly independent directions formed by the sample points.  An extensive study of optimization methods based on interpolation gradients can be found in \cite{ConnScheVice08c}.


An alternative approach for estimating gradients using an arbitrary number of function value samples is based on random sample points.  The essence of these methods is to compute gradient estimates as a sum of estimates of directional derivatives along random (e.g., Gaussian) directions.
\ks{Using randomized directional derivative estimates was pioneered in \cite{nesterov2017random}, where these estimates are computed using only 
two function evaluations per iteration, as opposed to $n+1$ evaluations required by the finite difference method. While this appears advantageous, the consequence is that the  step size parameter has to 
be $n$ times smaller and thus, the overall iteration complexity  $n$ times larger,  than those for methods relying on accurate gradient approximations such as finite difference. 
The question then arises  - can using {\em multiple} randomized directional derivative estimates have practical or theoretical advantage over finite difference schemes?   Such methods have become popular in recent literature for policy optimization in reinforcement learning (RL)  \cite{ES,choro, choro2, rowland, fazel2018global,TRPO} as a particular case of simulation optimization. For example, in  \cite{ES} a gradient approximation is constructed 
by averaging a relatively large number directional derivative estimates along Gaussian directions \cite{nesterov2017random}. In \cite{fazel2018global}
a large number of directional derivative estimates along random unit sphere directions is used. In each case the number of these directions seems to be chosen to fit the specific method and this choice is somewhat obscure. 

Our goal is to derive bounds on the number of directional derivative estimates along random directions that are needed to establish gradient approximation that are comparable in accuracy to those obtained by a traditional finite difference schemes. What we observe is that this number is at least as large as $n$, and it is thus our conclusion that these new methods offer {\em no theoretical or practical advantage} at least in the setting fo standard optimization algorithms, such as line search. The randomized schemes may offer some advantage in some noisy optimization setting, since randomization itself may provide some algorithmic robustness, but such setting is yet to be discovered and analyzed. }

Overall, the methods we consider in this paper compute an estimate of the gradient $\nabla \phi(x)$ (denoted by $g(x)$),  as follows
 \begin{align}		\label{eq:GSG_intro}
	g(x) =  \sum_{i=1}^N \frac{f(x+\sigma u_i) - f(x)}{\sigma} \tilde u_i, 
\end{align}
 or using the central (symmetric or antithetic)  version
\begin{align}		\label{eq:cGSG_intro}
	g(x) = \sum_{i=1}^N \frac{f(x+\sigma u_i) - f(x-\sigma u_i)}{2\sigma} \tilde u_i,
\end{align}
where $\{u_i: i=1, \ldots, N\}$\footnote{Throughout the paper, $N$ denotes the \emph{size of the sample set} $\{u_i: i=1, \ldots, N\}$. Note that for the central versions of the gradient approximations, the number of \emph{sampled functions} is equal to $2N$.} is a \emph{set of directions} that depend on the method, $\tilde u_i$ depends on $u_i$, and  $\sigma$ is the \emph{sampling radius}. 
In particular, for the finite difference methods $N=n$ and  $u_i=\tilde u_i=e_i$, where $e_i$ denotes the $i$-th column of the identity matrix. For interpolation, $N=n$, $\{u_i: i=1, \ldots, N\}$ is a set of arbitrary linearly independent vectors with $\|u_i\|\leq 1$\footnote{The norms used in this paper are Euclidean norms.} for all $i$, and $\tilde u_i$ are the columns of $Q_\mathcal{X}^{-1}$, where the $i$th row of $Q_\mathcal{X} \in \mathbb{R}^{N \times n}$ is $u_i$. 
A special case of linear interpolation has been explored in \cite{choro,choro2, rowland} where the $u_i$'s are random orthogonal directions; this approach can also be viewed as rotated finite differences.
In the case of Gaussian smoothing, the directions $u_i$ are random directions from a standard Gaussian distribution and $\tilde u_i=\frac{1}{N}u_i$. 
Finally, a variant of this method that selects the directions $u_i$ from a  uniform distribution on a unit sphere, where $\tilde u_i= \frac{n}{N} u_i$, has been explored in \cite{fazel2018global,flaxman2005online}. As is clear from \eqref{eq:cGSG_intro}, antithetic gradient approximations require $2N$ function evaluations.
Details about these methods are given in Section~\ref{sec:grad_approx}.


We are motivated by recent empirical use of these methods in the RL literature. In \cite{ES}, the authors showed that the Gaussian smoothing approach is an efficient way to compute gradient estimates when $N\sim n$. In follow-up works \cite{choro,choro2, rowland} it was shown empirically that better gradient estimates can be obtained for the same optimization problems by using interpolation with orthogonal directions.  While the numerical results in these works confirmed the feasibility of  use of \eqref{eq:GSG_intro} and \eqref{eq:cGSG_intro} for various RL benchmark sets  and different choice of  directions, there is no theoretical analysis, neither comparing the accuracy of resulting gradient estimates, nor analyzing the connection between such accuracy and downstream optimization gains. 
To the best of our best knowledge, there has been no systematic analysis of the accuracy of the (stochastic) gradient estimates used in the DFO literature (such as Gaussian smoothing and smoothing on a unit sphere) specifically in conjunction with requirements of obtaining descent directions. 


In this paper, we develop theoretical bounds on the gradient approximation errors $\|g(x)-\nabla \phi(x)\|$ for all aforementioned gradient estimation methods and show their dependence on the number of samples. Another key quantity we consider is the radius of sampling $\sigma$. In the absence of noise, $\sigma$ can be chosen arbitrarily small, however, when noise is present, small values of $\sigma$ can lead to large inaccuracies in the gradient estimates. We derive the values for $\sigma$ which ensure that the gradient estimates are sufficiently accurate and thus can be used in conjunction with efficient gradient based methods. 


A number of works have used smoothing techniques for gradient approximations  within   stochastic gradient descent schemes with a {\em fixed} step size parameter or a {\em predetermined} sequence of step size parameters; see e.g., \cite{nesterov2017random, ES, fazel2018global,flaxman2005online,duchi2015optimal,dvurechensky2020accelerated,bayandina2017gradient}. The complexity results derived in these papers depend on the assumptions made on the underlying functions as well as the algorithm employed. In \cite{nesterov2017random} the objective function is assumed to be deterministic, and the convergence rate that is obtained for a gradient method with gradients approximated via Gaussian smoothing is the same (in terms of dependence on the dimension $n$ and the iteration count) as for deterministic gradient descent. Notably, \cite{duchi2015optimal} establishes convergence rates with better dependence on the dimension, but worse dependence on the iteration count. This perhaps is not surprising since the objective function is assumed to be stochastic in \cite{duchi2015optimal}. 

In this paper we address functions with bounded noise (so more general than \cite{nesterov2017random} but more restrictive than  \cite{duchi2015optimal}).   
 We  provide a rigorous quantitative analysis of the error between the various gradient estimates and the true gradient. The resulting error bounds presented in this paper can be used to establish convergence results for different variants of (stochastic) gradient methods. The deterministic bounds (finite differences and interpolation) can be used to establish convergence of a gradient descent scheme with fixed or adaptive step sizes. The resulting error bounds for the randomized methods can be used to establish convergence for simple stochastic gradient-type methods or adaptive methods such as the line search method studied in  \cite{berahas2019global}.
Our results show that in order to obtain gradient accuracy comparable to interpolation (or more generally methods that use orthogonal directions), smoothing methods  with  Gaussian or unit sphere  directions (scaled or not scaled) can require significantly more samples.  With both theoretical and empirical evidence, we argue that while  smoothing methods (Gaussian or unit sphere) can be applied with $N \ll n$, the resulting estimates generally have lower accuracy (and thus can result in slow convergence when employed within an optimization algorithm) than the estimates computed via linear interpolation.

\paragraph{Organization} The paper is organized as follows. In the remainder of this section, we introduce the assumptions we make for our analysis, and then present the main results of the paper. We define and derive theoretical results for the gradient approximation methods in Section~\ref{sec:grad_approx}. We present a numerical comparison of the gradient approximations and illustrate the performance of a line search DFO algorithm that employs these gradient approximations in Section~\ref{sec:num}. Finally, in Section \ref{sec:finrem}, we make some concluding remarks and discuss avenues for future research.

\subsection{Assumptions}

Throughout the paper we assume that the noise in the function evaluations $\epsilon(x)$ is bounded for all $x \in \mathbb{R}^n$, and that $\phi$ is Lipschitz smooth. 
\begin{assumption}	\label{assum:bounded_noise} 
\textbf{(Boundedness of Noise in the Function)} There is a constant $\epsilon_f \geq 0$ such that $| f(x) - \phi(x)| = |\epsilon(x)| \leq \epsilon_f$ for all $x \in \mathbb{R}^n$.
\end{assumption}
\begin{assumption}	\label{assum:lip_cont} 
\textbf{(Lipschitz continuity of the gradients of $\pmb{\phi}$)} 
The function $\phi$ is continuously differentiable, and the gradient of $\phi$ is $L$-Lipschitz continuous for all $x \in \mathbb{R}^n$. 
\end{assumption}

In some cases, to establish better approximations of the gradient, we will assume that  $\phi$ has  Lipschitz continuous Hessians.  
\begin{assumption}	\label{assum:lip_cont_hess} 
\textbf{(Lipschitz continuity of the Hessian of $\pmb{\phi}$)} 
The function $\phi$ is twice continuously differentiable, and the Hessian of $\phi$ is $M$-Lipschitz continuous for all $x \in \mathbb{R}^n$. 
\end{assumption}


\subsection{Summary of Results}
We begin by stating a condition that is often used in the analysis of first order methods with inexact gradient computations:
\begin{align}	\label{eq:theta_cond}
	\| g(x) - \nabla \phi(x) \| \leq \theta \|  \nabla \phi(x) \|,  
\end{align}
for some $\theta \in [0,1)$. This condition, referred to as the \emph{norm condition}, was introduced and studied in \cite{carter1991global,polyakintroduction}. In \cite{berahas2019global} the authors establish expected complexity bounds for a generic line search algorithm that uses gradient approximations in lieu of the true gradient, under the condition that the gradient estimate $g(x)$ satisfies \eqref{eq:theta_cond} for sufficiently small $\theta$ and with sufficiently high probability $1-\delta$. Note, this condition implies that $g(x)$ is a descent direction for the function $\phi$. Clearly, unless we know $\|  \nabla \phi(x) \|$, condition \eqref{eq:theta_cond} may be hard or impossible to verify or guarantee. There is significant amount of work that attempts to circumvent this difficulty; see e.g., \cite{byrd2012sample,cartis2018global,paquette2018stochastic}. In \cite{byrd2012sample} a practical approach to estimate $\|  \nabla \phi(x_k) \|$ is proposed  and  used to ensure some approximation of \eqref{eq:theta_cond} holds. In \cite{cartis2018global,paquette2018stochastic} the condition \eqref{eq:theta_cond} is replaced by 
\begin{align*}
	\| g(x) - \nabla \phi(x) \| \leq \kappa \alpha_k   \|  g(x) \|,  
\end{align*}
for some $\kappa > 0 $, and convergence rate analyses are derived for a line search method that has access to deterministic function values in \cite{cartis2018global} and stochastic function values (with additional assumptions) in \cite{paquette2018stochastic}. However, for the methods studied in this paper, condition \eqref{eq:theta_cond} turns out to be achievable. We establish conditions under which  \eqref{eq:theta_cond}
 holds either deterministically or with sufficiently high probability.

Given a point $x$, all methods compute $g(x)$  via either \eqref{eq:GSG_intro} or \eqref{eq:cGSG_intro}. The methods vary in their selection of the \emph{size of the sample set} $N$, the set  $\{u_i: i=1, \ldots, N\}$ and the corresponding set $\{\tilde u_i: i=1, \ldots, N\}$, and the \emph{sampling radius} $\sigma$.  Here, upfront, we present a simplified summary of the conditions on $N$, $\sigma$ and $\nabla \phi(x)$ for each method that we consider in this paper to guarantee condition \eqref{eq:theta_cond}; see Table \ref{tbl:bounds1}. For more detailed results  see Section \ref{sec:summary}, Table 
 \ref{tbl:bounds_full2}. Note that for the smoothing methods \eqref{eq:theta_cond} holds with probability $1-\delta$ and the number of samples depends on $\delta$. Moreover, the bounds on $N$ for the smoothing methods are a simplification of the more detailed bounds derived in the paper and apply when $n\geq4$, while for smaller $n$ some of the constants are larger. We should note that the constants in the bound on $N$ are smaller for larger $n$.

 \begin{table}[h!]

\caption{ Simplified conditions under assumption that $n\geq4$. Bounds on $N$, $\sigma$ and $\nabla \phi(x)$ that ensure $\| g(x) - \nabla \phi(x) \| \leq \theta\|\nabla \phi(x)\|$ ($\pmb{^*}$ denotes result is with probability $1-\delta$).}
\small
\label{tbl:bounds1}
\centering
\begin{tabular}{lccc}
\toprule
\begin{tabular}[l]{@{}l@{}}\textbf{Gradient} \\  \textbf{Approximation}\end{tabular} &
 \textbf{$\pmb{N}$} &
 \textbf{$\pmb{\sigma}$} &  \textbf{$\pmb{\| \nabla \phi(x) \|}$} \\  \midrule

\begin{tabular}[l]{@{}l@{}}\textbf{Forward Finite} \\  \textbf{Differences}\end{tabular} & $n$ & $ 2\sqrt{\frac{\epsilon_f}{L}}$ & $\frac{2\sqrt{nL\epsilon_f}}{\theta}$ \\ \hdashline
 \begin{tabular}[l]{@{}l@{}}\textbf{Central Finite} \\  \textbf{Differences}\end{tabular} & $n$ & $  \sqrt[3]{\frac{6\epsilon_f}{M}}$ & $\frac{2  \sqrt[3]{ n^{3/2}M \epsilon_f^2}}{ \theta}$  \\ \hdashline
 \begin{tabular}[l]{@{}l@{}}\textbf{Linear} \\  \textbf{Interpolation}\end{tabular} & $n$ & $2 \sqrt{\frac{\epsilon_f}{L}}$ & $\frac{2\| Q_\mathcal{X}^{-1}\|\sqrt{nL\epsilon_f}}{\theta}$ \\ \hdashline
 \begin{tabular}[l]{@{}l@{}}\textbf{Gaussian Smoothed} \\  \textbf{Gradients}$\pmb{^*}$\end{tabular} & $\frac{36n }{\delta  \theta^2 } + \frac{3n+24}{16\delta }  $ & $ \sqrt{\frac{\epsilon_f}{L}}$ & $\frac{6\sqrt{n^2L\epsilon_f}}{\theta}$ \\ \hdashline
 \begin{tabular}[l]{@{}l@{}}\textbf{Centered Gaussian} \\  \textbf{Smoothed Gradients}$\pmb{^*}$\end{tabular} & $\frac{36n  }{\delta  \theta^2 } +
	 \frac{n+15}{48\delta  }  $ & $ \sqrt[3]{\frac{\epsilon_f}{\sqrt{n}M}}$ & $\frac{12\sqrt[3]{n^{7/2}M\epsilon_f^2}}{\theta}$ \\ \hdashline
	  \begin{tabular}[l]{@{}l@{}}\textbf{Sphere Smoothed} \\  \textbf{Gradients}$\pmb{^*}$\end{tabular} & $ \left[\frac{24n }{\theta^2 } + \frac{8n}{3 \theta} + \frac{3n}{8} + \frac{\sqrt{n}}{3} +  \frac{13}{6}\right] \log \frac{n+1}{\delta}$ & $ \sqrt{\frac{n\epsilon_f}{L}}$ & $\frac{4\sqrt{n^2L\epsilon_f}}{\theta}$ \\ \hdashline 
 \begin{tabular}[l]{@{}l@{}}\textbf{Centered Sphere} \\  \textbf{Smoothed Gradients}$\pmb{^*}$\end{tabular} & $ \left[\frac{24n }{\theta^2 } + \frac{8n}{3 \theta} + \frac{n}{24} + \frac{\sqrt{n}}{9} +  \frac{17}{24}\right] \log \frac{n+1}{\delta}$ & $ \sqrt[3]{\frac{n\epsilon_f}{M}}$ & $\frac{4\sqrt[3]{n^{7/2}M\epsilon_f^2}}{\theta}$ \\

 \bottomrule
\end{tabular}
 \end{table}
The bounds $N$ for all methods in Table \ref{tbl:bounds1} are the upper bounds, in the sense that they give the value of $N$ that guarantees the desired gradient estimate accuracy (with high probability). Clearly for deterministic methods these bounds are also the lower bounds, that is, 
no gradient accuracy can be guaranteed (in general) with a smaller value of $N$. For the smoothing methods, deriving accurate lower bound on $N$ is nontrivial. We show, however, that this lower bound is linear in $n$ and via numerical simulation confirm that the constants in the bound are significantly larger than those for deterministic methods, such as finite differences. This suggests that deterministic methods may be more efficient, at least in the setting considered in this paper, when accurate gradient estimates are desired.   The bounds on the sampling radius are comparable for the smoothing and deterministic methods, as we will discuss in detail later in the paper. Finally, our numerical results support our theoretical observations. 



\section{Gradient Approximations and Sampling}	\label{sec:grad_approx}

In this section, we analyze several existing methods for constructing gradient approximations using only noisy function information. We establish conditions under which the gradient approximations constructed via these methods satisfy the bound \eqref{eq:theta_cond} for any given $\theta\in[0,1)$. 

The common feature amongst these methods is that they construct  approximations $g(x)$ of the gradient $\nabla \phi(x)$ using (possibly noisy) function values $f(y)$ for $y\in \mathcal{X}$, where  $\mathcal{X}$ is a \emph{sample set} centered around $x$.  These methods differ in the way they select $\mathcal{X}$ and the manner in which the function values $f(y)$, on all sample points $y\in \mathcal{X}$, are used to construct $g(x)$. The methods have different costs in terms of number of evaluations of $f$, as well as other associated computations. Our goal is to compare these costs when computing gradient estimates that satisfy  \eqref{eq:theta_cond} for some  $\theta\in [0,1)$. For each method, we derive bounds on the number of samples and the sampling radius which guarantee \eqref{eq:theta_cond}, the sufficient  condition for convergence of the line search method in \cite{berahas2019global}.

\subsection{Gradient Estimation via Standard Finite Differences}	\label{sec:fd}

The first method we analyze is the  standard finite difference method. The forward finite difference (FFD) approximation to the gradient of $\phi$ at $x \in \mathbb{R}^n$ is computed using the sample set $\mathcal{X} = \{ x + \sigma e_i \}_{i=1}^n\cup \{x\}$, where $\sigma>0$ is the finite difference interval and $e_i \in \R^n$ is the $i$-th column of the identity matrix,  as follows
\begin{align*}
	[g({x})]_i = \frac{f({x} + \sigma e_i) - f({x})}{\sigma}, \ \ \text{for} \ \ i=1,\dots,n.
\end{align*}
 Alternatively, gradient approximations can be computed using central finite differences (CFD) based on the sample set $\mathcal{X} = \{ x + \sigma e_i \}_{i=1}^n\cup  \{ x - \sigma e_i \}_{i=1}^n$
, as
\begin{align*}
	[g({x})]_i = \frac{f({x} + \sigma e_i) - f({x} - \sigma e_i)}{2\sigma}, \ \ \text{for} \ \ i=1,\dots,n.
\end{align*}
 FFD and CFD approximations require $n$ and $2n$ functions evaluations, respectively. CFD approximations tends to be 
 more accurate and stable, as we show below. 
 
We begin by stating two standard  gradient approximation bounds, i.e., the error between the finite difference approximation to the gradient and the gradient of $\phi$.
\begin{theorem}\label{thm:FFDbound} Under Assumptions \ref{assum:bounded_noise} and \ref{assum:lip_cont}, 
 let $g(x)$ denote the forward finite difference (FFD) approximation to the gradient $\nabla \phi(x)$. Then, for all $x \in \mathbb{R}^n$,
\begin{align*}
	\|g({x}) - \nabla \phi({x})\| \leq \frac{\sqrt{n} L \sigma}{2} + \frac{2\sqrt{n} \epsilon_f}{\sigma}.
\end{align*}
\end{theorem}

\begin{theorem}\label{thm:CFDbound} Under Assumptions \ref{assum:bounded_noise} and \ref{assum:lip_cont_hess}, let $g(x)$ denote the central finite difference (CFD) approximation to the gradient $\nabla \phi(x)$. Then, for all $x \in \mathbb{R}^n$,
\begin{align*}
	\|g({x}) - \nabla \phi({x})\| \leq \frac{\sqrt{n} M \sigma^2}{6} + \frac{\sqrt{n} \epsilon_f}{\sigma}.
\end{align*}
\end{theorem}

It is apparent from Theorems  \ref{thm:FFDbound} and \ref{thm:CFDbound} that the finite difference interval $\sigma>0$ should be chosen not to be too small or too large in order to control the bound on $\|g({x}) - \nabla \phi({x})\|$. The precise range of acceptable values of $\sigma$ depends on the Lipschitz constant $L$ of $\nabla \phi(x)$, and the level of noise $\epsilon_f$.  We derive expressions for $\sigma$ based on Theorems \ref{thm:FFDbound} and \ref{thm:CFDbound}, and then discuss the implications of not knowing $L$ and $\epsilon_f$ precisely. 

First we consider  the FFD case and thus Theorem \ref{thm:FFDbound}. \ks{In order for the estimate of  $\nabla \phi(x)$ computed by FFD to satisfy 
 \eqref{eq:theta_cond} for some given  $x\in\mathbb{R}^n$ we chose  $\sigma$ such that the following holds }
\begin{align}\label{eq:FFDbndsig1}
 	\frac{\sqrt{n} L \sigma}{2} + \frac{2\sqrt{n} \epsilon_f}{\sigma}\leq  \theta\|\nabla \phi(x)\|, 
 \end{align}
  which can be written as a quadratic inequality,
\begin{align*}
	\frac{\sqrt{n} L}{2} \sigma^2 - \theta\|\nabla \phi(x)\| \sigma + 2\sqrt{n} \epsilon_f &\le 0.
\end{align*}
The case when $L=0$, and known, is not interesting in our context, because then the function is linear and gradient approximation should be 
performed outside of any optimization scheme. 
Hence, we assume that (the upper bound of) the Lipschitz constant of $\nabla \phi(x)$, $L$, is strictly positive. Then,  the interval of $\sigma$ values that satisfy the quadratic inequality is
\begin{align} \label{eq:sigma_bound_quad}
	\frac{\theta \|\nabla \phi(x)\| - \sqrt{\theta^2 \|\nabla \phi(x)\|^2 - 4nL\epsilon_f}}{\sqrt{n} L} \le \sigma \le \frac{\theta \|\nabla \phi(x)\| + \sqrt{\theta^2 \|\nabla \phi(x)\|^2 - 4nL\epsilon_f}}{\sqrt{n} L}.  
\end{align}
This interval is nonempty when $\theta^2 \|\nabla \phi(x)\|^2 \ge 4nL\epsilon_f$, which constitutes to a condition on $ \|\nabla \phi(x)\|$, with respect to $L$ and $\epsilon_f$, for which FFD,  with the appropriate choice of $\sigma$,  can satisfy \eqref{eq:theta_cond}. When $\theta^2 \|\nabla \phi(x)\|^2 \ge 4nL\epsilon_f$, any choice of $\sigma$ satisfying  \eqref{eq:sigma_bound_quad} works, however, since we do not know $\|\nabla\phi(x)\|$, we set $\sigma$ to the known value, 
\begin{equation}\label{eq:sigma_set}
\sigma= 2\sqrt{ \frac{\epsilon_f}{L}},
\end{equation}
which minimizes the left hand side of \eqref{eq:FFDbndsig1} and thus satisfies \eqref{eq:sigma_bound_quad}.

When $ \|\nabla \phi(x)\|$ falls below $\frac{2\sqrt{nL\epsilon_f}}{\theta}$, finite difference approximations to the gradient can no longer ensure sufficiently accurate approximations, and any optimization process reliant on these approximations may fail to progress. Thus, the  implication of not knowing $\epsilon_f$ and $L$ precisely, but replacing them with overestimates when defining $\sigma$, results in earlier stalling of an optimization algorithm based on FFD (and all other gradient estimates schemes that we will discuss in this manuscript). This observation agrees with related results in \cite{berahas2019global}, where it is shown that a line search algorithm for noisy objective functions,  based on gradient approximations that satisfy \eqref{eq:theta_cond}, enjoys fast convergence rates until it reaches  a neighborhood of optimality dictated by the estimate $\epsilon_f$.

Applying the same logic as above to Theorem \ref{thm:CFDbound}, in order to ensure that \eqref{eq:theta_cond} holds, we require
\begin{align}\label{eq:CFDbndsig1}
	\frac{\sqrt{n} M \sigma^2}{6} + \frac{\sqrt{n} \epsilon_f}{\sigma}\leq  \theta\|\nabla \phi(x)\|,
\end{align}
which can be written as a cubic inequality,
\begin{align*}
 \frac{\sqrt{n}M}{6} \sigma^3 - \theta \|\nabla\phi(x)\| \sigma + \sqrt{n} \epsilon_f \le 0. 
\end{align*}
The cubic left-hand side has three roots. The first root is a negative number, while the second and third roots are positive real numbers if 
\begin{align*}
	 \| \nabla \phi(x)\| \ge \frac{\sqrt{n} \sqrt[3]{9M\epsilon_f^2}}{2 \theta}, 
\end{align*}
which constitutes to a condition on $\|\nabla \phi(x)\|$  for which CFD can deliver a gradient estimate satisfying \eqref{eq:theta_cond} if $\sigma$ is chosen as a value inside the interval between the second and third roots.  
Choosing $\sigma$ to satify
\begin{align*}
	 \sigma = \sqrt[3]{\frac{3\epsilon_f}{M}}
\end{align*}
minimizes the left-hand side of \eqref{eq:CFDbndsig1} in the interval between the second and the third root. 
\subsection{Gradient Estimation via Linear Interpolation}	\label{sec:lin_mod}

We now consider a more general method for approximating gradients using polynomial interpolation that has become a popular choice for model based trust region methods in the DFO setting \cite{ConnScheToin97,ConnScheVice08c,ConnScheToin98,Powe74,Powe06,wild2008orbit,maggiar2018derivative}. These methods construct surrogate models of the objective function  using interpolation (or regression). While typically, in the DFO setting, interpolation is used to construct quadratic models of the objective function around $x \in \mathbb{R}^n$ of the form
\begin{align}		\label{eq:quad_mod}
	m(y)=f(x)+g(x)^\intercal (y-x)+\frac{1}{2}(y-x)^\intercal H(x)(y-x),
\end{align}
where $f \in \mathbb{R}$ and $g \in \mathbb{R}^n$, or $H \in \mathbb{R}^{n\times n}$, in this paper we focus on the simplest case of linear models, 
\begin{align}		\label{eq:lin_mod}
	m(y)=f(x)+g(x)^\intercal (y-x),
\end{align}
as the focus of this paper is on line search methods, whereas the use of \eqref{eq:quad_mod} requires a trust region approach due to the general nonconvexity of $m(y)$ \cite{ConnScheVice08c}. 


Let us consider the following sample set  $\mathcal{X}=\{x+\sigma u_1, x+\sigma u_2, \ldots, x+\sigma u_n\}$ for some $\sigma >0$. In other words, we have $n$ directions denoted by $u_i \in \mathbb{R}^n$ and we sample $f$ along those directions, around $x$, using a sampling radius of size $\sigma$. We assume $f(x)$ is known (function value at $x$).  Let $F_\mathcal{X} \in \mathbb{R}^n$ be a vector whose entries are
$f(x+\sigma u_i) - f(x)$, for $i=1\dots n$, and let $Q_\mathcal{X} \in \mathbb{R}^{n \times n}$ define a matrix whose rows are given by $u_i$ for $i=1\dots n$. The model in \eqref{eq:lin_mod} is constructed to satisfy  the interpolation conditions, 
\begin{align*}
	f(x+\sigma u_i)=m(x+\sigma u_i), \quad \forall i=1, \ldots, n,\\
\end{align*}
which can be written as
\begin{align} \label{eq:main_system}
	\sigma Q_\mathcal{X} g=F_\mathcal{X}. 
\end{align}

If the matrix $Q_\mathcal{X}$ is nonsingular, then $m(y) = f(x) + g(x)^\intercal (y-x)$, with $g(x) = \frac{1}{\sigma}Q_\mathcal{X}^{-1}F_\mathcal{X}$, is a linear interpolation model of $f(y)$ on the sample set $\mathcal{X}$. When $Q_\mathcal{X}$ is the identity matrix, then we recover standard forward finite difference gradient estimation. In the specific case when $Q_\mathcal{X}$ is
orthonormal, then $Q_\mathcal{X}^{-1} = Q_\mathcal{X}^\intercal$, thus $g(x)$ is written as
\begin{align*}
	g(x) = \sum_{i=1}^n \frac{f(x+\sigma u_i) - f(x)}{\sigma}u_i.
\end{align*}


Next we derive a bound on $\|g(x)-\nabla \phi(x)\|$. This result is an extension of the results presented in \cite{ConnScheVice08c,conn2008geometry} that accounts for the noise in the function evaluations. 

\begin{theorem} \label{thm:bnd_linmod}Suppose that Assumptions \ref{assum:bounded_noise} and \ref{assum:lip_cont} hold. Let $\mathcal{X}=\{x+\sigma u_1,\dots,x+\sigma u_n \}$ be a set of interpolation points such that $\max_{1\leq i \leq n}\|u_i\|\leq 1$  and  
$Q_\mathcal{X}$ be nonsingular. Then, for all $x \in \mathbb{R}^n$,
\begin{align*}
	\|g({x}) - \nabla \phi({x})\| \leq \frac{\| Q_\mathcal{X}^{-1}\|_2\sqrt{n}L \sigma }{2} + \frac{2 \| Q_\mathcal{X}^{-1}\|_2\sqrt{n} \epsilon_f}{\sigma}.
\end{align*}
\end{theorem}

\begin{proof}
From the interpolation conditions and the mean value theorem, $\forall i=1, \ldots, n$ we have 
\begin{align*}
\sigma g(x)^\intercal u_i &=f(x+\sigma u_i)-f(x)=\phi(x+\sigma u_i)-\phi(x)+\epsilon(x+\sigma u_i)-\epsilon(x)\\
&=\int_{0}^1 \sigma u_i^\intercal\nabla \phi(x+t\sigma u_i)dt+\epsilon(x+\sigma u_i)-\epsilon(x).
\end{align*}
From the $L$-smoothness of $\phi(\cdot)$ and the bound on $\epsilon(\cdot)$ we have 
\begin{align*}
\sigma |(g(x)-\nabla \phi(x))^\intercal u_i| \leq  \frac {L\sigma^2 \|u_i\|^2}{2}+2\epsilon_f, \quad \forall i=1, \ldots, n
\end{align*}
which in turn implies 
\begin{align*}
 \| Q_{\cal X}(g(x)-\nabla \phi(x))\| \leq  \frac {\sqrt{n}L\sigma  }{2}+\frac{2 \sqrt{n} \epsilon_f}{\sigma},
\end{align*}
and the theorem statement follows. 
\end{proof}

This result has the implication that large $\| Q_\mathcal{X}^{-1}\|$ can cause large deviation of $g(x)$ from $\nabla \phi(x)$. Thus, it is  desirable to select $\mathcal{X}$ in such a way that the condition number of $Q_\mathcal{X}^{-1}$ is small, which is clearly optimized when $Q_\mathcal{X}$ is orthonormal. Thus, we trivially recover the theorem for FFD, and moreover, extend this result to any orthonormal set of directions $\{u_1, u_2 \ldots, u_n\}$, such as those used in \cite{choro2}.  Aside from the condition number, the  important difference between general interpolation sets and orthonormal ones is in the computational cost of evaluating $g(x)$. In particular, $g(x)$ is obtained by solving a system of linear equations given by \eqref{eq:main_system},
which in general requires ${\cal O}(n^3)$ computations, but that reduces to ${\cal O}(n^2)$ in the case of general orthornormal matrices $Q_\mathcal{X}$, and further reduces to ${\cal O}(n)$ for $Q_\mathcal{X}=I$, as in the case of FFD. 
In  \cite{choro2}, it is proposed to use  scaled randomized Haddamard matrices as $Q_\mathcal{X}$. This is only possible if the problem dimension is a power of $2$, but it reduces linear algebra cost  of matrix-vector products
from  ${\cal O}(n^2)$ to  ${\cal O}(n\log n)$. 

On the other hand, using general sample sets allows for greater flexibility (within an optimization algorithm), in particular enabling the re-use of sample points from prior iterations. When using FFD to compute $g(x)$, $n$ function evaluations are always required, while when using interpolation it is possible to update the interpolation set by replacing only one (or a few) sample point(s) in the set $\mathcal{X}$. It is important to note that while $\mathcal{X}$ can be fairly general, the condition number of  the matrix $Q_\mathcal{X}$ has to remain bounded for Theorem \ref{thm:bnd_linmod} to be useful. The sets with bounded condition number of  $Q_\mathcal{X}$ are called {\em well-poised}; see  \cite{ConnScheVice08c} for  details about the construction and maintenance of interpolation sets in model based trust region DFO methods. 

The bounds of Theorem \ref{thm:bnd_linmod} are similar to those of Theorem \ref{thm:FFDbound}, hence, if the sampling radius $\sigma$ and the  the gradient norm satisfy
 \begin{align*}
	\sigma =2\sqrt{ \frac{\epsilon_f}{L}} \quad \text{and} \quad \|\nabla \phi(x)\|\geq\frac{2 \| Q_\mathcal{X}^{-1}\|\sqrt{nL\epsilon_f}}{\theta},
 \end{align*}

respectively, then \eqref{eq:theta_cond} holds.  
 
 It is possible to derive an analogue of Theorem \ref{thm:CFDbound} by including $n$ additional sample points  $\{x-\sigma u_1,\dots,x-\sigma u_n \}$ in the gradient estimation procedure. Namely, two sample sets are used,  $\mathcal{X}^+=\{x+\sigma u_1, x+\sigma u_2, \ldots x+\sigma u_n\}$ and $\mathcal{X}^-=\{x-\sigma u_1, x-\sigma u_2, \ldots x-\sigma u_n\}$, with corresponding matrices $Q_{\mathcal{X}^+}$
and $Q_{\mathcal{X}^-}$. The linear model $m(y)=f(x)+g^\intercal(y-x)$ is then computed as an average of the two interpolation models, that is 
\begin{align*}
	 g=\frac{g_0^+ + g_0^-}{2}= \frac{1}{2\sigma}[Q_{\mathcal{X}^+}^{-1}F_{\mathcal{X}^+}+Q_{\mathcal{X}^-}^{-1}F_{\mathcal{X}^-}].
\end{align*}
The gradient estimates are computed in this way in \cite{choro}, for the case of orthonormal sets and symmetric finite difference computations. Similarly to the CFD, this results in better accuracy bounds in terms of $\sigma$; however, this requires additional $n$ function evaluations at each iteration, which contradicts the original idea of using interpolation as a means for reducing the per-iteration function evaluation cost.

\subsection{Gradient Estimation via Gaussian Smoothing} \label{sec:gauss_smooth}

Gaussian smoothing has recently become a popular tool for building gradient approximations using only function values. This approach has been exploited in several recent papers; see e.g., \cite{nesterov2017random, maggiar2018derivative,ES, NES2,bayandina2017gradient}.
 
Gaussian smoothing of a given function $f$ is obtained as  follows:
\begin{align}	\label{eq: ES obj}
F(x) &= \mathbb{E}_{y \sim \mathcal{N}(x, \sigma^2I)} [f(y)] = \int_{\R^n} f(y) \pi (y|x,\sigma^2I) dy\nonumber \\
	& = \mathbb{E}_{u \sim \mathcal{N}({0}, I)} [f(x+\sigma u)] = \int_{\R^n} f(x+\sigma u) \pi (u|0,I) du,  
\end{align}
where $\mathcal{N}(x, \sigma^2I)$ denotes the multivariate normal distribution with mean $x$ and covariance matrix $\sigma^2I$, $\mathcal{N}({0}, I)$ denotes the standard multivariate normal distribution, and the functions $\pi (y|x,\sigma^2I)$ and $\pi(u|0,I)$ denote the probability density functions (pdf) of $\mathcal{N}(x,\sigma^2 I)$ evaluated at $y$ and $\mathcal{N}(0,I)$ evaluated at $u$, respectively.  Using properties of derivatives of expected value functions \cite{AsmussenGlynnBook}, the gradient of $F$ can be expressed as
\begin{align}		\label{eq: ES g}
    &\nabla  F(x) = \frac{1}{\sigma} {\mathbb{E}}_{u \sim \mathcal{N}({0}, I)} [f(x+\sigma u) u]. 
\end{align}
Assume $f$ is an approximation of $\phi$ with the approximation error bounded by $\epsilon_f$ uniformly, i.e., Assumption \ref{assum:bounded_noise} holds. If Assumption \ref{assum:bounded_noise}  holds, then the following bounds hold for the error between $\nabla F(x)$ and $\nabla \phi(x)$. 
If $\phi$ has $L$-Lipschitz continuous gradients, that is if Assumption \ref{assum:lip_cont} holds, then
\begin{align}	\label{eq:GSG_bound1}
	\| \nabla F(x) - \nabla \phi(x)\| \le \sqrt{n} L \sigma +  \frac{\sqrt{n} \epsilon_f}{\sigma};
\end{align}
see Appendix \ref{app:GSG_bound1} for the proof\footnote{The bound \eqref{eq:GSG_bound1} was presented in \cite{maggiar2018derivative} without proof; we would like to thank the first author of \cite{maggiar2018derivative} for providing us with guidance of this proof.}. If the function $\phi$ has $M$-Lipschitz continuous Hessians, that is if Assumption \ref{assum:lip_cont_hess} holds, then
\begin{align}	\label{eq:cGSG_bound1}
	\| \nabla F(x) - \nabla \phi(x)\| \le nM \sigma^2 + \frac{\sqrt{n} \epsilon_f}{\sigma};
\end{align}
see Appendix \ref{app:cGSG_bound1} for proof. 

In order to approximate $\nabla \phi(x)$ one can approximate $  \nabla  F(x) $, with sufficient accuracy, by sample average approximation applied to \eqref{eq: ES g}, i.e.,
\begin{align}		\label{eq:ncGSG}
	g(x)=\frac{1}{N\sigma} \sum_{i=1}^N f(x+\sigma u_i) u_i, 
\end{align}
where $u_i \sim \mathcal{N}({0}, I)$ for $i = 1,2,\dots, N$.  It can be easily verified that $g(x)$ computed via \eqref{eq:ncGSG} has large variance (the variance explodes as $\sigma$ goes to $0$). The following simple modification, 
\begin{align}		\label{eq:GSG}
	g(x) = \frac{1}{N} \sum_{i=1}^N \frac{f(x+\sigma u_i) - f(x)}{\sigma} u_i, 
\end{align}
 eliminates this problem  and is indeed  used in practice instead of \eqref{eq:ncGSG}; see \cite{ES,choro, choro2}. Note that  the expectation of \eqref{eq:GSG} is also  $\nabla F(x)$, since ${\mathbb{E}}_{u_i \sim \mathcal{N}(\mathbf{0}, I)} [f(x) u]$ is an all-zero vector for all $i$. 
In what follows we will refer to $g(x)$ computed via \eqref{eq:GSG}  as the Gaussian smoothed gradient (GSG). As pointed out in \cite{nesterov2017random}, $\frac{f(x+\sigma u_i) - f(x)}{\sigma} u_i$ can be 
interpreted as a forward finite difference version of the directional derivative of $f$ at $x$ along $u_i$. Moreover, one can also consider the central difference variant of \eqref{eq:GSG}--central Gaussian smoothed gradient (cGSG)--which is computed as follows,
\begin{align}		\label{eq:cGSG}
	g(x) = \frac{1}{2N} \sum_{i=1}^N \frac{f(x+\sigma u_i) - f(x-\sigma u_i)}{\sigma} u_i.
\end{align}

The properties of \eqref{eq: ES obj} and \eqref{eq:GSG}, with $N=1$, were analyzed in \cite{nesterov2017random}. However, this analysis does not explore the effect of $N>1$ on the variance of  $g(x)$. On the other hand, in \cite{ES} the authors propose an algorithm that uses GSG estimates, \eqref{eq:GSG} and \eqref{eq:cGSG}, with large samples sizes $N$
in a fixed step size gradient descent algorithm, but without any analysis or discussion of the choices of $N$, $\sigma$ or $\alpha$ (where $\alpha$ is the step size). Thus, the purpose of this section is to derive bounds on the approximation error $\| g(x) - \nabla \phi(x) \|$ for GSG and cGSG, and to derive conditions on $\sigma$ and $N$ under which condition \eqref{eq:theta_cond} holds (and as a result the convergence results for a line search DFO algorithm \cite{berahas2019global} based on these approximations also hold). 

We first note that  there are two sources of error: $(i)$ approximation of the true function $\phi$ by the Gaussian smoothed function $F$ of the noisy function $f$, and $(ii)$ approximation of $\nabla F(x)$ via sample average approximations. Hence, we have that
\begin{align} \label{eq:2terms} 
	\| g(x) - \nabla \phi(x) \| &= \| (\nabla F(x) - \nabla \phi(x)) + (g(x) - \nabla F(x))\| \nonumber \\
		& \leq \| \nabla F(x) - \nabla \phi(x)\| + \| g(x) - \nabla F(x)\|.
\end{align} 
The bound on the first term is given by \eqref{eq:GSG_bound1} or \eqref{eq:cGSG_bound1}. 
What remains is to bound the second term $\| g(x) - \nabla F(x)\|$, the error due to the sample average approximation. 

Since \eqref{eq:GSG} (and \eqref{eq:cGSG}) is a (mini-)batch stochastic gradient estimate of $\nabla F(x)$, the probabilistic bound on $\| g(x) - \nabla F(x)\| $ is derived by bounding the expectation, which is equivalent to bounding the variance of the  (mini-)batch stochastic gradient. Existing bounds in the literature, see e.g., \cite{tripuraneni2017stochastic}, are  derived under the assumption that $\|g(x) - \nabla \phi(x)\|$ is uniformly bounded above almost surely, which does not hold for GSG because when $u$ follows a Gaussian distribution, $\frac{f(x+\sigma u) - f(x)}{\sigma} u$ can be arbitrarily large with positive probability. Here, we bound $\| g(x) - \nabla F(x)\| $ only under Assumptions \ref{assum:lip_cont} or \ref{assum:lip_cont_hess}. It is shown in \cite{nesterov2017random} that Assumption \ref{assum:lip_cont} implies that $\nabla F(x)$ is $L$-Lipschitz continuous; by applying similar logic it can be shown that Assumption \ref{assum:lip_cont_hess} implies that $\nabla^2 F(x)$ is $M$-Lipschitz continuous.

The variance for \eqref{eq:GSG} can be expressed as 
\begin{align} \label{eq:var GSG}  
	\Var{g(x)} = \frac{1}{N} \mathbb{E}_{u \sim \mathcal{N}(0,I)} \left[\left(\frac{f(x+\sigma u) - f(x)}{\sigma} \right)^2 u u^\intercal \right] - \frac{1}{N} \nabla F(x) \nabla F(x)^\intercal , 
\end{align}
and the variance of \eqref{eq:cGSG} can be expressed as
\begin{align} \label{eq:var cGSG}
	\Var{g(x)} = \frac{1}{N} \mathbb{E}_{u \sim \mathcal{N}(0,I)} \left[\left( \frac{f(x+\sigma u) - f(x-\sigma u)}{2\sigma} \right)^2 u u^\intercal \right] - \frac{1}{N} \nabla F(x) \nabla F(x)^\intercal .
\end{align}

The following properties of a normally distributed multivariate random variable $u \in \R^n$ will be used in our analysis  and are derived in Appendix \ref{app:mvn}.  Let  $a \in \R^n$ be any constant vector, then
\begin{equation} \label{eq:mvn}  \everymath{\displaystyle} \begin{aligned} 
	&\mathbb{E}_{u \sim \mathcal{N}({0},I)} \left[(a^\intercal u)^2 u u^\intercal \right]= a^\intercal a I + 2 a a^\intercal \\
	&\mathbb{E}_{u \sim \mathcal{N}({0},I)} \left[a^\intercal u \cdot \|u\|^k \cdot u u^\intercal \right]= 0_{n\times n}  \text{ for } k = 0,1,2,... \\
	&\mathbb{E}_{u \sim \mathcal{N}({0},I)} \left[\|u\|^k u u^\intercal \right] 
	\begin{cases}
		= (n+2)(n+4)\cdots(n+k) I &\text{for } k = 0, 2, 4, \dots \\
		\preceq (n+1)(n+3)\cdots(n+k) \cdot n^{-0.5} I &\text{for } k = 1, 3, 5, \dots
	\end{cases}
\end{aligned} \end{equation}
It is interesting to note that only the last property is specific to the normal distribution, while the first two expressions hold for any random vector $u$, for which $u_i$ are symmetric iid random variables with unit variance. Thus, techniques presented in this paper can be extended to other distributions, such as the one used in \cite{spall2000adaptive}.

We now derive bounds for the variances of GSG and cGSG. 

\begin{lemma} \label{lemma:bnd_var}
Under Assumption  \ref{assum:lip_cont}, if $g(x)$ is calculated by \eqref{eq:GSG}, then, for all $x \in \mathbb{R}^n$, $\Var{g(x)} \preceq \kappa(x) I$ where 
\begin{align*}	
	\kappa(x) = \frac{3}{N} \left( 3\|\nabla\phi(x)\|^2 + \frac{L^2 \sigma^2}{4} (n+2)(n+4) + \frac{4\epsilon_f^2}{\sigma^2} \right).
\end{align*} 

Alternatively, under Assumption  \ref{assum:lip_cont_hess}, if $g(x)$ is calculated by  \eqref{eq:cGSG}, then, for all $x \in \mathbb{R}^n$, $\Var{g(x)} \preceq \kappa(x) I$ where
\begin{align*}	
	\kappa(x) = \frac{3}{N} \left( 3\|\nabla\phi(x)\|^2 + \frac{M^2 \sigma^4}{36} (n+2)(n+4)(n+6) + \frac{\epsilon_f^2}{\sigma^2} \right).
\end{align*} 
\end{lemma}

\begin{proof} 
Since $\nabla F(x) \nabla F(x)^\intercal \succeq 0$, we derive from \eqref{eq:var GSG}
\begin{equation*} \begin{aligned} 
\Var{g(x)} \preceq& \frac{1}{N} \mathbb{E}_{u \sim \mathcal{N}(0,I)} \left[\left(\frac{f(x+\sigma u) - f(x)}{\sigma} \right)^2 u u^\intercal \right] \\
=&\frac{1}{N} \mathbb{E}_{u \sim \mathcal{N}(0,I)} \left[\left(\frac{f(x+\sigma u) - f(x)}{\sigma} u \right) \left(\frac{f(x+\sigma u) - f(x)}{\sigma} u \right)^\intercal \right]. 
\end{aligned} \end{equation*} 
The term in the paretheses can be written as 
\begin{align*}
\omit\rlap{$ \frac{f(x+\sigma u) - f(x)}{\sigma} u$} \\
&= \frac{\phi(x+\sigma u) + \epsilon(x+\sigma u) - \phi(x) - \epsilon(x)}{\sigma} u \\
&= \frac{\phi(x+\sigma u) - \phi(x) - \nabla\phi(x)^\intercal \sigma u}{\sigma} u + \frac{\epsilon(x+\sigma u) - \epsilon(x)}{\sigma} u + \nabla\phi(x)^\intercal u u.
\end{align*}
Considering for any three vectors $\{v_1,v_2,v_3\} \subset \R^n$, it must be $(v_1+v_2+v_3)(v_1+v_2+v_3)^\intercal \preceq 3v_1 v_1^\intercal + 3v_2 v_2^\intercal  + 3v_3 v_3^\intercal$, we have 
\begin{align*}
	\omit\rlap{$\Var{g(x)}$} \\
	\preceq& \frac{3}{N} \mathbb{E}_{u \sim \mathcal{N}({0},I)} \left[\left(\frac{\phi(x+\sigma u) - \phi(x) - \nabla\phi(x)^\intercal \sigma u}{\sigma} \right)^2 u u^\intercal + \left[\left(\frac{\epsilon(x+\sigma u) - \epsilon(x)}{\sigma} \right)^2 u u^\intercal \right] + (\nabla\phi(x)^\intercal u)^2 uu^\intercal \right] \\
	\preceq& \frac{3}{N} \mathbb{E}_{u \sim \mathcal{N}({0},I)} \left[\left(\frac{L\sigma}{2} u^\intercal u\right)^2 u u^\intercal + \left(\frac{2\epsilon_f}{\sigma} \right)^2 u u^\intercal + (\nabla\phi(x)^\intercal u)^2 uu^\intercal \right] \\
	\stackrel{\mathclap{\mathrm{\eqref{eq:mvn}}}}{=}{}& \ \frac{3}{N} \left( \frac{L^2 \sigma^2}{4} (n+2)(n+4) I + \frac{4\epsilon_f^2}{\sigma^2} I + \|\nabla\phi(x)\|^2 I + 2\nabla\phi(x) \nabla\phi(x)^\intercal \right) \\
	\preceq& \frac{3}{N} \left( \frac{L^2 \sigma^2}{4} (n+2)(n+4) + \frac{4\epsilon_f^2}{\sigma^2} + 3\|\nabla\phi(x)\|^2 \right) I,  
\end{align*}
where the second inequality comes from the Lipschitz continuity of the gradients (Assumption \ref{assum:lip_cont}) and the bound on the noise, and the last inequality comes from the fact that $vv^\intercal \preceq \|v\|^2 I$ for any $v \in \R^n$. 

For cGSG, we follow the same logic as above. By \eqref{eq:var cGSG} we get 
\begin{equation*} \begin{aligned} 
\Var{g(x)} \preceq& \frac{1}{N} \mathbb{E}_{u \sim \mathcal{N}(0,I)} \left[\left(\frac{f(x+\sigma u) - f(x-\sigma u)}{2\sigma} \right)^2 u u^\intercal \right] \\
=&\frac{1}{N} \mathbb{E}_{u \sim \mathcal{N}(0,I)} \left[ \left(\frac{f(x+\sigma u) - f(x-\sigma u)}{2\sigma} u \right) \left(\frac{f(x+\sigma u) - f(x-\sigma u)}{2\sigma} u \right)^\intercal \right]. 
\end{aligned} \end{equation*} 
The term in the parentheses can be written as 
\begin{align*}
\omit\rlap{$\frac{f(x+\sigma u) - f(x-\sigma u)}{2\sigma} u$} \\
&= \frac{\phi(x+\sigma u) + \epsilon(x+\sigma u) - \phi(x-\sigma u) - \epsilon(x-\sigma u)}{2\sigma} u \\
&= \frac{\phi(x+\sigma u) - \phi(x-\sigma u) - 2\sigma \nabla\phi(x)^\intercal u}{2\sigma} u + \frac{\epsilon(x+\sigma u) - \epsilon(x)}{2\sigma} u + \nabla\phi(x)^\intercal u u \\
&= \frac{\left(\phi(x+\sigma u) - \phi(x) - \sigma \nabla\phi(x)^\intercal u - \frac{\sigma^2}{2} u^\intercal \nabla^2 \phi(x) u\right) - \left(\phi(x-\sigma u) - \phi(x) + \sigma \nabla\phi(x)^\intercal u - \frac{\sigma^2}{2} u^\intercal \nabla^2 \phi(x) u\right)}{2\sigma} u \\
&\quad + \frac{\epsilon(x+\sigma u) - \epsilon(x)}{2\sigma} u + \nabla\phi(x)^\intercal u u.
\end{align*}
Then,  for cGSG we have
\begin{align*}
	\omit\rlap{$\Var{g(x)}$} \\
	\preceq& \frac{3}{N} \mathbb{E}_{u \sim \mathcal{N}({0},I)} \left[\left(\frac{\phi(x+\sigma u) - \phi(x-\sigma u) - 2\sigma \nabla\phi(x)^\intercal u}{2\sigma} \right)^2 u u^\intercal + \left(\frac{\epsilon(x+\sigma u) - \epsilon(x)}{2\sigma} \right)^2 u u^\intercal + (\nabla\phi(x)^\intercal u)^2 uu^\intercal \right] \\
	\preceq& \frac{3}{N} \mathbb{E}_{u \sim \mathcal{N}({0},I)} \left[\left(\frac{M\sigma^2}{6} \|u\|^3 \right)^2 u u^\intercal + \left(\frac{2\epsilon_f}{2\sigma} \right)^2 u u^\intercal + (\nabla\phi(x)^\intercal u)^2 uu^\intercal \right] \\
	\stackrel{\mathclap{\mathrm{\eqref{eq:mvn}}}}{=}{}& \ \frac{3}{N} \left( \frac{M^2 \sigma^4}{36} (n+2)(n+4)(n+6) I + \frac{\epsilon_f^2}{\sigma^2} I + \|\nabla\phi(x)\|^2 I + 2\nabla\phi(x) \nabla\phi(x)^\intercal \right) \\
	\preceq& \frac{3}{N} \left( \frac{M^2 \sigma^4}{36} (n+2)(n+4)(n+6) + \frac{\epsilon_f^2}{\sigma^2} + 3\|\nabla\phi(x)\|^2 \right) I, 
\end{align*}
where the second inequality comes from the Lipschitz continuity of the Hessians (Assumption \ref{assum:lip_cont_hess}) and the bound on noise, and the last inequality comes from the fact that $vv^\intercal \preceq \|v\|^2 I$ for any $v \in \R^n$.
\end{proof}

Using the results of Lemma \ref{lemma:bnd_var}, we can now bound the quantity $\| g(x) - \nabla F(x) \|$ in \eqref{eq:2terms}, in probability, using Chebyshev's inequality. 

\begin{lemma} \label{lem:prob_bnd_smoothed} 
Let $F$ be a Gaussian smoothed approximation of $f$ \eqref{eq: ES obj}. Under Assumption  \ref{assum:lip_cont}, if $g(x)$ is  calculated via \eqref{eq:GSG} with sample size
\begin{align*}
	N \ge \frac{3n}{\delta r^2} \left( 3\|\nabla\phi(x)\|^2 + \frac{L^2 \sigma^2}{4} (n+2)(n+4) + \frac{4\epsilon_f^2}{\sigma^2} \right),
\end{align*}
then, for all $x \in \mathbb{R}^n$, $\|g(x) - \nabla F(x)\| \le r$ holds with probability at least $1 - \delta$, for any $r>0$ and $0< \delta<1$. 

Alternatively, under Assumption  \ref{assum:lip_cont_hess}, if $g(x)$ is calculated via \eqref{eq:cGSG} with sample size $2N$ where
\begin{align*}
	N \ge \frac{3n}{\delta r^2} \left( 3\|\nabla\phi(x)\|^2 + \frac{M^2 \sigma^4}{36} (n+2)(n+4)(n+6) + \frac{\epsilon_f^2}{\sigma^2} \right),
\end{align*}
then, for all $x \in \mathbb{R}^n$, $\|g(x) - \nabla F(x)\| \le r$ holds with probability at least $1 - \delta$, for any $r>0$ and $0< \delta <1$. 
\end{lemma}

\begin{proof}
By Chebyshev's inequality, for any $r > 0$, we have
\begin{align*}
\mathbb{P}\left\{\sqrt{(g(x) - \nabla F(x))^\intercal \Var{g(x)}^{-1} (g(x) - \nabla F(x))} > r \right\} 
&\le \frac{n}{r^2}. 
\end{align*}
Since by Lemma \ref{lemma:bnd_var} $\Var{g(x)} \preceq \kappa(x) I$, with the appropriate $\kappa(x)$ as shown in the statement of the Lemma, we have $\Var{g(x)}^{-1} \succeq \kappa (x)^{-1} I$ and 
\begin{align*}
 \sqrt{(g(x) - \nabla F(x))^\intercal \Var{g(x)}^{-1} (g(x) - \nabla F(x))} \ge \kappa (x)^{-\frac{1}{2}} \|g(x) - \nabla F(x)\|.
\end{align*}
Therefore, we have, 
\begin{align*}
\mathbb{P}\left\{\kappa (x)^{-\frac{1}{2}} \|g(x) - \nabla F(x)\| > r \right\} \le \frac{n}{r^2}  \ \ \Longrightarrow \ \
\mathbb{P}\left\{ \|g(x) - \nabla F(x)\| > r \right\} \le \frac{\kappa (x) n}{r^2}. 
\end{align*}
To ensure  $\mathbb{P}\left\{ \|g(x) - \nabla F(x)\| \le r \right\} \ge 1 - \delta$, we choose $\kappa$ such that $ \frac{\kappa (x) n}{r^2}\leq \delta$, by choosing large enough $N$. The exact bounds on $N$ (and thus the result of Lemma \ref{lem:prob_bnd_smoothed}) follow immediately from the two respective expressions for $\kappa(x)$ in  Lemma \ref{lemma:bnd_var}.
\end{proof}

Now with bounds for both terms in \eqref{eq:2terms}, we can bound $\|g(x) - \nabla \phi(x)\|$, in probability. 

\begin{theorem} \label{thm:prob_bnd_smoothed} 
Suppose that Assumption  \ref{assum:lip_cont} holds and  $g(x)$ is calculated via \eqref{eq:GSG}. 
If 
\begin{align*}
	N \ge \frac{3n}{\delta r^2} \left( 3\|\nabla\phi(x)\|^2 + \frac{L^2 \sigma^2}{4} (n+2)(n+4) + \frac{4\epsilon_f^2}{\sigma^2} \right),
\end{align*}
then, for all $x \in \mathbb{R}^n$ and $r>0$, 
\begin{align}	\label{eq:GSG_thm}
	\|g({x}) - \nabla\phi({x})\| \leq \sqrt{n} L \sigma + \frac{\sqrt{n}\epsilon_f}{\sigma} + r.
\end{align}
with probability at least $1 - \delta$.

Alternatively, suppose that Assumption  \ref{assum:lip_cont_hess} holds and  $g(x)$ is calculated via \eqref{eq:cGSG}. 
If 
\begin{align*}
	N \ge \frac{3n}{\delta r^2} \left( 3\|\nabla\phi(x)\|^2 + \frac{M^2 \sigma^4}{36} (n+2)(n+4)(n+6) + \frac{\epsilon_f^2}{\sigma^2} \right),
\end{align*}
then, for all $x \in \mathbb{R}^n$ and $r>0$, 
\begin{align}	\label{eq:cGSG_thm}
	\|g({x}) - \nabla \phi({x})\| \leq nM \sigma^2 + \frac{\sqrt{n}\epsilon_f}{\sigma} + r.
\end{align}
with probability at least $1 - \delta$.
\end{theorem}

\begin{proof} The proof of the first part \eqref{eq:GSG_thm} is a straightforward combination of the bound in \eqref{eq:GSG_bound1} and the result of the first part of Lemma \ref{lem:prob_bnd_smoothed}. The proof for the second part \eqref{eq:cGSG_thm} is a straightforward combination of the bound in \eqref{eq:cGSG_bound1} and the result of the second part of Lemma \ref{lem:prob_bnd_smoothed}.
\end{proof}

With the results of Theorem \ref{thm:prob_bnd_smoothed}, we can now derive bounds on $\sigma$ and $N$ that ensure that \eqref{eq:theta_cond} holds with  probability $1-\delta$. To ensure \eqref{eq:theta_cond}, with probability $1-\delta$, using  Theorem \ref{thm:prob_bnd_smoothed} we want the following to hold
\begin{align}
	 \sqrt{n} L \sigma +  \frac{\sqrt{n} \epsilon_f}{\sigma} &\leq \lambda \theta \|\nabla \phi({x})\| \label{eq:guass1},\\
	 r &\leq  (1-\lambda) \theta \|\nabla \phi({x})\|, \label{eq:guass2}
\end{align}
for some $\lambda\in (0,1)$. 

Let us first consider $g(x)$ calculated via \eqref{eq:GSG}. To ensure that \eqref{eq:guass1} holds, we impose conditions derived following the same logic as was done for the case of Forward Finite Differences. Namely,  
 \begin{align*}
	\sigma = \sqrt{ \frac{\epsilon_f}{L}} \quad \text{and} \quad  \|\nabla \phi(x)\|\geq \frac{2\sqrt{nL\epsilon_f}}{\lambda\theta}.
 \end{align*}

Now using these bounds and substituting $r = (1-\lambda) \theta \|\nabla \phi({x})\|$ into the first bound on $N$ in Theorem \ref{thm:prob_bnd_smoothed} we have
\begin{align}\label{eq:GSG_nbound}
	&\frac{3n}{\delta r^2} \left( 3\|\nabla\phi(x)\|^2 + \frac{L^2 \sigma^2}{4} (n+2)(n+4) + \frac{4\epsilon_f^2}{\sigma^2} \right)\\
	& \qquad \leq \frac{9n}{\delta  \theta^2 }\frac{1}{(1-\lambda)^2} + \left (\frac{3(n+2)(n+4)}{16\delta   }  + \frac{3}{\delta}\right ) \frac{\lambda^2}{(1-\lambda)^2} \nonumber.
\end{align}
We are interested in making the lower bound on $N$ as small as possible and hence we are concerned with its dependence on $n$, when $n$ is relatively large. Henceforth, we assume that $n>1$ and choose $\lambda$ such 
that $ \frac{\lambda^2}{(1-\lambda)^2} \leq \frac{1}{n+2}$ so as to reduce the scaling of the second term with $n$ and to simplify the expression. This is always possible, because  $\frac{\lambda^2}{(1-\lambda)^2}$ is monotonically increasing with $\lambda$ and equals $0$ for $\lambda=0$.  Specifically, we can choose $\lambda = \frac{1}{3\sqrt{n}}$, because it is easy to show that for this value of $\lambda$, 
 $ \frac{\lambda^2}{(1-\lambda)^2} \leq \frac{1}{n+2} \le \frac{1}{n}$ for all $n\geq 1$. In fact, for large values of $n$ we can choose
 $\lambda$ to be closer in value to $\frac{1}{\sqrt{n}}$, but for simplicity we will consider the choice that fits all $n$.
Using the fact that $\lambda\leq \frac{1}{\sqrt{n}}$, and thus $\frac{1}{(1-\lambda)^2}\leq \frac{n}{(\sqrt{n}-1)^2}$, and also that $\frac{1}{{n+2}}\leq \frac{1}{2}$, the right hand side of \eqref{eq:GSG_nbound} is bounded from above by 
\begin{align}\label{eq:GSG_nbound_simple}
 \frac{9n}{\delta  \theta^2 }\frac{n}{(\sqrt{n}-1)^2} + \frac{3(n+4)}{16\delta}  + \frac{3}{n\delta}. 
\end{align}
This implies that by choosing $N$ at least as large as the value of \eqref{eq:GSG_nbound_simple}
we ensure that \eqref{eq:guass2} holds. 

We now summarize  the   result  for the gradient approximation computed via \eqref{eq:GSG}, for $\lambda= \frac{1}{3\sqrt{n}}$.
\begin{corollary} \label{corr:GSG} 
Suppose that Assumption \ref{assum:lip_cont}  holds, $n>1$ 
and $g(x)$ is computed via  \eqref{eq:GSG}  with $N$ and $\sigma$ satisfying,
\begin{gather*}
N \geq \frac{9n}{\delta  \theta^2} \frac{n}{(\sqrt{n}-1)^2} + \frac{3(n+4)}{16\delta}  + \frac{3}{n\delta} 
\quad \text{and} \quad
	\sigma=\sqrt{ \frac{\epsilon_f}{L}}.
\end{gather*}
If $ \|\nabla \phi(x)\|\geq \frac{6n\sqrt{L\epsilon_f}}{\theta}$, then \eqref{eq:theta_cond} holds with  probability $1-\delta$. 
\end{corollary}


 
The bound for the number of samples for GSG is  larger than those required by FFD and interpolation, since the latter are fixed at $n$, although both scale linearly in $n$. 
Moreover, the dependence of $N$ on $\delta$ is high. However, this bound is derived as an upper bound, and hence in order to verify that GSG indeed requires a large number of samples to satisfy \eqref{eq:theta_cond} we need to establish the lower bound on $N$. In what follows we show that linear scaling of $N$ with respect to $n$ is necessary to guarantee that \eqref{eq:theta_cond} is satisfied. The dependence on $\delta$ is likely to be too pessimistic and is an artifact of using Chebychev's inequality. In the next section
we analyze a method that estimates gradients using samples uniformly distributed on a sphere, and for which we obtain better dependence on $\delta$ but still linear scaling with $n$. Note, also, that the dependence of the lower bound for $\| \nabla \phi(x)\|$  on $n$ in the GSG case is larger by a factor of $\sqrt{n}$ as compared to the FFD case
 
We now derive the analogous bounds on $N$ and $\sigma$ for the case when $g(x)$ is calculated via \eqref{eq:cGSG}. To ensure \eqref{eq:theta_cond}, with probability $1-\delta$, using  Theorem \ref{thm:prob_bnd_smoothed} we want the following to hold
\begin{align}
	 n M \sigma^2 +  \frac{\sqrt{n} \epsilon_f}{\sigma} &\leq \lambda \theta \|\nabla \phi({x})\| \label{eq:guass3},\\
	 r &\leq  (1-\lambda) \theta \|\nabla \phi({x})\|, \label{eq:guass4}
\end{align}
for some $\lambda\in (0,1)$. In order to ensure that \eqref{eq:guass3} holds, we use the same logic as was done for Central Finite Differences in Section \ref{sec:fd}. Namely, we require the following:
 \begin{align*}
	\sigma = \sqrt[3]{ \frac{\epsilon_f}{2\sqrt{n}M}} \quad \text{and} \quad \|\nabla \phi(x)\|\geq 
\frac{3}{\lambda \theta} \sqrt[3]{\frac{n^2 M \epsilon_f^2}{4}}.
 \end{align*}
Now using these bounds and setting $r = (1-\lambda) \theta \|\nabla \phi({x})\|$ into the second bound on $N$ in Theorem \ref{thm:prob_bnd_smoothed} we have
\begin{align*}
	 &\frac{3n}{\delta r^2} \left( 3\|\nabla\phi(x)\|^2 + \frac{M^2 \sigma^4}{36} (n+2)(n+4)(n+6) + \frac{\epsilon_f^2}{\sigma^2} \right) \\
 & \qquad \leq \frac{9n }{\delta  \theta^2} \frac{1}{(1-\lambda)^2} + \left( \frac{(n+2)(n+4)(n+6)}{48n\delta  }  + \frac{3}{4\delta} \right) \frac{ \lambda^2}{ (1-\lambda)^2}.
\end{align*}
As before, we are interested in making the lower bound on $N$ to scale at most linearly with $n$.  Thus, to achieve this and to simplify the expression we choose $\lambda$ such that $\frac{ \lambda^2}{ (1-\lambda)^2}\leq \frac{n}{(n+2)(n+4)} \le \frac{1}{n}$, which reduces the scaling of the second term with respect to $n$ and simplifies the expression. It is easy to show that $\lambda = \frac{1}{6\sqrt{n}}\leq\frac{1}{\sqrt{n}}$ satisfies  this condition.
 Then, using again the fact that  $\frac{1}{(1-\lambda)^2}\leq \frac{n}{(\sqrt{n}-1)^2}$ and $\frac{n}{(n+2)(n+4)}\leq \frac{1}{2}$  the above expression is bounded by
\begin{align*}
 \frac{9n}{\delta  \theta^2} \frac{n}{(\sqrt{n}-1)^2} + \frac{n+6}{48\delta } + \frac{3}{4n\delta}. 
\end{align*}

We now summarize  the   result  for the gradient approximation computed via \eqref{eq:cGSG}, for $\lambda =  \frac{1}{6\sqrt{n}}$.
\begin{corollary} \label{corr:cGSG} 
Suppose that Assumption \ref{assum:lip_cont_hess} holds, $n>1$
and $g(x)$ is computed via  \eqref{eq:cGSG}  with $N$ and $\sigma$ satisfying,
\begin{gather*}
N \geq  \frac{9n}{\delta  \theta^2} \frac{n}{(\sqrt{n}-1)^2} + \frac{n+6}{48\delta } + \frac{3}{4n\delta} \quad \text{and} \quad
	\sigma= \sqrt[3]{ \frac{\epsilon_f}{2\sqrt{n}M}}.
\end{gather*}
If $\|\nabla \phi(x)\|\geq \frac{18 }{\theta} \sqrt[3]{\frac{n^{7/2} M \epsilon_f^2}{4}}$, then \eqref{eq:theta_cond} holds with  probability $1-\delta$.
\end{corollary}

\subsubsection{Lower bound on $\delta$}
We have demonstrated that if $N\geq\Omega (\frac{9n}{\theta^2\delta})$, then  $\mathbb{P}(\|g(x) - \nabla \phi(x)\| \leq  \theta \|\nabla \phi(x)\|) \geq 1- \delta$; that is, having a   large enough number of samples is {\em sufficient} to ensure accurate gradient approximations with a desired probability. A question that remains is how many samples are {\em necessary} to ensure that accurate gradient approximations are obtained with high probability. Here we derive a lower bound on the probability of failure for \eqref{eq:GSG} to satisfy condition \eqref{eq:theta_cond}; i.e., 
\begin{align} \label{eq:likelybad}
	\mathbb{P}\left(\|g(x) - \nabla \phi(x)\| > \theta \|\nabla \phi(x)\| \right).  
\end{align}


We derive the lower bound for \eqref{eq:likelybad} theoretically, and then illustrate the lower bounds via numerical simulations for the specific case of a simple linear function of the form $f(x) =\phi(x)=a^\intercal x$,  where $a$ is an arbitrary nonzero vector in $\R^n$. For simplicity, through this subsection we assume that $\epsilon(x) = 0 $ for all $x \in \mathbb{R}^n$. In this case, for any $\sigma$, $\nabla F(x)=a$. Note also that in this case $\nabla f(x) = \nabla \phi(x) = \nabla F(x)=a$. We show that while theory gives us a weak lower bound, numerical simulations indicate that the true lower bound is much closer to the upper bound, in terms of dependence on $n$.


We use the following lower bound on the tail of a random variable $X$, derived in \cite{petrov2007lower}. For any $b$ that satisfies $0 \le b \le \mathbb{E}\left[|X|\right] < \infty$, 
\begin{align*}
	\mathbb{P}(|X| > b) \ge \frac{(\mathbb{E}[|X|] - b)^2}{\mathbb{E}[|X|^2]}. 
\end{align*}

We apply this bound  to the random variable $\|g(x) - \nabla F(x)\|$,  and $b=\theta\|a\|$. We have 
\begin{align*}
	\mathbb{P}\left(\|g(x) - \nabla \phi(x)\| > \theta \|\nabla \phi(x)\| \right) &= \mathbb{P}(\|g(x) - \nabla F(x)\| > b)\\
	  &= \mathbb{P}(\|g(x) - \nabla F(x)\|^2 > b^2) \\
	  &\ge \frac{\left(\mathbb{E}\left[\|g(x) - \nabla F(x)\|^2\right] - b^2 \right)^2}{\mathbb{E}\left[\|g(x) - \nabla F(x)\|^4\right]}
\end{align*}
for any $b$ such that $0 \le b \le \sqrt{\mathbb{E} \left[ \|g(x) - \nabla F(x)\|^2 \right]}$. The reason we consider the squared version of the condition is we are unable to calculate $\mathbb{E}\left[\|g(x) - \nabla F(x)\|^k\right]$ when $k$ is odd. 

For brevity, we omit the derivations of $\mathbb{E}\left[\|g(x) - \nabla F(x)\|^2\right]$ and $\mathbb{E}\left[\|g(x) - \nabla F(x)\|^4\right]$ from the main paper, and refer the reader to Appendices  \ref{app:lb_1} and \ref{app:lb_2}, respectively. The required expressions are:
\begin{align}	\label{eq:lower_req_eq1}
	\mathbb{E} \left[\|g(x) - \nabla F(x)\|^2\right] &= \frac{1}{N} (n + 1) a^\intercal a, \\
	\mathbb{E} \left[ \|g(x) - \nabla F(x)\|^4 \right] 
&= \frac{1}{N^3}\left((N-1)(n^2 + 4n + 7) (a^\intercal a)^2 + (3n^2 + 20n + 37) (a^\intercal a)^2\right). \label{eq:lower_req_eq2}
\end{align}
Thus for $\phi(x) = a^\intercal x$,
\begin{align*}
\mathbb{P} \left( \|g(x) - a\| > \theta \|a\|\right) &\ge \frac{ N^4 \left( \frac{1}{N}(n+1) a^\intercal a - \theta^2 a^\intercal a \right)^2 }{N(N-1)(n^2 + 4n + 7) (a^\intercal a)^2 + N(3n^2 + 20n + 37) (a^\intercal a)^2} \\
&= \frac{ N \left( (n+1) a^\intercal a - N \theta^2 a^\intercal a  \right)^2 }{(N-1)(n^2 + 4n + 7) (a^\intercal a)^2 + (3n^2 + 20n + 37) (a^\intercal a)^2}\\ 
&= \frac{ N \left( (n+1)  - \theta^2 N \right)^2 }{(N-1)(n^2 + 4n + 7) + (3n^2 + 20n + 37) }
\end{align*} 
for any $\theta$ and $N$ such that $0 \le \theta^2 a^\intercal a  \le \frac{1}{N} (n+1) a^\intercal a$. 


Consider $n$ large enough such that  $4(n+1)^2 \ge n^2 + 4n +7$ which is satisfied for $n \ge \frac{\sqrt{13} -2}{3} \approx 0.54$; and $4(n+1)^2 \ge 3n^2 +20n+37$ which is satisfied for $n \ge 6 + \sqrt{69} \approx 14.31$ (henceforth we assume that $n \geq 15$). Then we have,
\begin{align*}
\frac{N((n+1) - \theta^2 N)^2}{(N-1)(n^2+4n+7) + (3n^2 +20n+37)} \ge 
\frac{N((n+1) - \theta^2 N)^2}{(N-1)4(n+1)^2 + 4(n+1)^2} = 
\frac{((n+1) - \theta^2 N)^2}{4(n+1)^2}. 
\end{align*}
Thus, from 
\begin{align*} 
	\mathbb{P} \left( \|g(x) - a\| > \theta \|a\|\right) \ge \frac{((n+1) - \theta^2 N)^2}{4(n+1)^2} \ge \delta,  
\end{align*} 
we get
\begin{align*}
N \le \frac{(n+1)(1 - 2\sqrt{\delta})}{\theta^2} 
\Rightarrow 
\mathbb{P}(\|g(x) - \nabla \phi (x) \| > \theta \|\nabla \phi (x)\|) \ge \delta.
\end{align*}
It follows that for any $0 < \delta < \frac{1}{4}$, $n \ge 15$ and $N\leq \frac{1}{\theta^2}(1-2\sqrt{\delta})(n+1)$ 
\begin{align*}
\mathbb{P}(\|g(x) - \nabla \phi (x) \| > \theta \|\nabla \phi (x)\|) \ge \delta.
\end{align*}
In other words, to have $\mathbb{P}(\|g(x) - \nabla \phi (x) \| \le \theta \|\nabla \phi (x)\|) > 1 - \delta$, it is necessary to have 
$N > \frac{(1 - 2\sqrt{\delta})}{\theta^2}(n+1)$, which is a linear function in $n$.


We now show through numerical simulation of the specific case of $\phi(x)=a^\intercal x$ that in fact for much larger values of $\delta$, $N\geq n$ is required
to achieve \eqref{eq:theta_cond} for any $\theta<1$. Specifically, Figure \ref{fig:lowerbound} shows the distribution of  
\begin{align*}
   \theta=\frac{ \|g(x) - \nabla \phi(x)\| }{ \|\nabla \phi(x)\|}, 
\end{align*}
approximately computed via running $10000$ experiments,   where $\phi(x)=e^\intercal x$ ($e$ is a vector of all ones) and $n=32$, for different choices of $N \in \{ 1,2,4,8,16,32,64,128,256,512\}$. As is clear, $\theta$ is never smaller than $1$ when $N=1$. Moreover, $\theta$ is smaller than $\frac{1}{2}$, which is required by the theory in \cite{berahas2019global}, only about half the time when $N=128=4n$. Figure \ref{fig1k} shows the percent of successful trials ($\theta < \frac{1}{2}$) versus the size of the sample set ($N$), and Table \ref{tab1l} shows statistics of the empirical experiments for different  sizes of the sample set ($N$). 
As expected, as $N$ grows, the value of $\theta$ decreases, something that is not surprising, but at the same time not captured by the derived lower bound. Thus, we conclude  that the theoretical  lower bound we derive here is weak and to satisfy  \eqref{eq:theta_cond} with $\theta<  \frac{1}{2}$ and  probability of at least $\frac{1}{2}$ the size of the sample set needs to be larger than $n$.  A stronger theoretical lower bound supporting this claim remains an open question.


 In Section  \ref{sec:num} we present numerical evidence that shows that for a variety of functions choosing $N$ to be a small constant almost always results in  large values of $ \frac{ \|g(x) - \nabla \phi(x)\| }{ \|\nabla \phi(x)\|}
 $ with probability close to $1$.

\begin{figure}[ht]
    \centering
    \begin{subfigure}{.19\textwidth}
        \centering
        \includegraphics[trim=5 5 50 25,clip,width=0.95\linewidth]{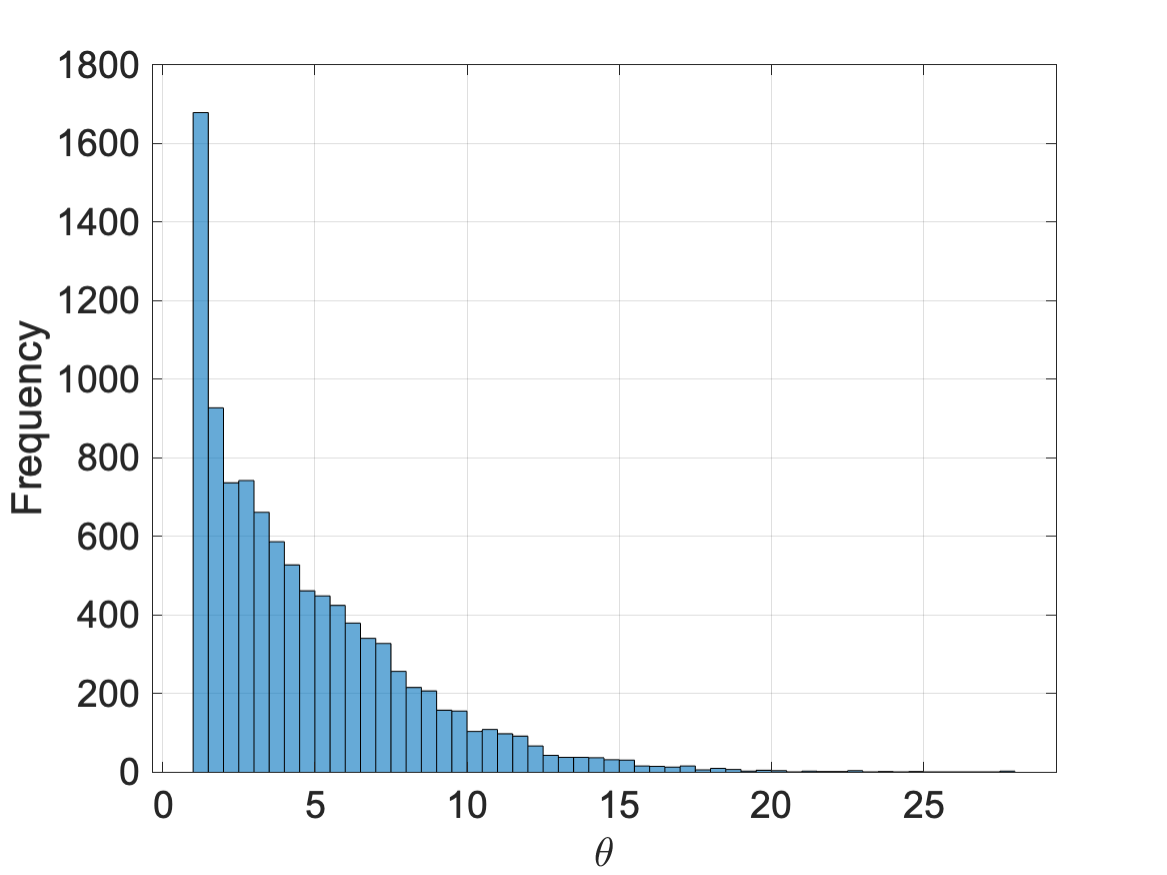}
        \caption{  $N=1 \; (n/32)$}
    \end{subfigure}
    \begin{subfigure}{.19\textwidth}
        \centering
        \includegraphics[trim=10 5 40 25,clip,width=0.95\linewidth]{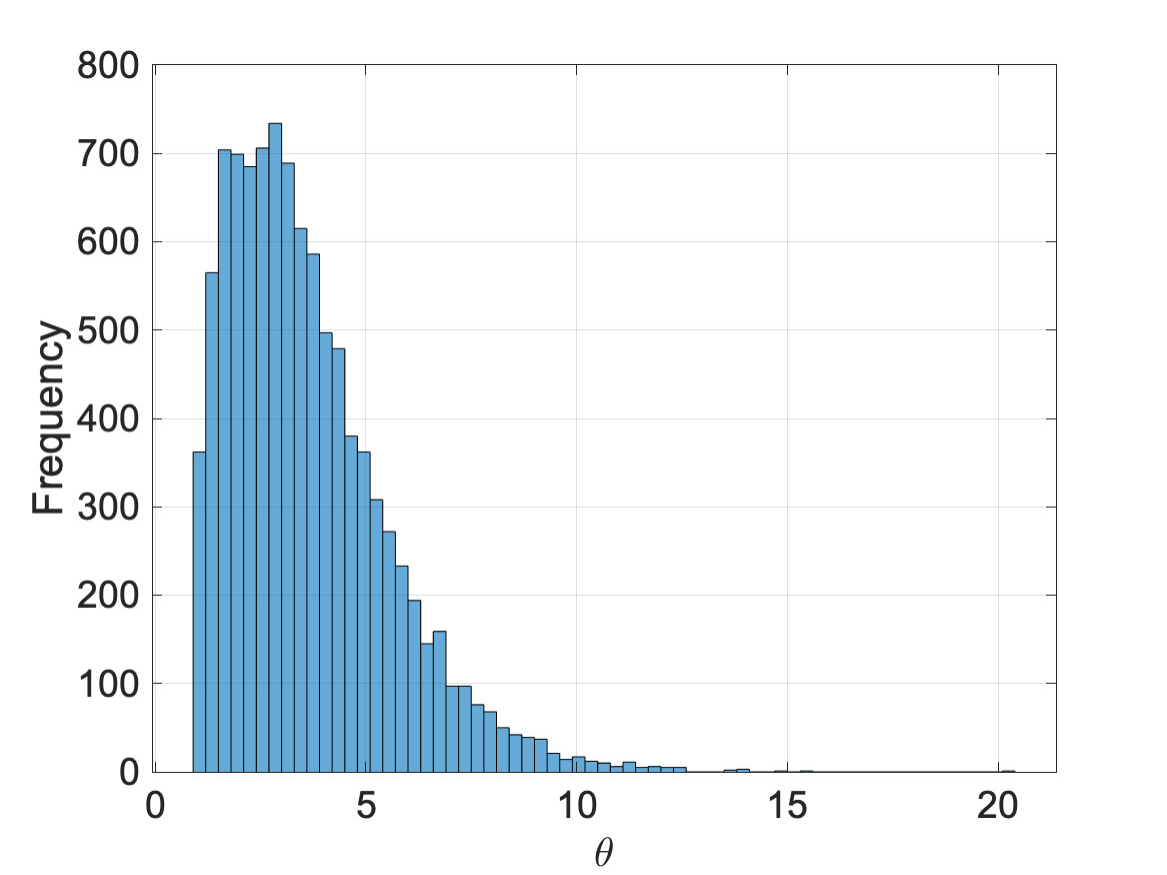}
        \caption{  $N=2 \; (n/16)$}
    \end{subfigure}
    \begin{subfigure}{.19\textwidth}
        \centering
        \includegraphics[trim=10 5 40 25,clip,width=0.95\linewidth]{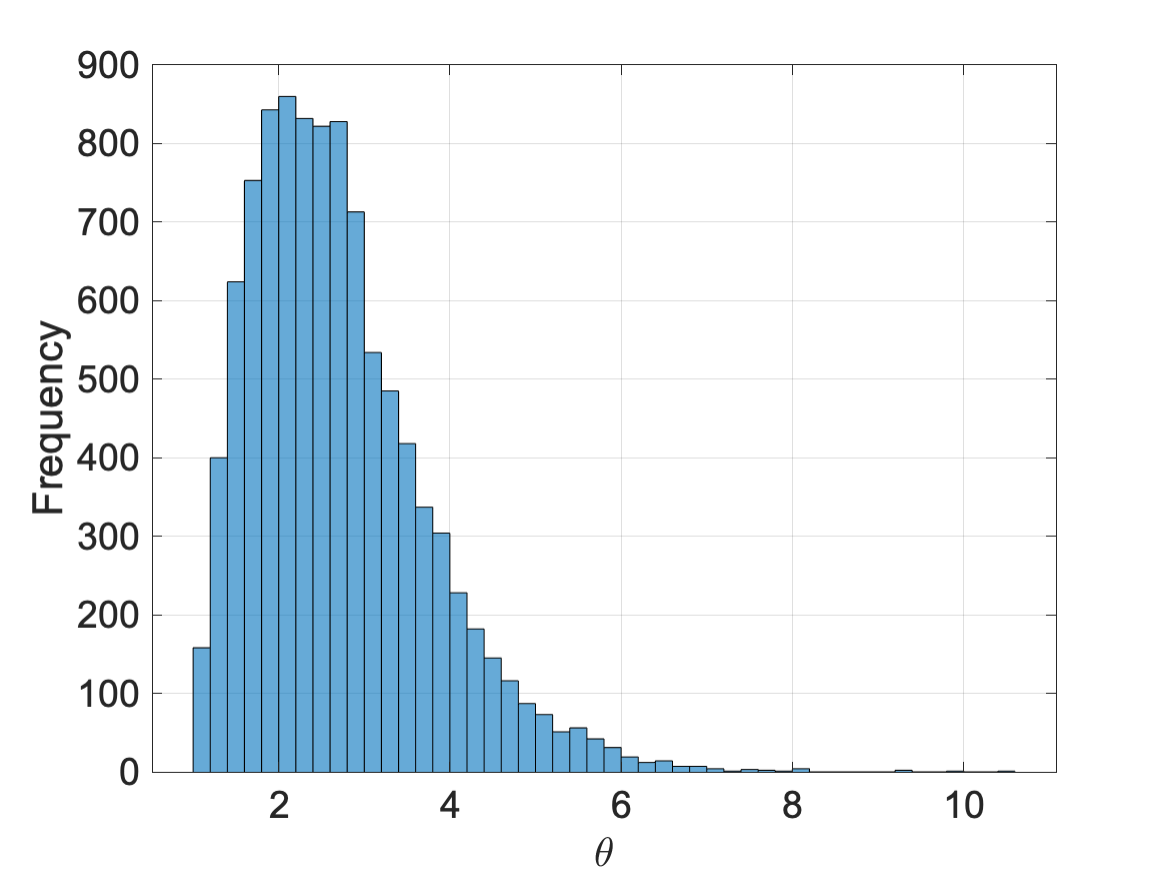}
        \caption{  $N=4 \; (n/8)$}
    \end{subfigure}
    \begin{subfigure}{.19\textwidth}
        \centering
        \includegraphics[trim=10 5 40 25,clip,width=0.95\linewidth]{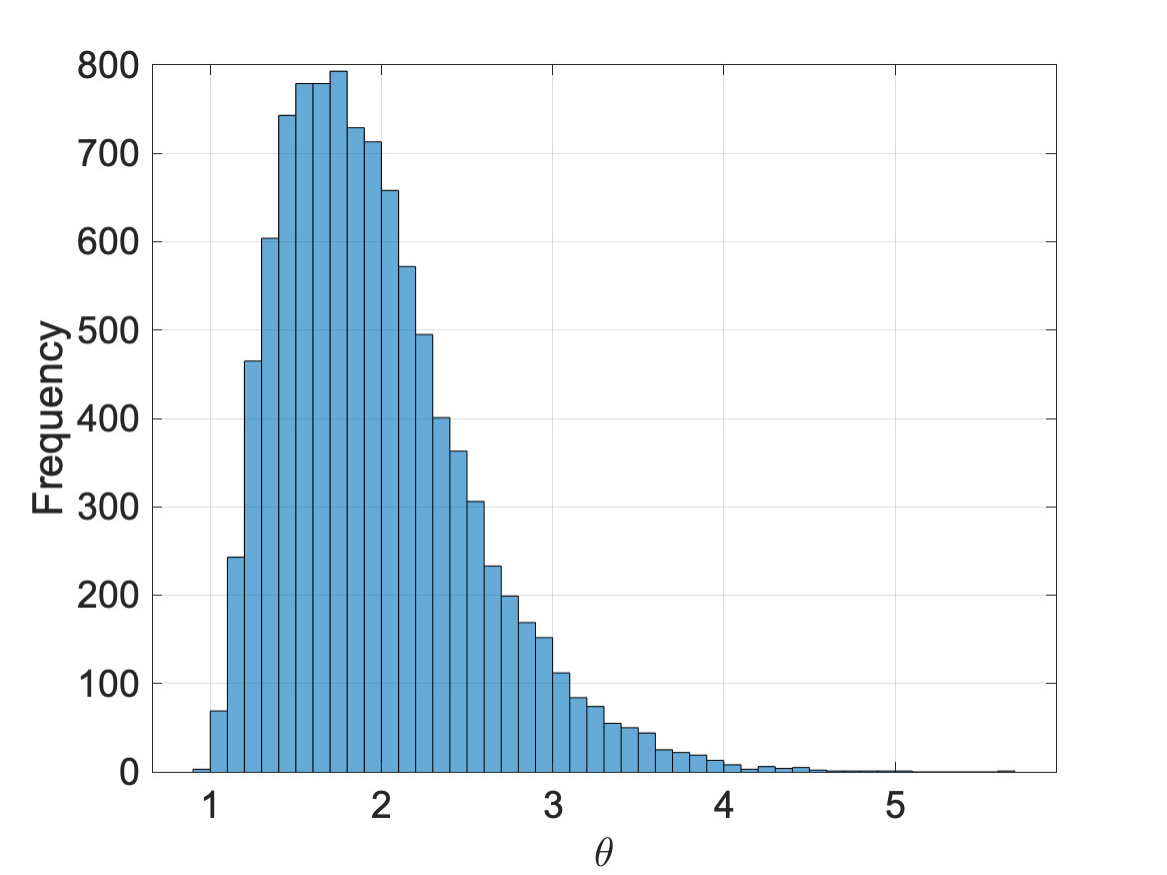}
        \caption{  $N=8 \; (n/4)$}
    \end{subfigure}
    \begin{subfigure}{.19\textwidth}
        \centering
        \includegraphics[trim=10 5 40 25,clip,width=0.95\linewidth]{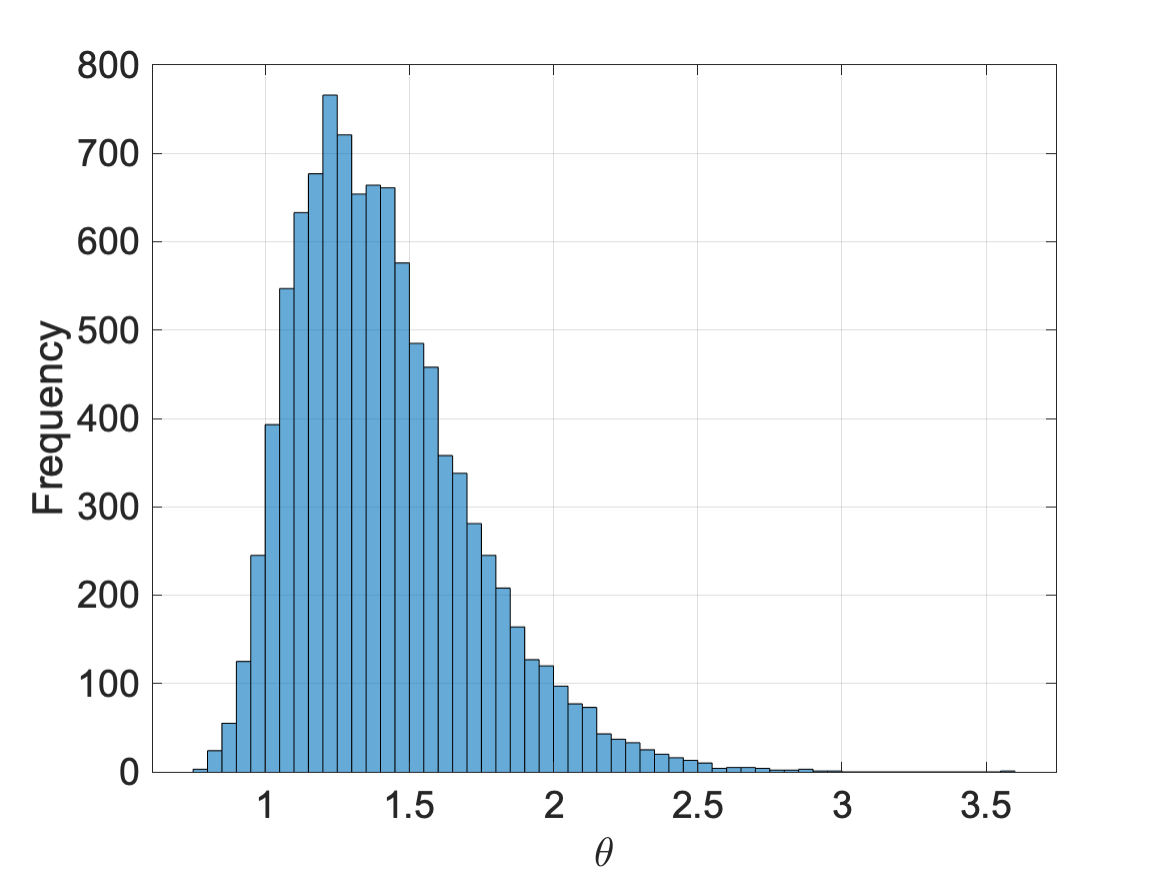}
        \caption{  $N=16 \; (n/2)$}
    \end{subfigure}
    
    \vspace{0.25cm}
     \begin{subfigure}{.19\textwidth}
        \centering
        \includegraphics[trim=10 5 40 25,clip,width=0.95\linewidth]{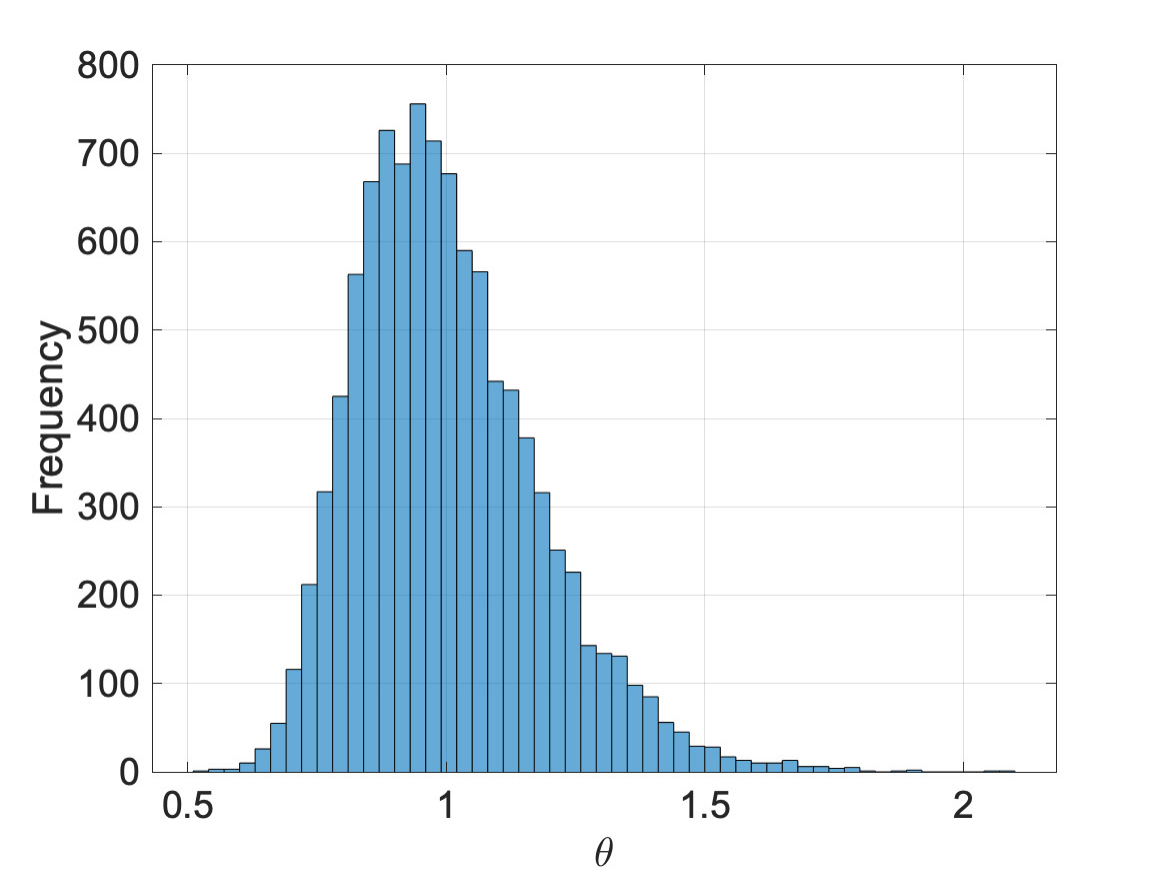}
        \caption{  $N=32 \; (n)$}
    \end{subfigure}
    \begin{subfigure}{.19\textwidth}
        \centering
        \includegraphics[trim=10 5 40 25,clip,width=0.95\linewidth]{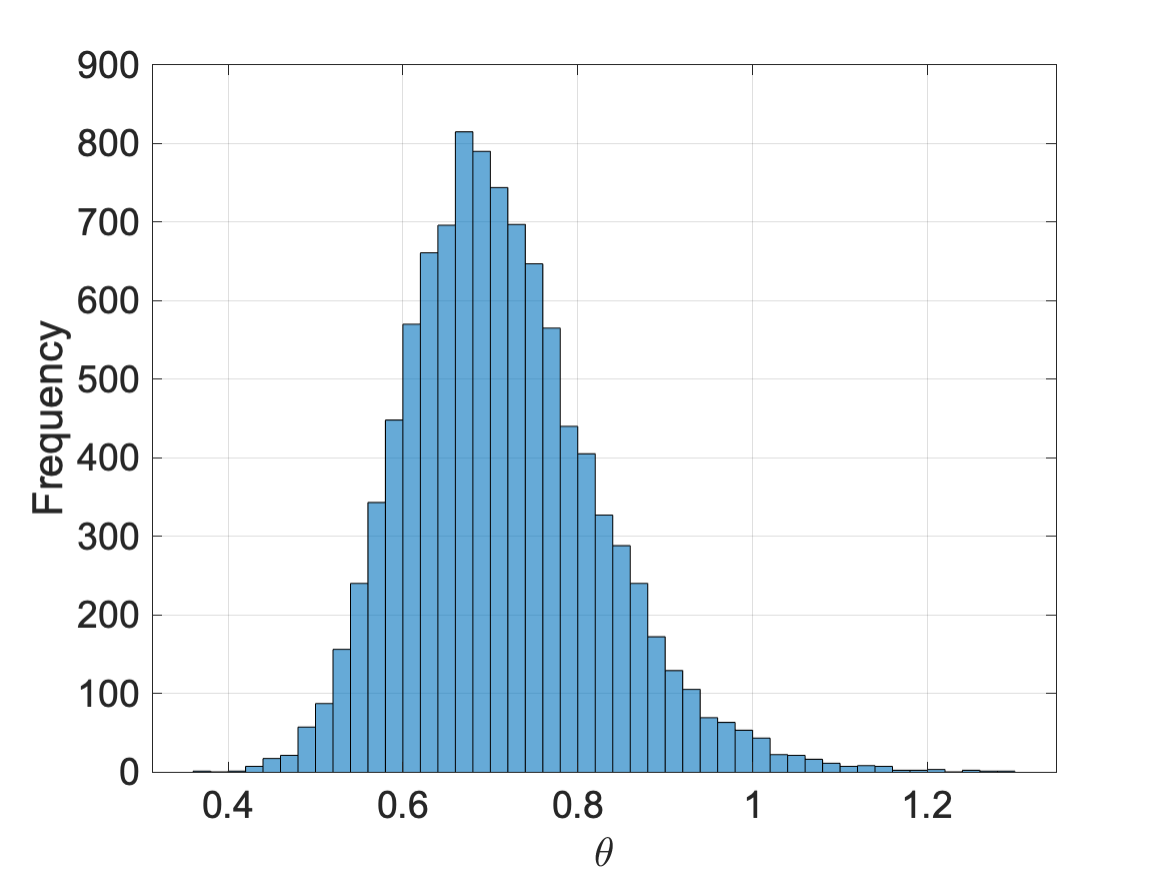}
        \caption{  $N=64 \; (2n)$}
    \end{subfigure}
    \begin{subfigure}{.19\textwidth}
        \centering
        \includegraphics[trim=10 5 40 25,clip,width=0.95\linewidth]{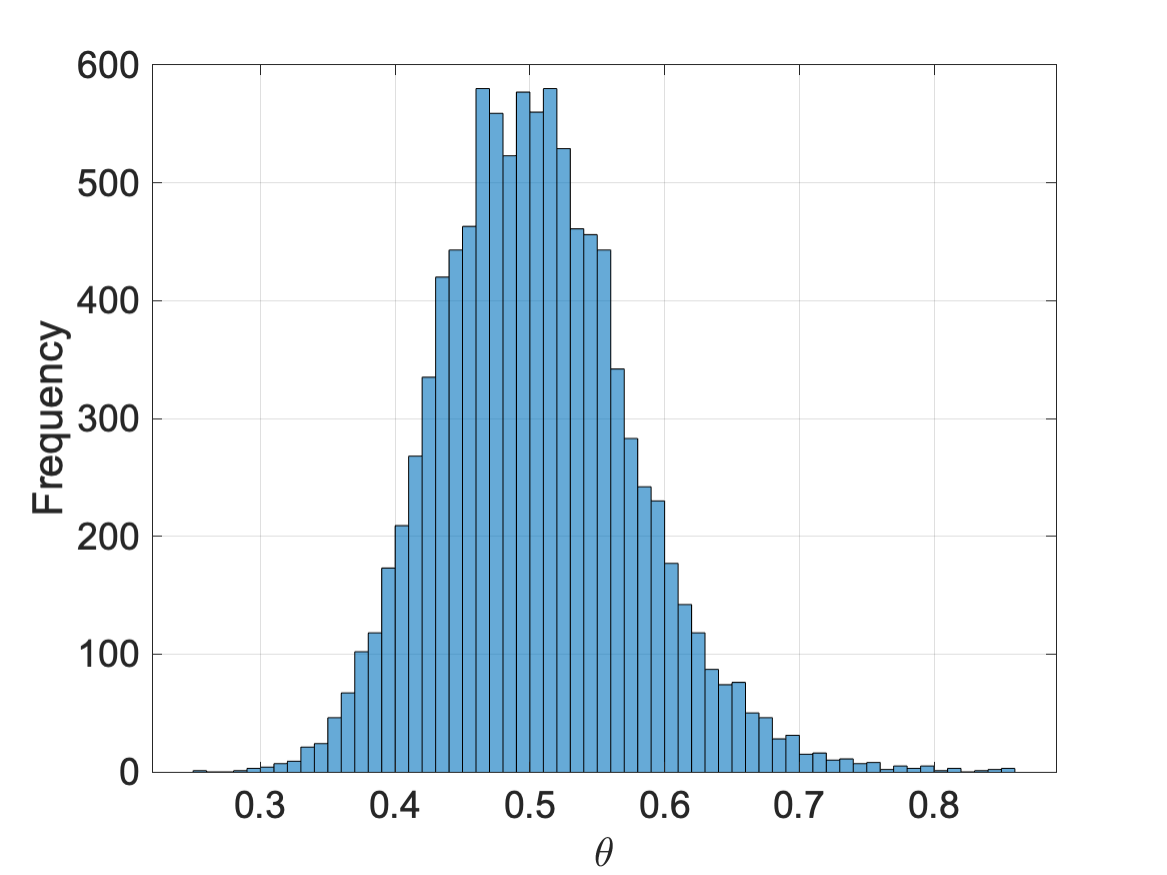}
        \caption{  $N=128 \; (4n)$}
    \end{subfigure}
    \begin{subfigure}{.19\textwidth}
        \centering
        \includegraphics[trim=10 5 40 25,clip,width=0.95\linewidth]{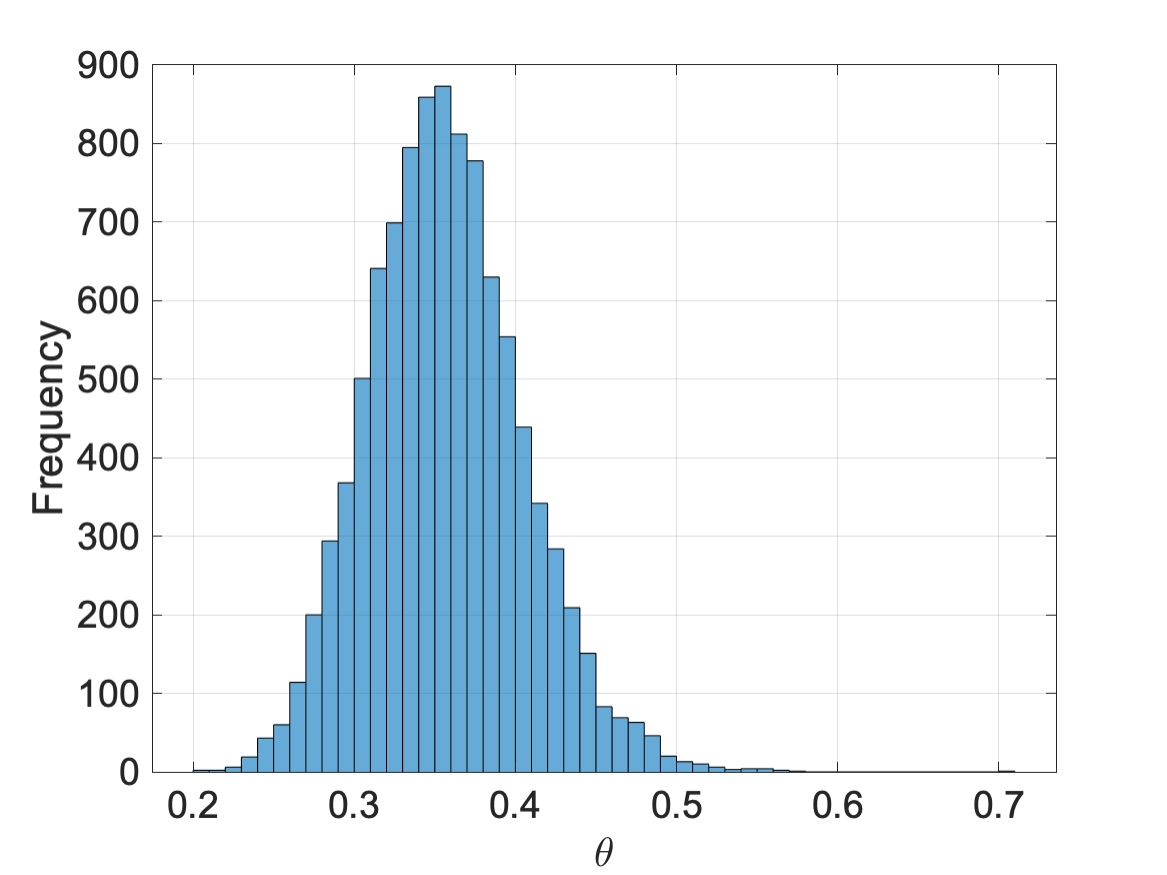}
        \caption{  $N=256 \; (8n)$}
    \end{subfigure}
    \begin{subfigure}{.19\textwidth}
        \centering
        \includegraphics[trim=10 5 40 25,clip,width=0.95\linewidth]{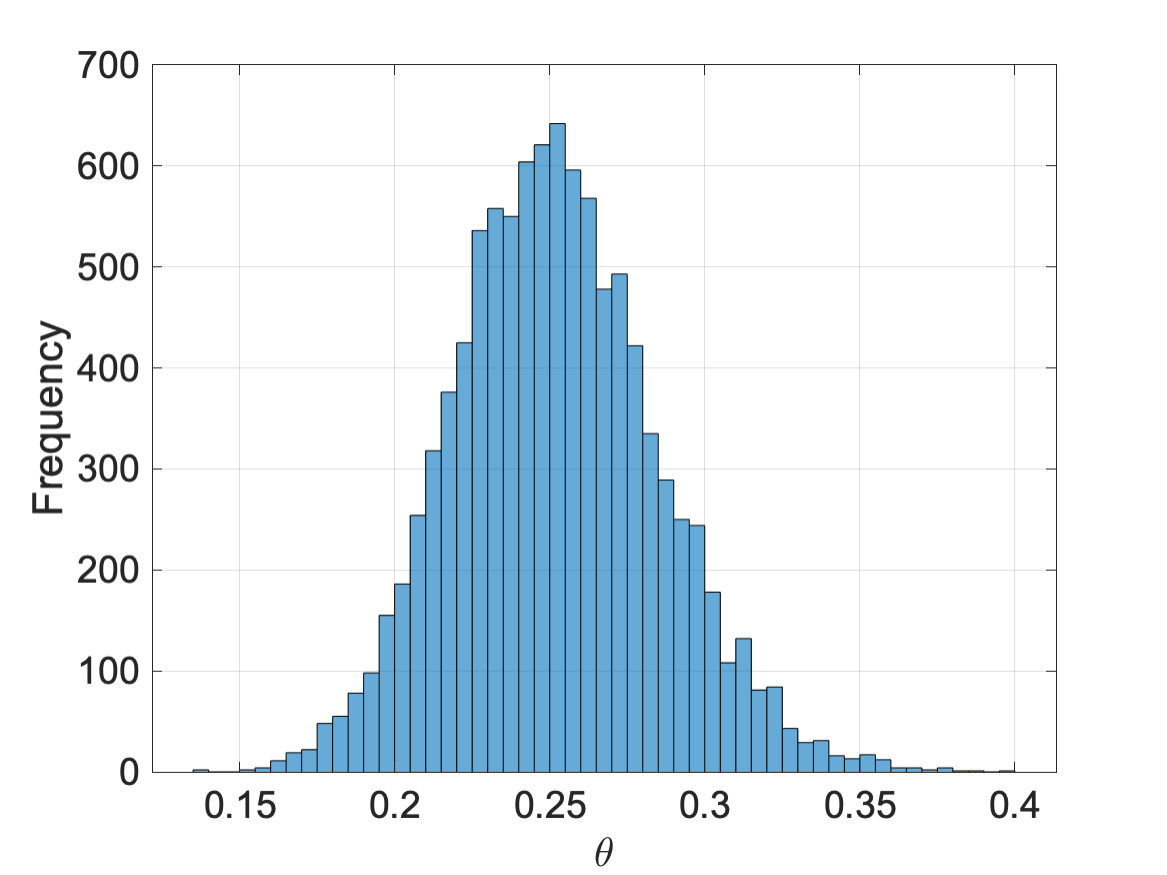}
        \caption{  $N=512 \; (16n)$}
    \end{subfigure}
    
     \vspace{0.25cm}
    \begin{subfigure}{.35\textwidth}
        \centering
        \includegraphics[trim=20 5 40 25,clip,width=0.95\linewidth]{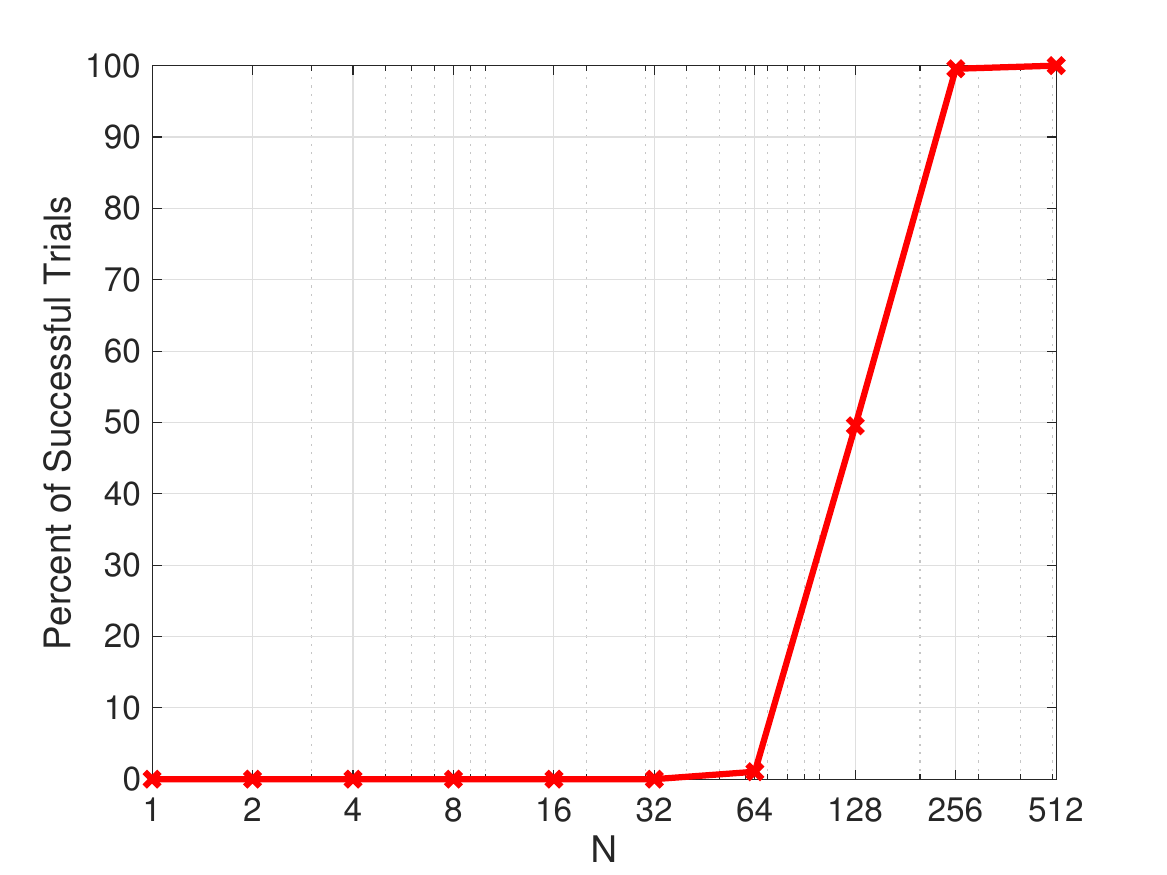}
        \caption{  Percent of successful trials vs. $N$ \label{fig1k}}
    \end{subfigure}
    \begin{subfigure}{.51\textwidth}
    \scriptsize
        \centering
\begin{tabular}{ccccc}
\toprule
$\pmb{N}$   & \textbf{Mean ($\pmb{\theta}$)} & \textbf{Median ($\pmb{\theta}$)} & \textbf{Variance ($\pmb{\theta}$)} & \begin{tabular}[c]{@{}c@{}}\textbf{Percent of} \\ \textbf{Successful} \\ \textbf{Trials}\end{tabular} \\ \midrule
1 ($n/32$)  & 4.62            & 3.69              & 11.44               & 0                                                                          \\ \hdashline
2  ($n/16$) & 3.62            & 3.23              & 3.67                & 0                                                                          \\ \hdashline
4   ($n/8$) & 2.70            & 2.53              & 1.07                & 0                                                                          \\ \hdashline
8   ($n/4$) & 1.96            & 1.87              & 0.31                & 0                                                                          \\ \hdashline
16  ($n/2$) & 1.41            & 1.36              & 0.093               & 0                                                                          \\ \hdashline
32  ($n$) & 1.00            & 0.98              & 0.032               & 0                                                                          \\ \hdashline
64 ($2n$) & 0.71            & 0.70              & 0.012               & 1.04                                                                       \\ \hdashline
128 ($4n$) & 0.50            & 0.50              & 0.0051              & 49.53                                                                      \\ \hdashline
256 ($8n$) & 0.36            & 0.35              & 0.0023              & 99.56                                                                      \\ \hdashline
512 ($16n$) & 0.25            & 0.25              & 0.0011              & 100                                                                  \\      \bottomrule
\end{tabular}
  \caption{ Summary of Results \label{tab1l}}
    \end{subfigure}

    \caption{Distribution of $\theta$ for $\phi(x)=e^\intercal x$, at $x=e$, where $e$ is a vector of all ones, and  $n=32$.  }    
    \label{fig:lowerbound}
\end{figure}

\subsection{Gradient Estimation via Smoothing on a Sphere} \label{sec:ball_smooth}

Similar to the Gaussian smoothing technique, one can also smooth the function $f$ with a uniform distribution on a ball, i.e.,  
\begin{align} \label{eq:Fball} 
F(x) &=  \mathbb{E}_{y \sim \mathcal{U}(\mathcal{B}(x, \sigma))} [f(y)] = \int_{\mathcal{B}(x, \sigma)} f(y) \frac{1}{V_n(\sigma)} dy \nonumber \\
&= \mathbb{E}_{u \sim \mathcal{U}(\mathcal{B}(0, 1))} [f(x+\sigma u)] = \int_{\mathcal{B}(0, 1)} f(x + \sigma u) \frac{1}{V_n(1)} d u,
\end{align} 
where $\mathcal{U}(\mathcal{B}(x, \sigma))$ denotes the multivariate uniform distribution on a ball  of radius $\sigma$ centered at $x$ and $\mathcal{U}(\mathcal{B}(0, 1))$ denotes the multivariate uniform distribution on a ball of radius $1$ centered at $0$. The function $V_n(\sigma)$ represents the volume of a ball  in $\R^n$ of radius $\sigma$. It was shown in \cite{flaxman2005online} that the gradient of $F$ can be expressed as
\begin{align*} 
	\nabla F(x) = \frac{n}{\sigma} \mathbb{E}_{u \sim \mathcal{U}(\mathcal{S}(0, 1))} [f(x+\sigma u) u], 
\end{align*}
where $\mathcal{S}(0, 1)$ represents a unit sphere of radius $1$ centered at $0$. This leads to three ways of approximating the gradient with only function evaluations using sample average approximations 
\begin{align}
g(x) &= \frac{n}{N\sigma} \sum_{i=1}^N f(x+\sigma u_i) u_i, \label{eq:ncBSG} \\ 
g(x) &= \frac{n}{N} \sum_{i=1}^N \frac{f(x+\sigma u_i) - f(x)}{\sigma} u_i, \label{eq:BSG} \\
g(x) &= \frac{n}{N} \sum_{i=1}^N \frac{f(x+\sigma u_i) - f(x-\sigma u_i)}{2\sigma} u_i, \label{eq:cBSG}
\end{align}
with $N$ independently and identically distributed random vectors $\{u_i\}_{i=1}^N$ following a uniform distribution on the unit sphere. Similar to the case with Gaussian smoothing, the variance of $\eqref{eq:ncBSG}$ explodes when $\sigma$ goes to zero, and thus we do not consider this formula. We analyze \eqref{eq:BSG}, which we refer to as ball smoothed gradient (BSG) and \eqref{eq:cBSG} which we refer to as central  BSG (cBSG). 

Again, as in the Gaussian smoothed case, there are two sources of error in the gradient approximations, and namely,
\begin{align} \label{eq:2terms_sphere} 
	\| g(x) - \nabla \phi(x) \|  \leq \| \nabla F(x) - \nabla \phi(x)\| + \| g(x) - \nabla F(x)\|.
\end{align} 
Let Assumption \ref{assum:bounded_noise} hold. One can bound the first term as follows; if the function $\phi$ has $L$-Lipschitz continuous gradients, that is if Assumption \ref{assum:lip_cont} holds, then
\begin{align}	\label{eq:BSG_bound1}
	\|\nabla F(x) - \nabla \phi(x)\| \le L\sigma  + \frac{n \epsilon_f}{\sigma},
\end{align}
and if the function $\phi$ has $M$-Lipschitz continuous Hessians, that is if Assumption \ref{assum:lip_cont_hess} holds, then
\begin{align}	\label{eq:cBSG_bound1}
	\|\nabla F(x) - \nabla \phi(x)\| \le M\sigma^2 + \frac{n \epsilon_f}{\sigma} .
\end{align}
The proofs are given in Appendices \ref{app:BSG_bound1} and \ref{app:cBSG_bound1}, respectively. 

For the second error term in \eqref{eq:2terms_sphere}, similar to the case of Gaussian smoothing, we begin with the variance of $g(x)$. The variance of \eqref{eq:BSG} can be expressed as
\begin{align} \label{eq:var BSG} 
	\Var{g(x)} = \frac{n^2}{N} \mathbb{E}_{u \sim \mathcal{U}(\mathcal{S}(0,1))} \left[\left(\frac{f(x+\sigma u) - f(x)}{\sigma} \right)^2 u u^\intercal \right] - \frac{1}{N} \nabla F(x) \nabla F(x)^\intercal,
\end{align} 
and the variance of \eqref{eq:cBSG} can be expressed as
\begin{align} \label{eq:var cBSG} 
\Var{g(x)} = \frac{n^2}{N} \mathbb{E}_{u \sim \mathcal{U}(\mathcal{S}(0,1))} \left[ \left(\frac{f(x+\sigma u) - f(x-\sigma u)}{2\sigma} \right)^2 u u^\intercal\right] - \frac{1}{N} \nabla F(x) \nabla F(x)^\intercal .
\end{align}

For a random variable $u \in \mathbb{R}^n$ that is uniformly distributed on the unit sphere $\mathcal{S}(0,1)$ $\subset \R^n$, we have 
\begin{equation} \label{eq:uos} \everymath{\displaystyle} \begin{aligned} 
	&\mathbb{E}_{u \sim \mathcal{U}(\mathcal{S}(0,1))} \left[(a^\intercal u)^2 u u^\intercal \right]= \frac{a^\intercal a I + 2 a a^\intercal}{n(n+2)} \\
	&\mathbb{E}_{u \sim \mathcal{U}(\mathcal{S}(0,1))} \left[a^\intercal u \|u\|^k uu^\intercal \right]= 0_{n\times n} \text{ for } k = 0,1,2,...  \\
	&\mathbb{E}_{u \sim \mathcal{U}(\mathcal{S}(0,1))} \left[ \|u\|^k u u^\intercal \right]  = \frac{1}{n} I \text{ for } k = 0,1,2,... , 
\end{aligned} \end{equation}
where $a \in \R^n$ is any constant vector; see Appendix~\ref{app:uos} for derivations. We now provide bounds for the variances of BSG and cBSG under the assumption of Lipschitz continuous gradients and Hessians, respectively. 

\begin{lemma}  \label{lem:BSG var bound}
Under Assumption  \ref{assum:lip_cont}, if $g(x)$ is calculated by \eqref{eq:BSG}, then, for all $x \in \mathbb{R}^n$, $\Var{g(x)} \preceq \kappa(x) I$ where 
\begin{align*}		
	\kappa(x) =  \frac{3}{N} \left( \frac{3n}{n+2} \|\nabla\phi(x)\|^2 + \frac{nL^2 \sigma^2}{4} + \frac{4n\epsilon_f^2}{\sigma^2} \right).
\end{align*} 

Alternatively, under Assumption  \ref{assum:lip_cont_hess}, if $g(x)$ is calculated by  \eqref{eq:cBSG}, then, for all $x \in \mathbb{R}^n$, $\Var{g(x)} \preceq \kappa(x) I$ where
\begin{align*}		
	\kappa(x) = \frac{3}{N} \left( \frac{3n}{n+2} \|\nabla\phi(x)\|^2 + \frac{nM^2 \sigma^4}{36} + \frac{n\epsilon_f^2}{\sigma^2} \right).
\end{align*} 
\end{lemma}

\begin{proof} 
Analoguous to the proof of Lemma~\ref{lemma:bnd_var}, we derive from \eqref{eq:var BSG} to get 
\begin{align*}
	\omit\rlap{$\Var{g(x)}$} \\
	\preceq& \frac{3n^2}{N} \mathbb{E}_{u \sim \mathcal{U}(\mathcal{S}(0,1))} \left[\left(\frac{\phi(x+\sigma u) - \phi(x) - \sigma \nabla\phi(x)^\intercal u}{\sigma} \right)^2 u u^\intercal + \left(\frac{\epsilon(x+\sigma u) - \epsilon(x)}{\sigma} \right)^2 u u^\intercal+ (\nabla\phi(x)^\intercal u)^2 uu^\intercal \right] \\
	\preceq& \frac{3n^2}{N} \mathbb{E}_{u \sim \mathcal{U}(\mathcal{S}(0,1))} \left[\left(\frac{L\sigma}{2} u^\intercal u\right)^2 u u^\intercal + \left(\frac{2\epsilon_f}{\sigma} \right)^2 u u^\intercal + \left(\nabla\phi(x)^\intercal u\right)^2 u u^\intercal \right] \\
	\stackrel{\mathclap{\mathrm{\eqref{eq:uos}}}}{=}{}& \ \frac{3n^2}{N} \left( \frac{L^2 \sigma^2}{4n} I + \frac{4\epsilon_f^2}{\sigma^2n} I + \frac{\|\nabla\phi(x)\|^2}{n(n+2)} I + \frac{2}{n(n+2)} \nabla\phi(x) \nabla\phi(x)^\intercal \right) \\
	\preceq& \frac{3}{N} \left( \frac{nL^2 \sigma^2}{4} + \frac{4n\epsilon_f^2}{\sigma^2} + \frac{3n}{n+2} \|\nabla\phi(x)\|^2 \right) I.
\end{align*}
For cBSG, by \eqref{eq:var cBSG} we have 
\begin{align*}
	\omit\rlap{$\Var{g(x)}$} \\
	\preceq& \frac{3n^2}{N} \mathbb{E}_{u \sim \mathcal{U}(\mathcal{S}(0,1))} \left[\left(\frac{\phi(x+\sigma u) - \phi(x-\sigma u) - 2\sigma \nabla\phi(x)^\intercal u}{2\sigma} \right)^2 u u^\intercal + \left(\frac{\epsilon(x+\sigma u) - \epsilon(x)}{2\sigma} \right)^2 u u^\intercal+ (\nabla\phi(x)^\intercal u)^2 uu^\intercal \right] \\
	\preceq& \frac{3n^2}{N} \mathbb{E}_{u \sim \mathcal{U}(\mathcal{S}(0,1))} \left[\left(\frac{M\sigma^2}{6} \|u\|^3 \right)^2 u u^\intercal + \left(\frac{2\epsilon_f}{2\sigma} \right)^2 u u^\intercal + (\nabla\phi(x)^\intercal u)^2 uu^\intercal \right] \\
	\stackrel{\mathclap{\mathrm{\eqref{eq:uos}}}}{=}{}& \ \frac{3n^2}{N} \left( \frac{M^2 \sigma^4}{36n} I + \frac{\epsilon_f^2}{\sigma^2n} I + \frac{\|\nabla\phi(x)\|^2}{n(n+2)} I + \frac{2}{n(n+2)} \nabla\phi(x) \nabla\phi(x)^\intercal \right) \\
	\preceq& \frac{3}{N} \left( \frac{nM^2 \sigma^4}{36} + \frac{n\epsilon_f^2}{\sigma^2} + \frac{3n}{n+2} \|\nabla\phi(x)\|^2 \right) I. 
\end{align*}
\end{proof}


Using the results of Lemma \ref{lem:BSG var bound}, we can bound the quantity $\| g(x) - \nabla F(x) \|$ in \eqref{eq:2terms_sphere}, with probability $1-\delta$, using Chebyshev's inequality, just as we did in the case of GSG. However, ball smoothed gradient approach has a significant advantage over Gaussian smoothing in that
it allows the use of Bernstein's  inequality \cite[Theorem 6.1.1]{tropp2015introduction} instead of Chebychev's and the resulting bound on $N$ has a significantly improved dependence on the probability $\delta$. 

Bernstein's inequality applies here because, unlike GSG (and cGSG), BSG (and cBSG) enjoys a deterministic bound on the error term $n \frac{f(x+\sigma {u}) - f(x)}{\sigma} {u}  - F(x)$; see proof of Lemma \ref{lem:prob_bnd_sphere_smoothed_bern}.

\begin{lemma} 	\label{lem:prob_bnd_sphere_smoothed_bern}
Let $F$ be a ball smoothed approximation of $f$ \eqref{eq:Fball}. Under Assumption  \ref{assum:lip_cont}, if $g(x)$ is  calculated via \eqref{eq:BSG} with sample size
\begin{align*}
	N \ge \left[ \frac{6n^2}{r^2} \left( \frac{\|\nabla \phi(x)\|^2}{n} + \frac{L^2\sigma^2}{4} + \frac{4\epsilon_f^2}{\sigma^2} \right)  + \frac{2n}{3r} \left( 2 \|\nabla \phi(x)\| + L \sigma + \frac{4\epsilon_f}{\sigma} \right) \right] \log \frac{n+1}{\delta},
\end{align*}
then, for all $x \in \mathbb{R}^n$, $\|g(x) - \nabla F(x)\| \le r$ holds with probability at least $1 - \delta$, for any $r>0$ and $0< \delta <1$. 

Alternatively, under Assumption  \ref{assum:lip_cont_hess} if $g(x)$ is calculated via \eqref{eq:cBSG} with sample size $2N$ where
\begin{align*}
	N \ge \left[ \frac{6n^2}{r^2} \left( \frac{\|\nabla \phi(x)\|^2}{n} + \frac{M^2\sigma^4}{36} + \frac{\epsilon_f^2}{\sigma^2}\right) + \frac{2n}{3r} \left( 2 \|\nabla \phi(x)\| + \frac{M \sigma^2}{3} + \frac{2 \epsilon_f}{\sigma} \right) \right] \log \frac{n+1}{\delta}, 
\end{align*}
then, for all $x \in \mathbb{R}^n$, $\|g(x) - \nabla F(x)\| \le r$ holds with probability at least $1 - \delta$, for any $r>0$ and $0< \delta <1$.
\end{lemma}

\begin{proof}
We first note that
\begin{align*}
\mathbb{E}_{u \sim \mathcal{U}(\mathcal{S}(0,1))} \left[\frac{n}{N} \frac{f(x+\sigma u) - f(x)}{\sigma} u - \frac{1}{N} \nabla F(x)\right] &= 0, 
\end{align*}
and 
\begin{align*}
&\left\| \frac{n}{N} \frac{f(x+\sigma {u}) - f(x)}{\sigma} {u}  - \frac{1}{N} \nabla F(x) \right\| \\
&\qquad ={} \left\| \frac{n}{N} \frac{f(x+\sigma {u}) - f(x)}{\sigma} {u}  - \frac{n}{N} \mathbb{E}_{v \sim \mathcal{U}(\mathcal{S}(0,1))}\left[ \frac{f(x+\sigma v) - f(x)}{\sigma} v \right] \right\| \\
&\qquad \le{}  \frac{n}{N\sigma} \left| f(x+\sigma u) - f(x)\right|  \|u\| + \frac{n}{N\sigma} \mathbb{E}_{v \sim \mathcal{U}(\mathcal{S}(0,1))}\left[ \left| f(x+\sigma v) - f(x) \right| \|v\| \right]\\
&\qquad ={} \frac{n}{N\sigma} \left| \phi(x+\sigma u) + \epsilon(x+\sigma u) - \phi(x) - \epsilon(x) \right| \\
&\qquad \quad +  \frac{n}{N\sigma} \mathbb{E}_{v \sim \mathcal{U}(\mathcal{S}(0,1))}\left[ \left| \phi(x+\sigma v) + \epsilon(x+\sigma v) - \phi(x) - \epsilon(x) \right| \right]\\
&\qquad \le{}  \frac{n}{N\sigma} \left(|\nabla \phi(x)^\intercal \sigma {u}| + \frac{L \|\sigma {u}\|^2}{2} + 2\epsilon_f \right) + \frac{n}{N\sigma} \mathbb{E}_{v \sim \mathcal{U}(\mathcal{S}(0,1))} \left[|\nabla \phi(x)^\intercal \sigma v| + \frac{L \|\sigma v \|^2}{2} + 2\epsilon_f \right] \\
&\qquad \le{}  \frac{n}{N}(2 \|\nabla \phi(x)\| + L \sigma + \frac{4 \epsilon_f}{\sigma}),
\end{align*}
for any $u \sim \mathcal{U}(\mathcal{S}(0,1))$. 
The \emph{matrix variance statistic} of $g(x) - \nabla F(x)$ is 
\begin{align*}
&v(g(x) - \nabla F(x)) \\
&\qquad ={} \max \left\{ \| \mathbb{E}\left[(g(x) - \nabla F(x))(g(x) - \nabla F(x))^\intercal \right] \|, \mathbb{E}\left[(g(x) - \nabla F(x))^\intercal (g(x) - \nabla F(x)) \right] \right\} \\
&\qquad \le{} \max \left\{ \frac{3}{N} \left( \frac{3n}{n+2} \|\nabla\phi(x)\|^2 + \frac{nL^2 \sigma^2}{4} + \frac{4n\epsilon_f^2}{\sigma^2} \right), \frac{3n^2}{N} \left( \frac{\|\nabla \phi(x)\|^2}{n} + \frac{L^2\sigma^2}{4} + \frac{4\epsilon_f^2}{\sigma^2}\right) \right\} \\
&\qquad ={} \frac{3n^2}{N} \left( \frac{\|\nabla \phi(x)\|^2}{n} + \frac{L^2\sigma^2}{4} + \frac{4\epsilon_f^2}{\sigma^2}\right), 
\end{align*}
where the two terms in the maximization are $\|\Var{g(x)}\|$ and trace($\Var{g(x)}$). The uppper bound on these two terms are from Lemma~\ref{lem:BSG var bound}. 
Then by Bernstein's inequality, we have 
\begin{align*}
&\mathbb{P}(\|g(x) - \nabla F(x)\| \ge r) \\
&\qquad \le{} (n+1) \exp\left( \frac{-r^2/2}{v(g(x) - \nabla F(x)) + \frac{nr}{3N} (2 \|\nabla \phi(x)\| + L \sigma + \frac{4\epsilon_f}{\sigma} )} \right) \\
&\qquad \le{} (n+1) \exp\left( \frac{-r^2/2}{\frac{3n^2}{N} \left( \frac{\|\nabla \phi(x)\|^2}{n} + \frac{L^2\sigma^2}{4} + \frac{4\epsilon_f^2}{\sigma^2}\right) + \frac{nr}{3N} (2 \|\nabla \phi(x)\| + L \sigma + \frac{4\epsilon_f}{\sigma})} \right). 
\end{align*}
In order to ensure that $\mathbb{P}(\|g(x) - \nabla F(x)\| \ge r) \le \delta$, for some $\delta \in (0,1)$, we require that 
\begin{align*}
(n+1) \exp\left( \frac{-r^2/2}{\frac{3n^2}{N} \left( \frac{\|\nabla \phi(x)\|^2}{n} + \frac{L^2\sigma^2}{4} + \frac{4\epsilon_f^2}{\sigma^2}\right) + \frac{nr}{3N} (2 \|\nabla \phi(x)\| + L \sigma + \frac{4\epsilon_f}{\sigma})} \right) \le \delta, 
\end{align*}
from which we conclude that
\begin{align*}
N \ge \left[ \frac{6n^2}{r^2} \left( \frac{\|\nabla \phi(x)\|^2}{n} + \frac{L^2\sigma^2}{4} + \frac{4\epsilon_f^2}{\sigma^2} \right)  + \frac{2n}{3r} \left( 2 \|\nabla \phi(x)\| + L \sigma + \frac{4\epsilon_f}{\sigma} \right) \right] \log \frac{n+1}{\delta}.
\end{align*}

For the cBSG case, note that
\begin{align*}
\mathbb{E}_{u \sim \mathcal{U}(\mathcal{S}(0,1))} \left[\frac{n}{N} \frac{f(x+\sigma u) - f(x-\sigma u)}{2\sigma} u - \frac{1}{N} \nabla F(x)\right] &= 0, 
\end{align*}
and 
\begin{align*}
&\left\| \frac{n}{N} \frac{f(x+\sigma {u}) - f(x-\sigma {u})}{2\sigma} {u}  - \frac{1}{N} \nabla F(x) \right\| \\
&\qquad \le{}  \frac{n}{2N\sigma} \left| f(x+\sigma u) - f(x-\sigma {u})\right|  \|u\| + \frac{n}{2N\sigma} \mathbb{E}_{v \sim \mathcal{U}(\mathcal{S}(0,1))}\left[ \left| f(x+\sigma v) - f(x-\sigma v) \right| \|v\| \right]\\
&\qquad ={}  \frac{n}{2N\sigma} \left| \phi(x+\sigma u) + \epsilon(x+\sigma u) - \phi(x-\sigma u) - \epsilon(x-\sigma u) \right| \\
&\qquad \qquad +  \frac{n}{2N\sigma} \mathbb{E}_{v \sim \mathcal{U}(\mathcal{S}(0,1))}\left[ \left| \phi(x+\sigma v) + \epsilon(x+\sigma v) - \phi(x) - \epsilon(x) \right| \right]\\
&\qquad \le{}  \frac{n}{2N\sigma} \left(|2\nabla \phi(x)^\intercal \sigma u| + \frac{M \|\sigma u\|^3}{3}  + 2\epsilon_f \right) + \frac{n}{2N\sigma} \mathbb{E}_{v \sim \mathcal{U}(\mathcal{S}(0,1))} \left[|2\nabla \phi(x)^\intercal \sigma v| + \frac{M \|\sigma v\|^3}{3}  + 2\epsilon_f \right] \\
&\qquad \le{}  \frac{n}{N} \left(2 \|\nabla \phi(x)\| + \frac{M \sigma^2}{3} + \frac{2 \epsilon_f}{\sigma} \right),
\end{align*}
for any $u \sim \mathcal{U}(\mathcal{S}(0,1))$. 
The \emph{matrix variance statistic} of $g(x) - \nabla F(x)$ is 
\begin{align*}
&v(g(x) - \nabla F(x)) \\
&\qquad ={} \max \left\{ \| \mathbb{E}\left[(g(x) - \nabla F(x))(g(x) - \nabla F(x))^\intercal \right] \|, \mathbb{E}\left[(g(x) - \nabla F(x))^\intercal (g(x) - \nabla F(x)) \right] \right\} \\
&\qquad \le{} \max \left\{ \frac{3}{N} \left( \frac{3n}{n+2} \|\nabla\phi(x)\|^2 + \frac{nM^2 \sigma^4}{36} + \frac{n\epsilon_f^2}{\sigma^2} \right), \frac{3n^2}{N} \left( \frac{\|\nabla \phi(x)\|^2}{n} + \frac{M^2\sigma^4}{36} + \frac{\epsilon_f^2}{\sigma^2}\right) \right\} \\
&\qquad ={}  \frac{3n^2}{N} \left( \frac{\|\nabla \phi(x)\|^2}{n} + \frac{M^2\sigma^4}{36} + \frac{\epsilon_f^2}{\sigma^2}\right) . 
\end{align*}
By Bernstein's inequality, we have 
\begin{align*}
&\mathbb{P}(\|g(x) - \nabla F(x)\| \ge r) \\
&\qquad \le{} (n+1) \exp\left( \frac{-r^2/2}{v(g(x) - \nabla F(x)) + \frac{nr}{3N} \left(2 \|\nabla \phi(x)\| + \frac{M \sigma^2}{3} + \frac{2 \epsilon_f}{\sigma} \right)} \right) \\
&\qquad \le{} (n+1) \exp\left( \frac{-r^2/2}{\frac{3n^2}{N} \left( \frac{\|\nabla \phi(x)\|^2}{n} + \frac{M^2\sigma^4}{36} + \frac{\epsilon_f^2}{\sigma^2}\right) + \frac{nr}{3N} \left(2 \|\nabla \phi(x)\| + \frac{M \sigma^2}{3} + \frac{2 \epsilon_f}{\sigma} \right)} \right). 
\end{align*}
In order to ensure that $\mathbb{P}(\|g(x) - \nabla F(x)\| \ge r) \le \delta$, for some $\delta \in (0,1)$, we require that 
\begin{align*}
(n+1) \exp\left( \frac{-r^2/2}{\frac{3n^2}{N} \left( \frac{\|\nabla \phi(x)\|^2}{n} + \frac{M^2\sigma^4}{36} + \frac{\epsilon_f^2}{\sigma^2}\right) + \frac{nr}{3N} \left(2 \|\nabla \phi(x)\| + \frac{M \sigma^2}{3} + \frac{2 \epsilon_f}{\sigma} \right)} \right) \le \delta, 
\end{align*}
from which we conclude that
\begin{align*}
N \ge \left[ \frac{6n^2}{r^2} \left( \frac{\|\nabla \phi(x)\|^2}{n} + \frac{M^2\sigma^4}{36} + \frac{\epsilon_f^2}{\sigma^2}\right)  + \frac{2n}{3r} \left( 2 \|\nabla \phi(x)\| + \frac{M \sigma^2}{3} + \frac{2 \epsilon_f}{\sigma} \right) \right] \log \frac{n+1}{\delta}.
\end{align*}
\end{proof}

Now, with bounds for both terms in \eqref{eq:2terms_sphere}, we can bound $\|g(x) - \nabla \phi(x)\|$, in probability. 

\begin{theorem} \label{thm:prob_bnd_sphere_smoothed_bern} 
Suppose that Assumption  \ref{assum:lip_cont} holds and  $g(x)$ is calculated via \eqref{eq:BSG}. 
If 
\begin{align*}
	N \ge \left[ \frac{6n^2}{r^2} \left( \frac{\|\nabla \phi(x)\|^2}{n} + \frac{L^2\sigma^2}{4} + \frac{4\epsilon_f^2}{\sigma^2} \right)  + \frac{2n}{3r} \left( 2 \|\nabla \phi(x)\| + L \sigma + \frac{4\epsilon_f}{\sigma} \right) \right] \log \frac{n+1}{\delta},
\end{align*}
then, for all $x \in \mathbb{R}^n$ and $r>0$, 
\begin{align}	\label{eq:BSG_thm}
	\|g({x}) - \nabla \phi({x})\| \leq L \sigma + \frac{n \epsilon_f}{\sigma} + r.
\end{align}
with probability at least $1 - \delta$.

Alternatively, suppose that Assumption  \ref{assum:lip_cont_hess} holds and  $g(x)$ is calculated via \eqref{eq:cBSG}. 
If 
\begin{align*}
	N \ge \left[ \frac{6n^2}{r^2} \left( \frac{\|\nabla \phi(x)\|^2}{n} + \frac{M^2\sigma^4}{36} + \frac{\epsilon_f^2}{\sigma^2}\right) + \frac{2n}{3r} \left( 2 \|\nabla \phi(x)\| + \frac{M \sigma^2}{3} + \frac{2 \epsilon_f}{\sigma} \right) \right] \log \frac{n+1}{\delta},
\end{align*}
then, for all $x \in \mathbb{R}^n$ and $r>0$, 
\begin{align}	\label{eq:cBSG_thm}
	\|g({x}) - \nabla \phi({x})\| \leq M \sigma^2 + \frac{n \epsilon_f}{\sigma}+ r.
\end{align}
with probability at least $1 - \delta$.
\end{theorem}

\begin{proof} The proof for the first part \eqref{eq:BSG_thm} is a straightforward combination of the bound in \eqref{eq:BSG_bound1} and the result of the first part of Lemma \ref{lem:prob_bnd_sphere_smoothed_bern}. The proof for the second part \eqref{eq:cBSG_thm} is a straightforward combination of the bound in \eqref{eq:cBSG_bound1} and the result of the second part of Lemma \ref{lem:prob_bnd_sphere_smoothed_bern}.
\end{proof}

In Theorem \ref{thm:prob_bnd_sphere_smoothed_bern} one should notice the improved dependence of the size of the sample set $N$ on the probability $\delta$ as compared to Theorem \ref{thm:prob_bnd_smoothed}. While Bernstein's inequality 
does not apply in the case of the Gaussian smoothed gradient, there may be other ways to establish a better dependence on $\delta$. However, the dependence on $n$ in all cases is linear, which as we have shown for the GSG case is a necessary dependence.  A similar lower bound result for BSG can be derived analogously. 

Using the results of Theorem \ref{thm:prob_bnd_sphere_smoothed_bern}, as before, we derive bounds on $\sigma$ and $N$ that ensure that \eqref{eq:theta_cond} holds with  probability $1-\delta$. To ensure \eqref{eq:theta_cond}, with probability $1-\delta$, using  Theorem \ref{thm:prob_bnd_sphere_smoothed_bern} we want the following to hold
\begin{align}
	  L \sigma +  \frac{n \epsilon_f}{\sigma} &\leq \lambda \theta \|\nabla \phi({x})\| \label{eq:ball1},\\
	 r &\leq  (1-\lambda) \theta \|\nabla \phi({x})\|, \label{eq:ball2}
\end{align}
for some $\lambda\in (0,1)$. 

Let us first consider $g(x)$ calculated via \eqref{eq:BSG}. As before, to ensure that \eqref{eq:ball1} holds, we impose the following conditions: 
 \begin{align*}
	\sigma = \sqrt{ \frac{n \epsilon_f}{L}} \quad \text{and} \quad \|\nabla \phi(x)\|\geq \frac{2\sqrt{nL\epsilon_f}}{\lambda\theta}.
 \end{align*}
Now using these bounds and substituting $r = (1-\lambda) \theta \|\nabla \phi({x})\|$ into the first bound on $N$ in Theorem \ref{thm:prob_bnd_sphere_smoothed_bern} we have
\begin{align*} 
	&\left[ \frac{6n^2}{r^2} \left( \frac{\|\nabla \phi(x)\|^2}{n} + \frac{L^2\sigma^2}{4} + \frac{4\epsilon_f^2}{\sigma^2} \right)  + \frac{2n}{3r} \left( 2 \|\nabla \phi(x)\| + L \sigma + \frac{4\epsilon_f}{\sigma} \right) \right] \log \frac{n+1}{\delta}\\
	& \qquad \leq \left[ \frac{6n}{\theta^2}\frac{1}{(1-\lambda)^2} + \left( \frac{3n^2}{8} + 6\right)  \frac{\lambda^2}{(1-\lambda)^2} + \frac{4n}{3 \theta}\frac{1}{1-\lambda} + \left(\frac{n}{3} + \frac{4}{3}\right)\frac{\lambda}{ 1-\lambda}\right] \log \frac{n+1}{\delta} \nonumber.
\end{align*}

As before, we are interested in making the lower bound on $N$ to scale at most linearly with $n$.  Thus, to achieve this and to simplify the expression we choose $\lambda=\frac{1}{2\sqrt{n}}$ so that $\frac{ \lambda^2}{ (1-\lambda)^2}\leq \frac{1}{n}$, for all $n$. Then, using that  $\frac{1}{(1-\lambda)^2}\leq \frac{n}{(\sqrt{n}-1)^2}$ the above expression is bounded by
\begin{align}\label{eq:BSG_nbound_simple}
 \left[\frac{6n }{\theta^2 }\frac{n}{(\sqrt{n}-1)^2} + \frac{3n}{8}  + \frac{6}{n} + \frac{4n}{3 \theta}\frac{\sqrt{n}}{\sqrt{n}-1} + \frac{\sqrt{n}}{3} + \frac{4}{3\sqrt{n}} \right] \log \frac{n+1}{\delta}. 
\end{align}

This implies that by choosing $N$ at least as large as the value of \eqref{eq:BSG_nbound_simple}
we ensure that \eqref{eq:ball2} holds.

We now summarize  the result  for the gradient approximation computed via \eqref{eq:BSG}, for $\lambda =  \frac{1}{2\sqrt{n}}$.
\begin{corollary} \label{corr:BSG} 
Suppose that Assumption \ref{assum:lip_cont}  holds, $n>1$
and $g(x)$ is computed via  \eqref{eq:BSG}  with $N$ and $\sigma$ satisfying,
\begin{gather*}
N \geq \left[\frac{6n }{\theta^2 }\frac{n}{(\sqrt{n}-1)^2} + \frac{3n}{8}  + \frac{6}{n} + \frac{4n}{3 \theta}\frac{\sqrt{n}}{(\sqrt{n}-1)} + \frac{\sqrt{n}}{3} + \frac{4}{3\sqrt{n}} \right] \log \frac{n+1}{\delta} \quad \text{and} \quad
	\sigma = \sqrt{ \frac{n \epsilon_f}{L}}.
	\end{gather*}
If $\|\nabla \phi(x)\|\geq \frac{4n\sqrt{L\epsilon_f}}{\theta}$, then 
 \eqref{eq:theta_cond} holds with  probability $1-\delta$. 
\end{corollary}


We now derive the analogous bounds on $N$ and $\sigma$ for the case when $g(x)$ is calculated via \eqref{eq:cBSG}. To ensure \eqref{eq:theta_cond}, with probability $1-\delta$, using  Theorem \ref{thm:prob_bnd_sphere_smoothed_bern} we want the following to hold
\begin{align}
	 M \sigma^2 +  \frac{n \epsilon_f}{\sigma} &\leq \lambda \theta \|\nabla \phi({x})\| \label{eq:sphere3},\\
	 r &\leq  (1-\lambda) \theta \|\nabla \phi({x})\|, \label{eq:sphere4}
\end{align}
for some $\lambda\in (0,1)$. In order to ensure that \eqref{eq:sphere3} holds, we use the same logic as was done for Central Finite Differences in Section \ref{sec:fd}. Namely, we require the following:
 \begin{align*}
	\sigma =\sqrt[3]{ \frac{n \epsilon_f}{2M}} \quad \text{and} \quad  \|\nabla \phi(x)\|\geq \frac{3}{\lambda \theta} \sqrt[3]{\frac{n^2 M \epsilon_f^2}{4}}.
 \end{align*}
Now using these bounds and setting $r = (1-\lambda) \theta \|\nabla \phi({x})\|$ into the second bound on $N$ in Theorem \ref{thm:prob_bnd_sphere_smoothed_bern} we have
\begin{align*}
	 & \left[ \frac{6n^2}{r^2} \left( \frac{\|\nabla \phi(x)\|^2}{n} + \frac{M^2\sigma^4}{36} + \frac{\epsilon_f^2}{\sigma^2}\right) + \frac{2n}{3r} \left( 2 \|\nabla \phi(x)\| + \frac{M \sigma^2}{3} + \frac{2 \epsilon_f}{\sigma} \right) \right] \log \frac{n+1}{\delta}\\
& \qquad \leq \left[ \frac{6n}{\theta^2}\frac{1}{(1-\lambda)^2} + \left( \frac{n^2}{24} + \frac{3}{2}\right)  \frac{\lambda^2}{(1-\lambda)^2} + \frac{4n}{3 \theta}\frac{1}{1-\lambda} + \left(\frac{n}{9} + \frac{2}{3} \right)\frac{\lambda}{ 1-\lambda}\right] \log \frac{n+1}{\delta}.
\end{align*}
As before, we are interested in making the lower bound on $N$ to scale at most linearly with $n$.  Thus, to achieve this and to simplify the expression we choose $\lambda$ such that $\frac{ \lambda^2}{ (1-\lambda)^2}\leq \frac{1}{n}$, which,  implies that $\lambda\leq\frac{1}{\sqrt{n}} $. Then, using again the fact that  $\frac{1}{(1-\lambda)^2}\leq \frac{n}{(\sqrt{n}-1)^2}$ and $\frac{1}{n}\leq 1$  the above expression is bounded by
\begin{align*}
 \left[\frac{6n }{\theta^2 }\frac{n}{(\sqrt{n}-1)^2} + \frac{n}{24}  + \frac{3}{2n} + \frac{4n}{3 \theta}\frac{\sqrt{n}}{\sqrt{n}-1} + \frac{\sqrt{n}}{9} + \frac{2}{3\sqrt{n}} \right] \log \frac{n+1}{\delta}. 
	 \end{align*}

We now summarize the result  for the gradient approximation computed via \eqref{eq:cBSG}, using the fact that $\lambda =  \frac{1}{2\sqrt{n}}$.
\begin{corollary} \label{corr:cBSG} 
Suppose that Assumption \ref{assum:lip_cont_hess} holds, $n>1$ and $g(x)$ is computed via  \eqref{eq:cBSG}  with $N$ and $\sigma$ satisfying,
\begin{gather*}
N \geq  \left[\frac{6n }{\theta^2 }\frac{n}{(\sqrt{n}-1)^2} + \frac{n}{24}  + \frac{3}{2n} + \frac{4n}{3 \theta}\frac{\sqrt{n}}{\sqrt{n}-1} + \frac{\sqrt{n}}{9} + \frac{2}{3\sqrt{n}} \right] \log \frac{n+1}{\delta} \quad \text{and} \quad
\sigma=	\sqrt[3]{ \frac{n \epsilon_f}{2M}}.
\end{gather*}
If  $\|\nabla \phi(x)\|\geq \frac{6}{\theta} \sqrt[3]{\frac{n^{7/2} M \epsilon_f^2}{4}}$, then \eqref{eq:theta_cond} holds with  probability $1-\delta$. 
\end{corollary}

\subsection{Smoothing vs. Interpolation gradients}
We now want to give some quick intuition explaining why GSG and BSG method do not provide as high accuracy as linear interpolation. Let us consider the two method of estimating gradients based on the same sample set. In particular, to compare GSG with linear interpolation, we choose the sample set  $\mathcal{X}=\{x+\sigma u_1, x+\sigma u_2, \ldots, x+\sigma u_n\}$ for some $\sigma >0$ with $u$ obeying the standard Gaussian distribution. 
 Recall the definition of the matrix $Q_\mathcal{X}$ and the vector $F_\mathcal{X}$ (see Section \ref{sec:lin_mod}), and the fact that
 the gradient estimate computed by linear interpolation satisfies 
  \begin{align*} 
Q_\mathcal{X}	g_{LI}=F_\mathcal{X}/\sigma. 
\end{align*}
The GSG estimate, on the other hand is written as,
  \begin{align*} 
	g_{GSG}= \frac{1}{n}Q_\mathcal{X}^TF_\mathcal{X}/\sigma=\frac{1}{n}Q_\mathcal{X}^T Q_\mathcal{X}g_{LI}.
\end{align*} 
Hence, we obtain 
  \begin{align*} 
\|g_{LI}-g_{GSG}\|= \left\|\left(I-\frac{1}{n}Q_\mathcal{X}^TQ_\mathcal{X}\right)g_{LI}\right\|. 
\end{align*} 
We know that, when $\epsilon(x) = 0$ for all $x \in \mathbb{R}^n$, the difference $\|g_{LI}-\nabla \phi(x)\|$ goes to zero as $\sigma\to 0$. However, $\|(I-\frac{1}{n}Q_\mathcal{X}^TQ_\mathcal{X})g_{LI}\|$ does not, as it does not depend on $\sigma$. While we have $\mathbb{E}[\frac{1}{n}Q_\mathcal{X}^TQ_\mathcal{X}]=I$, nevertheless, with non-negligible probability, the matrix 
$\|(I-\frac{1}{n}Q_\mathcal{X}^TQ_\mathcal{X})g_{LI}\|\geq \nu \|g_{LI}\|$ for some fixed non-negligible value of $\lambda$, for example, $\nu>1/2$. 

The intuition for the BSG can be derived in the same manner.

\subsection{Summary of Results}\label{sec:summary}
In this section, we summarize the results for all methods. Specifically, 
Table \ref{tbl:bounds_full2} summarizes the conditions on $N$, $\sigma$ and $\nabla \phi(x)$ for each method that we consider in this paper to guarantee condition \eqref{eq:theta_cond}. Note that for the smoothing methods the bounds hold with probability $1-\delta$ and the number of samples depends on $\delta$. From the table, it is clear that for large $n$ ($\frac{n}{(\sqrt{n}-1)^2}$ goes to $1$ as $n \rightarrow \infty$), all methods have the same  dependence (order of magnitude) on the dimension $n$; however, for the smoothing methods the constants in the bound can be significantly larger than those for deterministic methods, such as finite differences. This suggests that deterministic methods may be more efficient, at least in the setting considered in this paper, when accurate gradient estimates are  desired. The bounds on the sampling radius are comparable for the smoothing and deterministic methods

\setlength{\tabcolsep}{4pt}

\begin{table}[h!]
\caption{ Bounds on $N$, $\sigma$ and $\|\nabla \phi(x)\|$ that ensure $\| g(x) - \nabla \phi(x) \| \leq \theta\|\nabla \phi(x)\|$ ($\pmb{^*}$ denotes result is with probability $1-\delta$). }
\small
\label{tbl:bounds_full2}
\centering
\begin{tabular}{lccc}
\toprule
\begin{tabular}[l]{@{}l@{}}\textbf{Gradient} \\  \textbf{Approximation}\end{tabular} &
 \textbf{$\pmb{N}$} &
 \textbf{$\pmb{\sigma}$} &  \textbf{$\pmb{\| \nabla \phi(x) \|}$} \\  \midrule

\begin{tabular}[l]{@{}l@{}}\textbf{Forward Finite} \\  \textbf{Differences}\end{tabular} & {\scriptsize$n$} & {\scriptsize$2 \sqrt{\frac{\epsilon_f}{L}}$} & {\scriptsize$\frac{2\sqrt{nL\epsilon_f}}{\theta}$} \\ \hdashline
 \begin{tabular}[l]{@{}l@{}}\textbf{Central Finite} \\  \textbf{Differences}\end{tabular} & {\scriptsize$n$} & {\scriptsize$ \sqrt[3]{\frac{6\epsilon_f}{M}}$} & {\scriptsize$\frac{ \sqrt[3]{9}\sqrt[3]{ n^{3/2}M \epsilon_f^2}}{2 \theta}$}  \\ \hdashline
 \begin{tabular}[l]{@{}l@{}}\textbf{Linear} \\  \textbf{Interpolation}\end{tabular} & {\scriptsize$n$} & {\scriptsize$2 \sqrt{\frac{\epsilon_f}{L}}$} & {\scriptsize$\frac{2\| Q_\mathcal{X}^{-1}\|\sqrt{nL\epsilon_f}}{\theta}$} \\ \hdashline
 \begin{tabular}[l]{@{}l@{}}\textbf{Gaussian Smoothed} \\  \textbf{Gradients}$\pmb{^*}$\end{tabular} & {\scriptsize$\frac{9n}{\delta  \theta^2} \frac{n}{(\sqrt{n}-1)^2} + \frac{3(n+4)}{16\delta}  + \frac{3}{n\delta}$} & {\scriptsize$ \sqrt{\frac{\epsilon_f}{L}}$} & {\scriptsize$\frac{6n\sqrt{L\epsilon_f}}{\theta}$} \\ \hdashline
 \begin{tabular}[l]{@{}l@{}}\textbf{Centered Gaussian} \\  \textbf{Smoothed Gradients}$\pmb{^*}$\end{tabular}   & {\scriptsize$\frac{9n}{\delta  \theta^2} \frac{n}{(\sqrt{n}-1)^2} + \frac{n+6}{48\delta } + \frac{3}{4n\delta}$} & {\scriptsize$ \sqrt[3]{\frac{\epsilon_f}{\sqrt{n}M}}$} & {\scriptsize$\frac{18\sqrt[3]{n^{7/2}M\epsilon_f^2}}{\sqrt[3]{4}\theta}$} \\ \hdashline
  \begin{tabular}[l]{@{}l@{}}\textbf{Sphere Smoothed} \\  \textbf{Gradients}$\pmb{^*}$\end{tabular} & {\scriptsize$ \left[ \left(\frac{6n }{\theta^2 }\frac{\sqrt{n}}{(\sqrt{n}-1)} + \frac{4n}{3 \theta} \right)\frac{\sqrt{n}}{(\sqrt{n}-1)} +  \frac{3n}{8} + \frac{6}{n} + \frac{\sqrt{n}}{3} + \frac{4}{3\sqrt{n}} \right] \log \frac{n+1}{\delta}$} & {\scriptsize$ \sqrt{\frac{n\epsilon_f}{L}}$} & {\scriptsize$\frac{4n\sqrt{L\epsilon_f}}{\theta}$} \\ \hdashline 
 \begin{tabular}[l]{@{}l@{}}\textbf{Centered Sphere} \\  \textbf{Smoothed Gradients}$\pmb{^*}$\end{tabular} & {\scriptsize$ \left[ \left(\frac{6n }{\theta^2 }\frac{\sqrt{n}}{(\sqrt{n}-1)} + \frac{4n}{3 \theta} \right)\frac{\sqrt{n}}{(\sqrt{n}-1)} +  \frac{n}{24} + \frac{3}{2n} + \frac{\sqrt{n}}{9} + \frac{2}{3\sqrt{n}} \right] \log \frac{n+1}{\delta}$} & {\scriptsize$ \sqrt[3]{\frac{n\epsilon_f}{M}}$} & {\scriptsize$\frac{6\sqrt[3]{n^{7/2}M\epsilon_f^2}}{\sqrt[3]{4}\theta}$} \\

 \bottomrule
\end{tabular}
 \end{table}

\section{Numerical Results}\label{sec:num}

In this section, we test our theoretical conclusions via numerical experiments. First, we present numerical results evaluating the quality of gradient approximations constructed via finite differences, linear interpolation, Gaussian smoothing and smoothing on a unit sphere (Section \ref{sec:num_grad_approx}). We then illustrate the performance of a line search derivative-free optimization algorithm that employs the aforementioned gradient approximations  on standard DFO benchmarking problems as well as on Reinforcement Learning tasks (Section \ref{sec:alg_perf}).


\subsection{Gradient Approximation Accuracy} \label{sec:num_grad_approx}
We compare  the numerical accuracy of the gradient approximations obtained by the methods discussed in Section~\ref{sec:grad_approx}. 
We compare the resulting $\theta$, which   is the relative error, 
\begin{align} \label{eq:relative_error}
	\frac{\|g(x) - \nabla \phi(x) \|}{\|\nabla \phi(x)\|},
\end{align}
and report the average log of the relative error, i.e., $\log_{10}{\theta}$. Theory dictates that an optimization algorithm will converge if $\log_{10}{\theta} < \log_{10}{1/2} \approx -0.301$, namely $\theta < 1/2$, with sufficiently high probability; see \cite{berahas2019global}. 

\paragraph{Gradient estimation on a synthetic function}
We first conduct tests on a synthetic function, 
\begin{align}		\label{eq:sincos}
	\phi(x) = \left( \sum_{i=1}^{n/2} M \sin (x_{2i-1}) + \cos (x_{2i}) \right) + \frac{L-M}{2n} x^\intercal 1_{n\times n} x, 
\end{align}
where  $n$ is an even number denoting the input dimension, $1_{n\times n}$ denotes an $n$ by $n$ matrix of all ones, and $L>M>0$. We approximate the gradient of $\phi$ at the origin, for which $\|\nabla \phi(0)\| = \sqrt{\frac{n}{2}} M$. The Lipschitz constants for the first and second derivatives are  $L$ and $\max\{M,1\}$, respectively. The function given in \eqref{eq:sincos} allows us to vary all the moving components in the gradient approximations, namely, the dimension $n$, the Lipschitz constants $L$ and $M$ of the gradients and Hessians, respectively, the sampling radius $\sigma$ and the size of the sample set $N$, in order to evaluate the different gradient approximation methods. We show results for two regimes: $(1)$ the noise-free regime where $f(x) = \phi(x)$ (Figure \ref{fig:fig_grad_approx_sin}, left column); and, $(2)$ the noisy regime where $f(x) = \phi(x) + \epsilon(x)$ and $\epsilon(x) \sim U([-\epsilon_f, \epsilon_f])$ with $\epsilon_f = 0.0001$ (Figure \ref{fig:fig_grad_approx_sin}, right column).   

We illustrate the relative approximation errors of the different methods 
using two sets (noise-free and noisy) of $5$ box plots (Figure \ref{fig:fig_grad_approx_sin}). The default values of the parameters are: $n=20$, $M=1$, $L=2$, $\sigma = 0.01$, and $N= 4n$ (for the smoothing methods). For each box plot, we vary one of the parameters.  Since the acutal sampling radius for Gaussian smoothing methods is not $\sigma$ but $\sigma \mathbb{E}_{u \sim \mathcal{N}(0,I)}$, the $\sigma$ used for these methods was $\sigma$ divided by $\mathbb{E}_{u \sim \mathcal{N}(0,I)}$. Note, when comparing the relative errors for different values of $M$, the constant $L$ is was set to $M+1$. For all randomized methods, including linear interpolation, $\nabla \phi(0)$ is estimated 100 times, i.e., we compute 100 realizations of $g(0)$. For linear interpolation, the directions $\{u_i\}_{i=1}^n$ are chosen as $u_i \sim \mathcal{N}(0,I)$ for all $i = 1,2,\dots, n$, and then normalized so that they lie in a unit ball $u_i \gets u_i / \max_{j\in\{1,\dots,n\}} \|u_j\|$. Moreover, all experiments in the noisy regime were conducted 100 times. Finally, in each of the plots in Figure \ref{fig:fig_grad_approx_sin} one parameter was varied and all the rest were set to their default values.

\begin{figure}[h!]
\begin{subfigure}{.5\textwidth}
  \centering
  \includegraphics[trim=51 0 60 10,clip,width=0.95\linewidth]{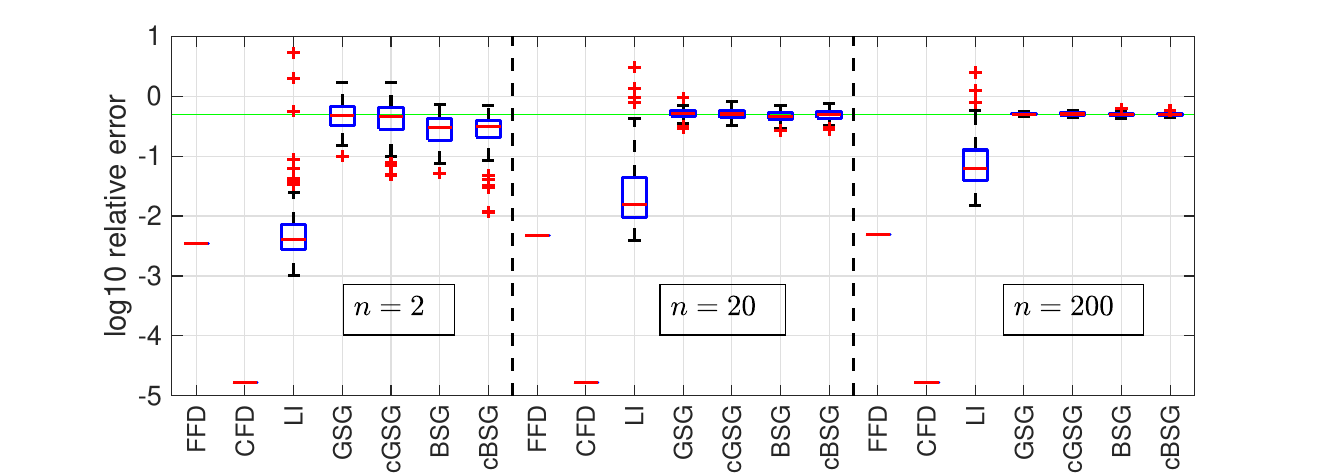}
  \caption{Different $n$ ($n \in \{2,20,200\}$).}
  \label{fig:sfig1}
\end{subfigure}%
\begin{subfigure}{.5\textwidth}
  \centering
  \includegraphics[trim=51 0 60 10,clip,width=0.95\linewidth]{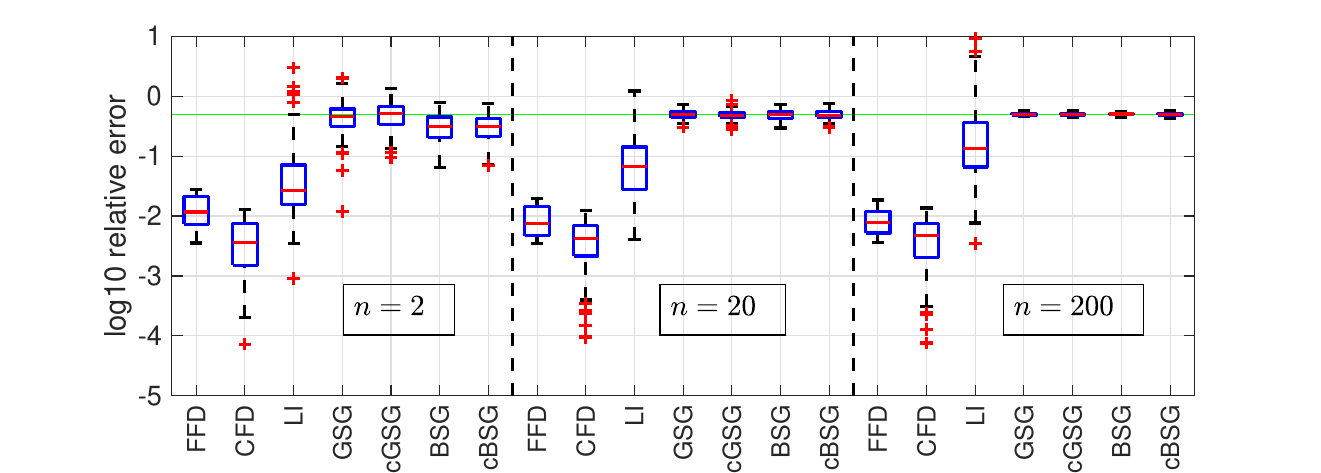}
  \caption{Different $n$ ($n \in \{2,20,200\}$).}
  \label{fig:sfig1e}
\end{subfigure}

\begin{subfigure}{.5\textwidth}
  \centering
  \includegraphics[trim=51 0 60 10,clip,width=0.95\linewidth]{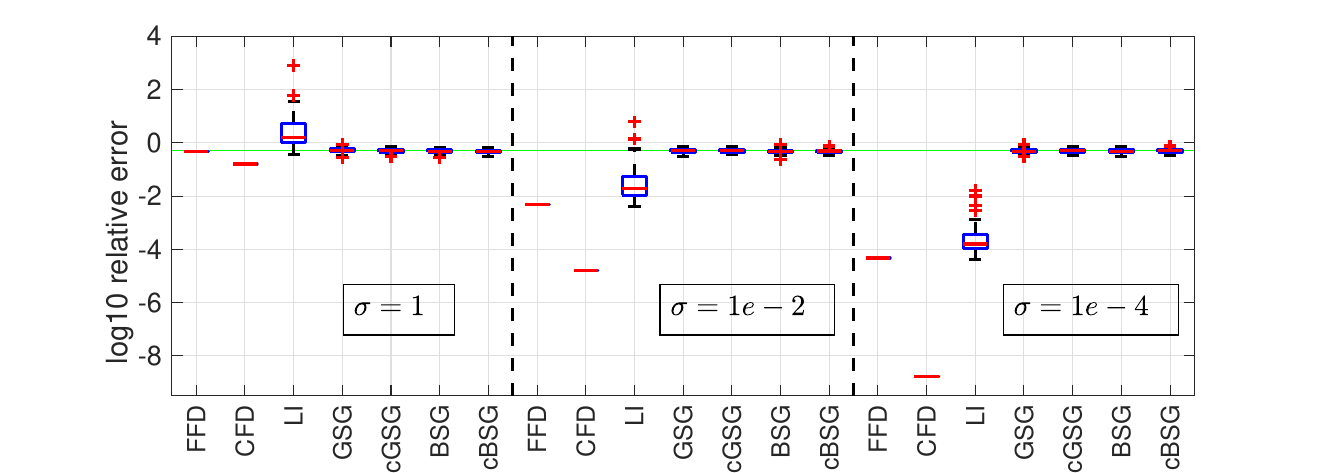}
  \caption{Different $\sigma$ ($\sigma \in \{10^0,10^{-3},10^{-6}\}$).}
  \label{fig:sfig2}
\end{subfigure} %
\begin{subfigure}{.5\textwidth}
  \centering
  \includegraphics[trim=51 0 60 10,clip,width=0.95\linewidth]{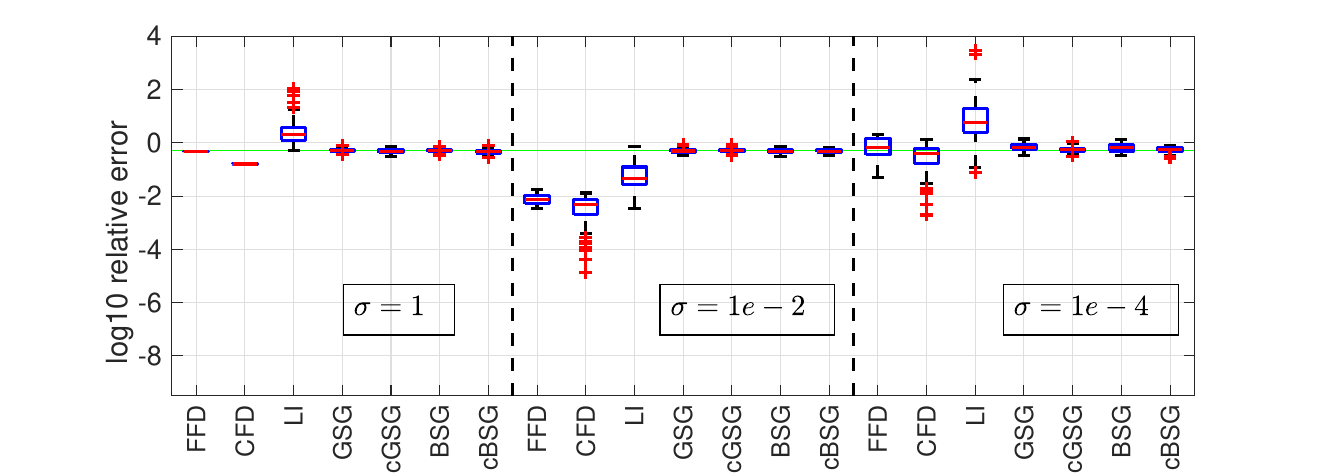}
  \caption{Different $\sigma$ ($\sigma \in \{10^0,10^{-2},10^{-4}\}$).}
  \label{fig:sfig2e}
\end{subfigure}

\begin{subfigure}{.5\textwidth}
  \centering
  \includegraphics[trim=51 0 60 10,clip,width=0.95\linewidth]{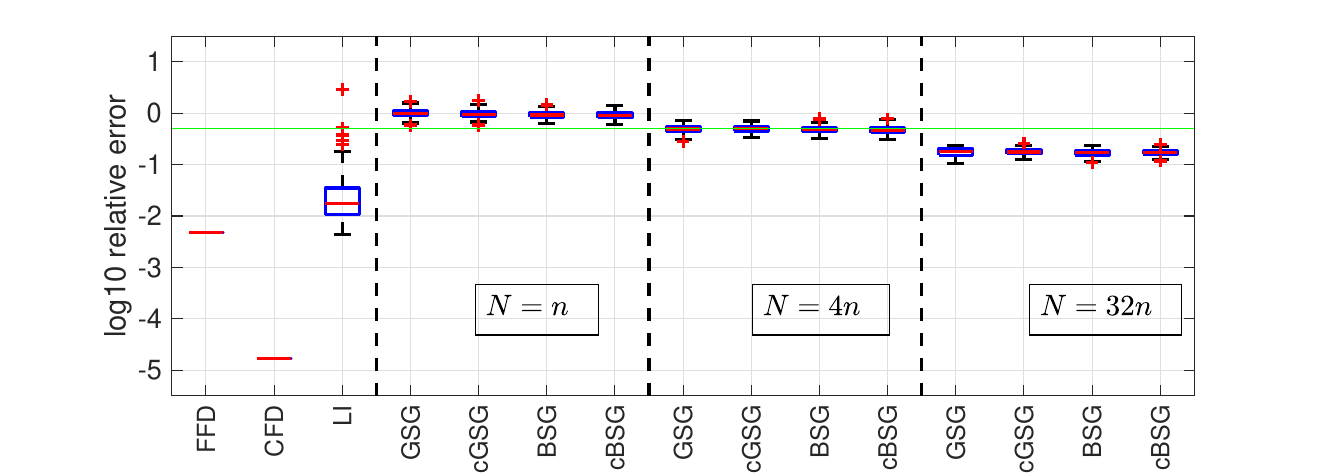}
  \caption{Different $N$ ($N \in \{n,4n,32n\}$, smoothing methods).}
  \label{fig:sfig3}
\end{subfigure}%
\begin{subfigure}{.5\textwidth}
  \centering
  \includegraphics[trim=51 0 60 10,clip,width=0.95\linewidth]{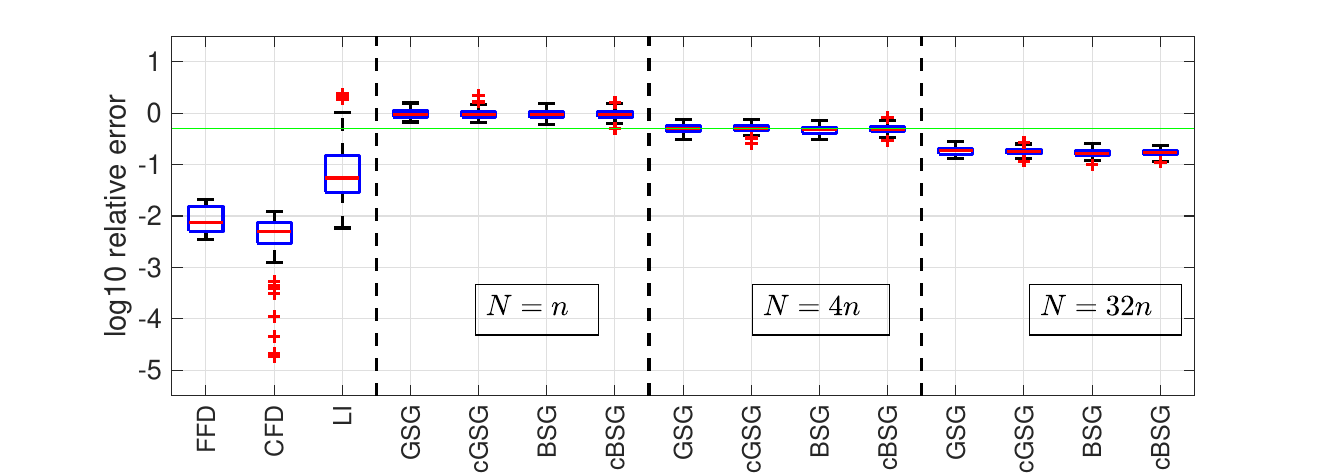}
  \caption{Different $N$ ($N \in \{n,4n,32n\}$, smoothing methods).}
  \label{fig:sfig3e}
\end{subfigure}

\begin{subfigure}{.5\textwidth}
  \centering
  \includegraphics[trim=51 0 60 10,clip,width=0.95\linewidth]{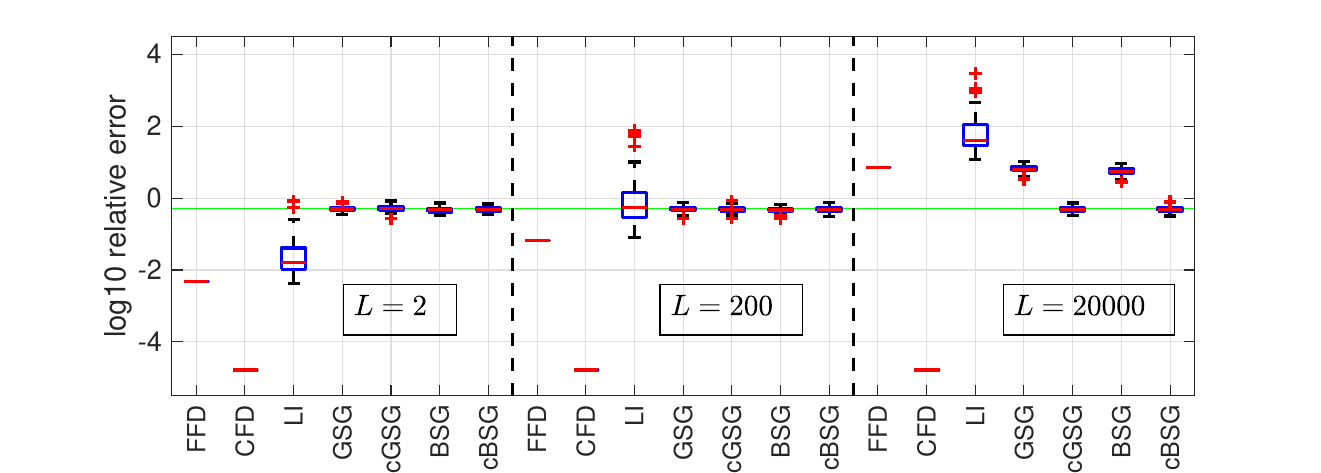}
  \caption{Different $L$ ($L \in \{2,200,20000\}$).}
  \label{fig:sfig4}
\end{subfigure}%
\begin{subfigure}{.5\textwidth}
  \centering
  \includegraphics[trim=51 0 60 10,clip,width=0.95\linewidth]{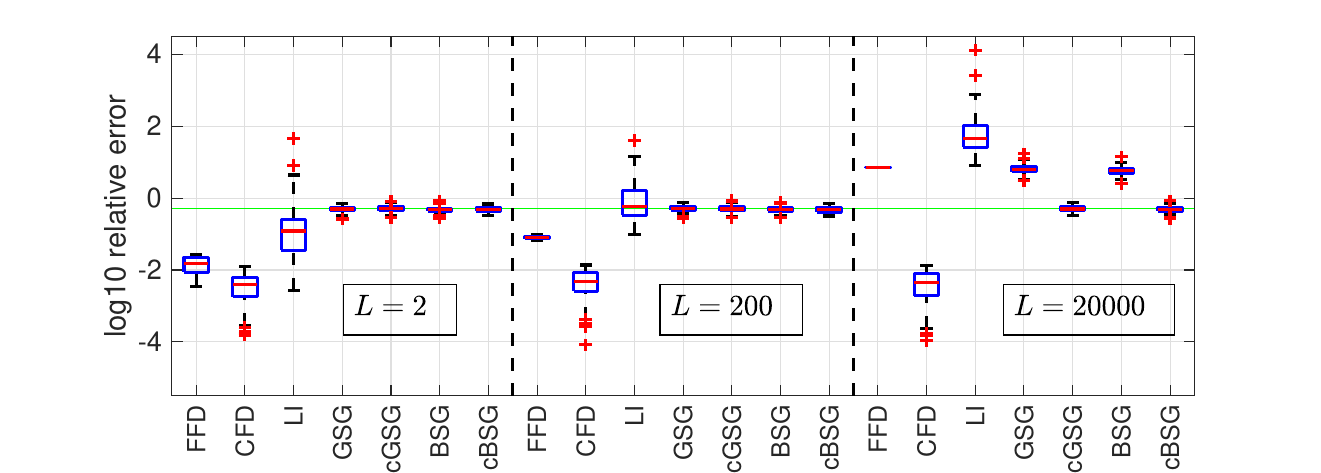}
  \caption{Different $L$ ($L \in \{2,200,20000\}$).}
  \label{fig:sfig4e}
\end{subfigure}

\begin{subfigure}{.5\textwidth}
  \centering
  \includegraphics[trim=51 0 60 10,clip,width=0.95\linewidth]{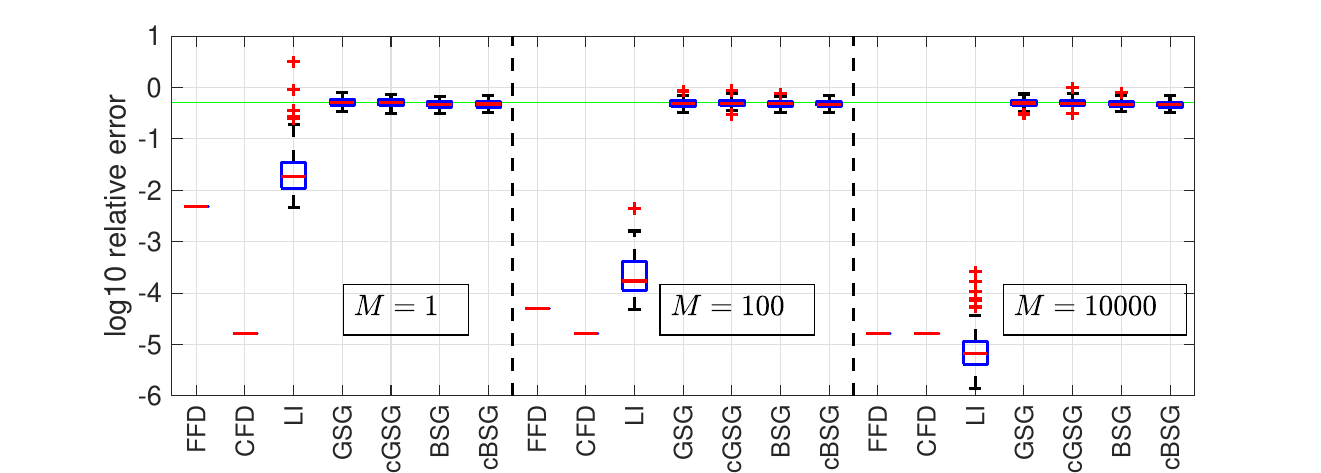}
  \caption{Different $M$ ($M \in \{1,100,10000\}$). Note, $L=M+1$.}
  \label{fig:sfig5}
\end{subfigure}%
\begin{subfigure}{.5\textwidth}
  \centering
  \includegraphics[trim=51 0 60 10,clip,width=0.95\linewidth]{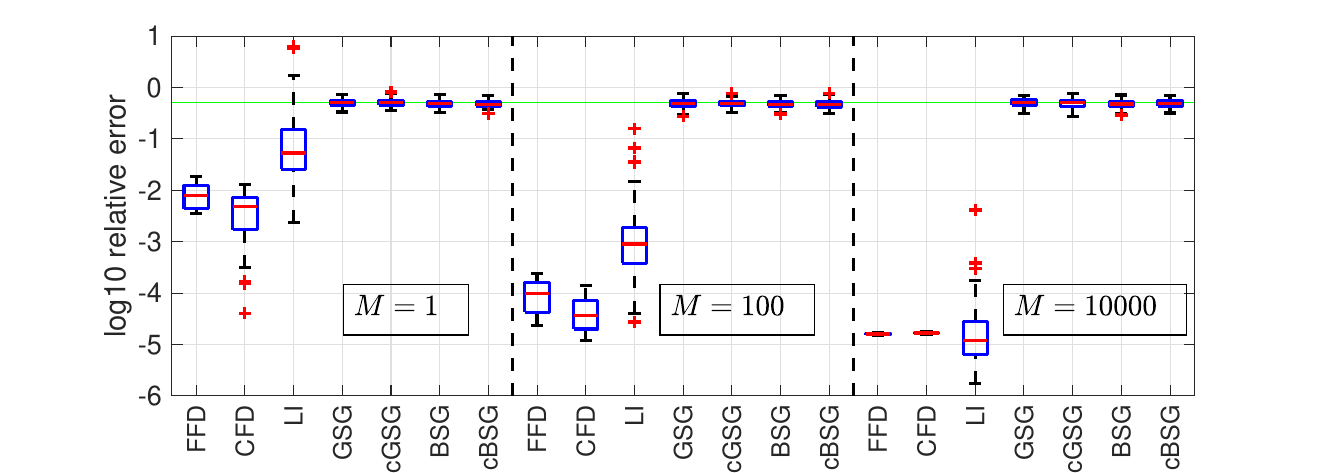}
  \caption{different $M$ ($M \in \{1,100,10000\}$). Note, $L=M+1$.}
  \label{fig:sfig5e}
\end{subfigure}
\caption{Log of relative error \eqref{eq:relative_error} of gradient approximations (FFD, CFD, LI, GSG, cGSG, BSG, cBSG) with different $n$, $\sigma$, $N$, $L$ and $M$. Left column: noise-free ($\epsilon_f=0$); Right column: noisy (iid noise $U(-\epsilon_f, \epsilon_f)$ for each point and $\epsilon_f=0.0001$). \label{fig:fig_grad_approx_sin}}
\end{figure}

In accordance with our theory, we see in Figure \ref{fig:sfig1} that the relative approximation errors of most methods are not affected by the dimension $n$ as long as the sampling radius and the number of sample points is chosen appropriately.
The only method that is affected is interpolation; this is because as the dimension increases the matrix $Q_{\cal{X}}$ formed by the sampling directions (chosen randomly) may become more ill-conditioned. The effect of the dimension $n$ becomes more apparent in the noisy regime; see Figure \ref{fig:sfig1e}.  In Figure \ref{fig:sfig2}, we observe that the size of $\sigma$, the sampling radius, has a significant effect on the deterministic methods (FFD and CFD) and LI. As predicted by the theory, in the noise-free setting, the gradient approximations improve as the sampling radius is reduced. For the randomized methods, GSG, cGSG, BSG and cBSG, in the noise-free setting, it appears that the sampling radius has no effect on the approximation quality. This is not surprising as our theory indicates that one of the terms in the error bound does not diminish with $\sigma$; see \ref{fig:sfig2}. We should note that the randomized approximations are significantly worse than the approximations constructed by the deterministic methods in the noise-free regime. In the noisy regime, diminishing the sampling radius does not necessarily improve the approximations; see Figure \ref{fig:sfig2e}. This is predicted by the theory, as the error bounds have two terms, one that is diminishing with $\sigma$ and one that is increasing with $\sigma$. In Figures \ref{fig:sfig3} and \ref{fig:sfig3e}, we see that having more samples improves the accuracy achieved by GSG, cGSG, BSG and cBSG, in both the noise-free and noisy regimes. Finally, in Figures \ref{fig:sfig4}, \ref{fig:sfig4e}, \ref{fig:sfig5} and \ref{fig:sfig5e}, we see how the approximations are affected by changes in the Lipschitz constants. For example, the FFD, GSG and BSG approximations are affected by changes in $L$, whereas, the CFD cGSG and cBSG approximations are immune to these changes, but are affected by changes in $M$. All these effects are predicted by the theory. Note, in our experiments the FFD, GSG and BSG approximations are sensitive to changes in $M$, this is due to the fact that the constant $L$ is linked to $M$ ($L = M+1$).

\begin{figure}[h!]
\centering
\begin{subfigure}{\textwidth}
    \centering
    \includegraphics[trim=51 0 60 10,clip,width=\linewidth]{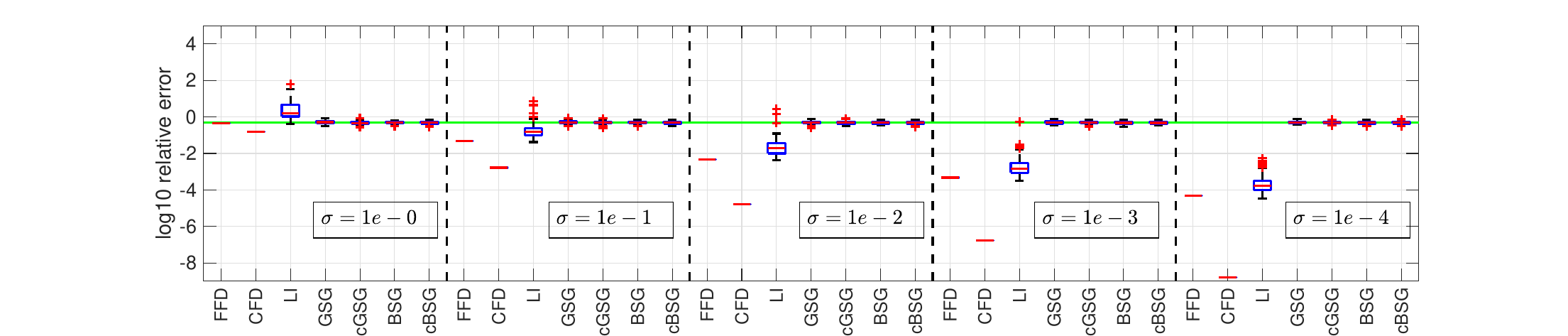}
    \caption{$\epsilon_f = 0$ with $\sigma \in \{10^0,10^{-1},10^{-2},10^{-3},10^{-4}\}$ \label{fig:box1}}
\end{subfigure}

\begin{subfigure}{\textwidth}
    \centering
    \includegraphics[trim=51 0 60 10,clip,width=\linewidth]{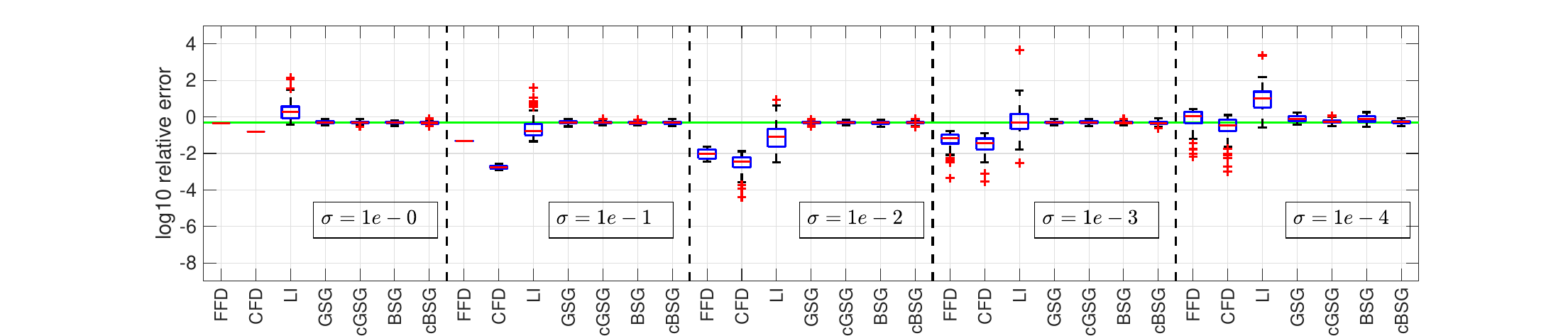}
    \caption{$\epsilon_f = 10^{-4}$ with $\sigma \in \{10^0,10^{-1},10^{-2},10^{-3},10^{-4}\}$ \label{fig:box2}}
\end{subfigure}

\begin{subfigure}{\textwidth}
    \centering
    \includegraphics[trim=51 0 60 10,clip,width=\linewidth]{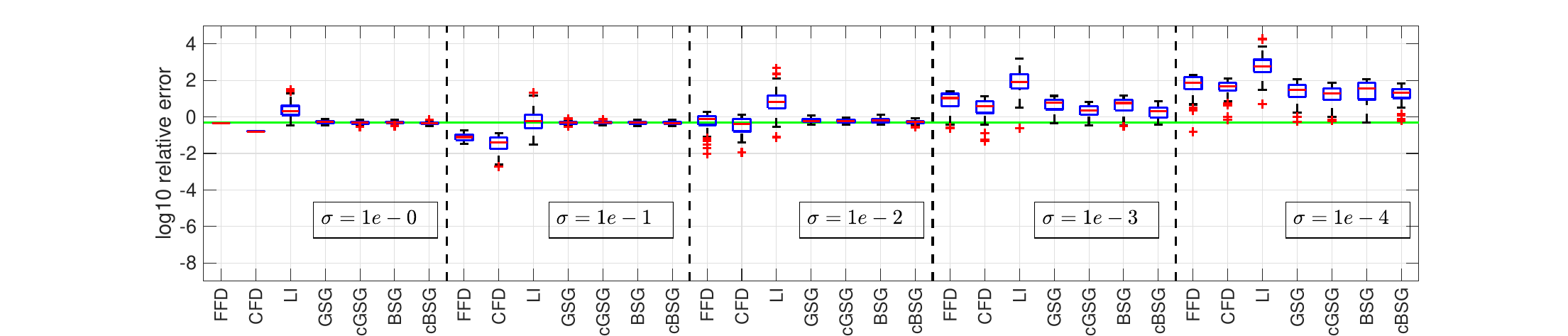}
    \caption{$\epsilon_f = 10^{-2}$ with $\sigma \in \{10^0,10^{-1},10^{-2},10^{-3},10^{-4}\}$ \label{fig:box3}}
\end{subfigure}
\caption{Log of relative error \eqref{eq:relative_error} of gradient approximations (FFD, CFD, LI, GSG, cGSG, BSG, cBSG) with different $\sigma$. Top row: $\epsilon_f=0$; Middle row: $\epsilon_f=10^{-4}$; Bottom row: $\epsilon_f=10^{-2}$. \label{fig:fig_grad_approx_sin2}}
\end{figure}

In order to further illustrate the effects of noise $\epsilon_f$ and sampling radius $\sigma$, we ran experiments on the function given in \eqref{eq:sincos} and varied these two parameters; see Figure \ref{fig:fig_grad_approx_sin2}. Each row illustrates results for a different noise level $\epsilon_f \in \{ 0, 10^{-4}, 10^{-2}\}$ for different sampling radii $\sigma \in \{ 10^0, 10^{-1},10^{-2},10^{-3},10^{-4}\}$. In the absence of noise (Figure \ref{fig:box1}), as the sampling radius is reduced the approximations get better. As predicted by the theory, this is not the case in the presence of noise (Figures \ref{fig:box2} and \ref{fig:box3}).

\paragraph{Gradient estimation on Schittkowski functions \cite{schittkowski2012more}}
Next, we test the different gradient approximations on the $69$ functions from the Schittkowski test set \cite{schittkowski2012more}. The methods we compare are the same as in the case of the synthetic function.  
We computed the gradient approximations for a variety of points with diverse values for $\nabla \phi(x_k)$ and local Lipschitz constants $L$. For each problem we generated points by running gradient descent with a fixed step size for either $100$ iterations or until the norm of the true gradient reached a value of $10^{-2}$. Since for several problems the algorithm terminated in less than 100 iterations, the actual number of points we obtained was $5330$.

Tables~\ref{tab:num_error3} and \ref{tab:num_error4} summarize the results of these experiments for the noise-free and noisy ($\epsilon_f = 10^{-4}$) regimes, respectively. We show the average of the log of the relative error \eqref{eq:relative_error} for the $5330$ points and the percentage of gradient estimates achieving $\theta < 1/2$ for different choices of $\sigma$, and, where appropriate, different choices of $N$. The values in bold indicate cases where the average of $\log_{10} \theta < \log_{10} 1/2$ or the percentage of gradient estimates achieving $\theta < 1/2$ is greater than 50\%, repectively.

Table~\ref{tab:num_error3} illustrates the results in the noise-free regime. For these experiments, the sampling radius was chosen as $\sigma \in \{10^{-2}, 10^{-5}, 10^{-8}\}$. As predicted by the theory, in the noise-free case as the sampling radius decreases the quality of the approximations increases. This is true for all methods. We observe that for the smoothing methods more than $4n$ samples are needed to reliably obtain $\log_{10}{\theta}<\log_{10} 1/2 \approx-0.301$ (or $\theta < 1/2$). Moreover, this experiment indicates that the relative errors ${\theta}$ for FFD, CFD and LI methods are significantly smaller than those obtained by the smoothing methods.

\begin{table}[h!]

\caption{ Average (Log) Relative Error of Gradient Approximations for $5330$ Problems ($\epsilon_f=0$).}
\footnotesize
\label{tab:num_error3}
\centering
\begin{tabular}{ccccc}
\toprule
\textbf{Method} &
$\pmb{ N}$ &
 \textbf{$\pmb{ {\sigma}} = 10^{-2}$} &
 \textbf{$\pmb{ {\sigma}} = 10^{-5}$} &
 \textbf{$\pmb{ {\sigma}} = 10^{-8}$}  \\  \midrule
\textbf{FFD} & $ {n}$ &  {-0.1651/42.68\%} &\textbf{-3.0124} / \textbf{95.10\%} &\textbf{-5.7176} / \textbf{98.57\%} \\ \hline
\textbf{CFD} & $ {n}$ &\textbf{-4.0112} / \textbf{93.41\%} &\textbf{-8.4448} / \textbf{98.76\%} &\textbf{-7.3651} / \textbf{98.57\%}  \\ \hline
\textbf{LI} & $ {n}$ &  {0.3808 / 27.64\%} &\textbf{-2.4616} / \textbf{91.44\%} &\textbf{-5.0777} / \textbf{98.22\%}  \\ \hline
\textbf{GSG} & $ {n}$ &  {0.4067 / 4.05\%} &  {-0.0060 / 6.19\%} &-0.0425 / 7.11\% \\ \hdashline
		&$2 {n}$ &  {0.3108 / 8.01\%} &-0.1252 / 14.50\% &-0.1754 / 15.91\%  \\ \hdashline
	 	&$4 {n}$ &  {0.1790 / 24.39\%} &-0.2669 / 49.74\% &\textbf{-0.3188} / \textbf{51.73\%} \\ \hdashline
	 	&$8 {n}$ &  {0.0477 / 45.82\%} &\textbf{-0.4117} / \textbf{84.00\%} &\textbf{-0.4625} / \textbf{86.85\%}  \\ \hline
 \textbf{cGSG}  &$ {n}$ &  {0.0215 / 6.19\%} &-0.0435 / 6.90\% &-0.0430 / 6.42\% \\ \hdashline
	 & $2 {n}$ &  {-0.0983 / 14.80\%} &-0.1822 / 17.58\% &-0.1723 / 15.89\% \\ \hdashline
	 & $4 {n}$ &-0.2307 / 48.05\% &\textbf{-0.3195} / \textbf{52.12\%} &\textbf{-0.3163} / \textbf{51.16\%} \\ \hdashline
	 & $8 {n}$ &\textbf{-0.3568} / \textbf{81.84\%} &\textbf{-0.4665} / \textbf{87.28\%}  &\textbf{-0.4634} / \textbf{86.40\%} \\ \hline
 \textbf{BSG} & $ {n}$ &  {0.3478 / 6.21\%} &  {-0.0823 / 12.38\%} &-0.1192 / 12.23\% \\ \hdashline
	 & $2 {n}$ &  {0.2033 / 15.59\%} &-0.2202 / 28.29\%  &-0.2609 / 29.55\% \\ \hdashline
	 & $4 {n}$ &  {0.0544 / 38.46\%} &\textbf{-0.3649} / \textbf{67.37\%} &\textbf{-0.4097} / \textbf{70.58\%} \\ \hdashline
	 & $8 {n}$ &  -0.0956 / \textbf{60.11\%} &\textbf{-0.5163} / \textbf{93.62\%} &\textbf{-0.5593} / \textbf{96.81\%}  \\ \hline
  \textbf{cBSG} & ${n}$ &  {-0.0503 / 10.38\%} &-0.1242 / 11.95\% &-0.1258 / 12.36\% \\ \hdashline
	 & $2 {n}$ &-0.1861 / 26.70\% &-0.2677 / 30.19\% &-0.2639 / 29.64\% \\ \hdashline
	 & $4 {n}$ &\textbf{-0.3247} / \textbf{66.40\%} &\textbf{-0.4109} / \textbf{70.00\%} &\textbf{-0.4125} / \textbf{71.52\%} \\ \hdashline
	 & $8 {n}$ &\textbf{-0.4625} / \textbf{91.52\%} &\textbf{-0.5593} / \textbf{97.13\%} &\textbf{-0.5677} / \textbf{96.94\%}   \\ \bottomrule
\end{tabular}
\end{table}

Table \ref{tab:num_error4} illustrates the performance of the gradient approximation in the presence of noise ($\epsilon_f = 10^{-4}$). Here the sampling radius was chosen as $\sigma \in \{10^{-1}, 10^{-2}, 10^{-3}, 10^{-4}\}$. As in the noise-free regime, it appears that overall the gradient approximations computed via FFD, CFD and LI have smaller relative errors than those obtained by the smoothing methods. Moreover, as predicted by the theory in the noisy regime one needs to carefully select the sampling radius in order to achieve the smallest relative error.

\begin{table}[h!]
\caption{ Average (Log) Relative Error of Gradient Approximations for $5330$ Problems ($\epsilon_f = 10^{-4}$).}
\footnotesize
\label{tab:num_error4}
\centering
\begin{tabular}{cccccc}
\toprule
\textbf{Method} &
$\pmb{ {N}}$&
 \textbf{$\pmb{ {\sigma}} = 10^{-1}$} &
 \textbf{$\pmb{ {\sigma}} = 10^{-2}$} &
 \textbf{$\pmb{ {\sigma}} = 10^{-3}$} &
 \textbf{$\pmb{ {\sigma}} = 10^{-4}$}  \\  \midrule
\textbf{FFD} & $ {n}$ &0.8593 / 12.03\% &-0.0827 / 41.71\%  &\textbf{-0.5450} / \textbf{58.99\%}  &0.0724 / 31.26\% \\ \hline
\textbf{CFD} & $ {n}$ &\textbf{-0.7297} / \textbf{62.61\%} &\textbf{-1.7849} / \textbf{91.48\%} & \textbf{-1.2902} / \textbf{80.56\%} &\textbf{-0.3664} / 45.52\% \\ \hline
\textbf{LI} & $ {n}$ &1.4604 / 8.37\% &0.4718 / 24.86\% &0.0841 / 38.07\% &0.7335 / 21.33\% \\ \hline
\textbf{GSG} & $ {n}$ &1.1284 / 1.67\% &0.4105 / 4.73\% &0.2262 / 4.97\% &0.4954 / 3.08\% \\ \hdashline
		&$2 {n}$ &1.0574 / 2.57\% &0.3085 / 8.39\% &0.1052 / 11.09\% &0.3728 / 7.94\% \\ \hdashline
	 	&$4 {n}$ &0.9970 / 7.49\% &0.1888 / 22.70\% &-0.0344 / 32.68\% &0.2366 / 20.24\% \\ \hdashline
	 	&$8 {n}$ &0.8835 / 14.02\% &0.0503 / 45.55\% &-0.1686 / \textbf{62.55\%} &0.0960 / 36.79\%  \\ \hline
 \textbf{cGSG}  &$ {n}$ &0.3144 / 4.37\% &0.0178 / 6.62\% &0.0178 / 6.42\% &0.2783 / 4.26\% \\ \hdashline
	 & $2 {n}$ &0.2472 / 10.19\% &-0.0988 / 14.95\% &-0.1151 / 14.33\% &0.1446 / 11.07\% \\ \hdashline
	 & $4 {n}$ &0.2049 / 28.99\% &-0.2256 / 46.62\% &-0.2499 / 42.27\% &0.0054 / 26.21\% \\ \hdashline
	 & $8 {n}$ &0.1441 / \textbf{52.51\%} &\textbf{-0.3594} / \textbf{81.97\%} &\textbf{-0.3891} / \textbf{76.68\%}  &-0.1341 / 47.35\%  \\ \hline
 \textbf{BSG} & $ {n}$ &1.0705 / 1.93\% &0.3460 / 6.42\% &0.1848 / 7.90\% &0.4919 / 4.62\% \\ \hdashline
	 & $2 {n}$ &0.9383 / 5.07\% &0.2016 / 16.12\% &0.0381 / 20.36\% &0.3473 / 11.52\% \\ \hdashline
	 & $4 {n}$ &0.8119 / 12.53\% &0.0541 / 38.03\% &-0.1149 / 45.57\% &0.1957 / 26.68\% \\ \hdashline
	 & $8 {n}$ &0.6725 / 20.49\% &-0.0954 / \textbf{59.81\%} &-0.2603 /\textbf{71.31\%} &0.0492 / 42.23\%  \\ \hline
  \textbf{cBSG} & $ {n}$ &0.2210 / 7.95\% &-0.0510 / 10.94\% &-0.0422 / 9.91\% &0.2565 / 7.67\% \\ \hdashline
	 & $2 {n}$ &0.1311 / 16.85\% &-0.1775 / 25.91\%  &-0.1833 / 24.50\% &0.1093 / 17.02\% \\ \hdashline
	 & $4 {n}$ &0.0369 / 42.20\% &\textbf{-0.3149} / \textbf{64.20\%}  &\textbf{-0.3303} / \textbf{56.85\%} &-0.0418 / 37.41\% \\ \hdashline
	 & $8 {n}$ &-0.0582 / \textbf{63.60\%} &\textbf{-0.4636} / \textbf{90.71\%}  &\textbf{-0.4754} / \textbf{86.30\%} &-0.1877 / \textbf{54.18\%}  \\ \bottomrule
\end{tabular}
\end{table}

\subsection{Performance of Line Search DFO Algorithm with Different Gradient Approximations} \label{sec:alg_perf}
The ability to approximate the gradient sufficiently accurately is a crucial ingredient of model based, and in particular line search, DFO algorithms. The numerical results presented in Section \ref{sec:num_grad_approx} illustrated the merits and limitations of the different gradient approximations. In this section, we investigate how these methods perform in conjunction with a line search DFO algorithm \cite[Algorithm 1]{berahas2019global}.

\paragraph{Mor\'e \& Wild Problems \cite{more2009benchmarking}}

Several algorithms could be considered in this section. We focus on line search DFO algorithms that either compute steepest descent search directions ($d_k = -g(x_k)$) or L-BFGS \cite{NW} search directions ($d_k = -H_kg(x_k)$). Moreover, we considered both adaptive line search variants as well as variants that used a constant, tuned step size parameter. Overall, we investigated the performance of $17$ different algorithms
. We considered algorithms that approximate the gradient using FFD, CFD and the four smoothing methods with steepest descent or L-BFGS search directions and an adaptive line search strategy. We also considered methods that approximate the gradient using the  smoothing methods with steepest descent search directions and a constant step size parameter. Finally, as a benchmark, we compared the performance of the aforementioned methods against the popular DFOTR algorithm \cite{ASBandeira_LNVicente_KScheinberg_2011}.

\begin{figure}[ht]
    \centering
    \begin{subfigure}{.32\textwidth}
        \centering
        \includegraphics[trim=35 20 30 20,clip,width=0.95\linewidth]{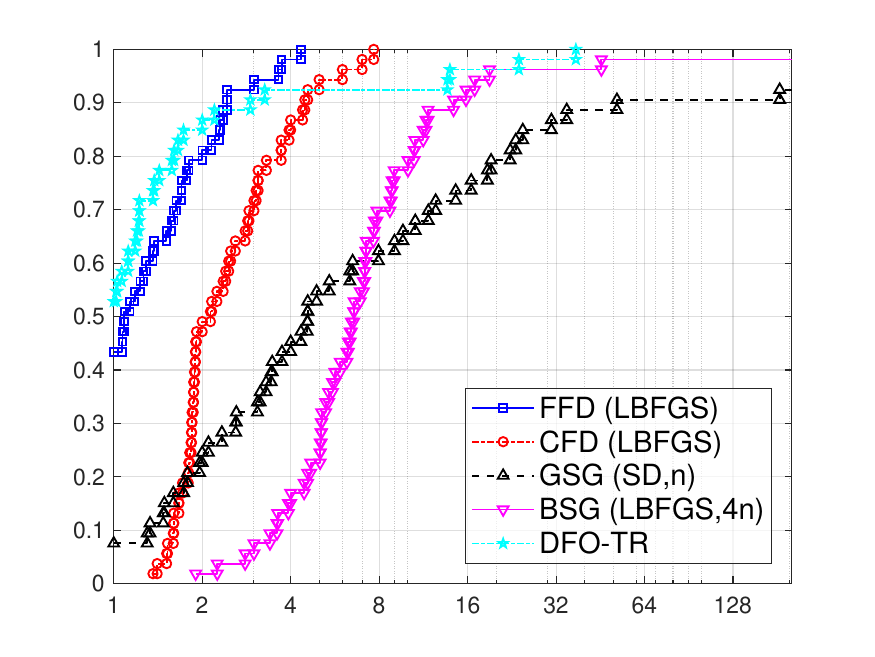}
        \caption{  $\tau = 10^{-1}$}
    \end{subfigure}%
    \begin{subfigure}{.32\textwidth}
        \centering
        \includegraphics[trim=35 20 30 20,clip,width=0.95\linewidth]{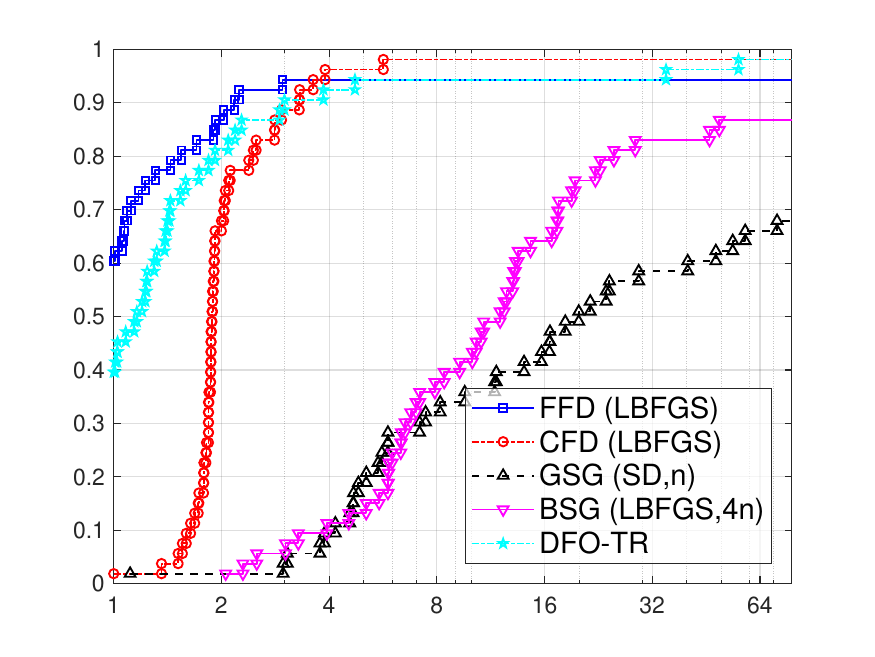}
        \caption{  $\tau = 10^{-3}$}
    \end{subfigure}%
    \begin{subfigure}{.32\textwidth}
        \centering
\includegraphics[trim=35 20 30 20,clip,width=0.95\linewidth]{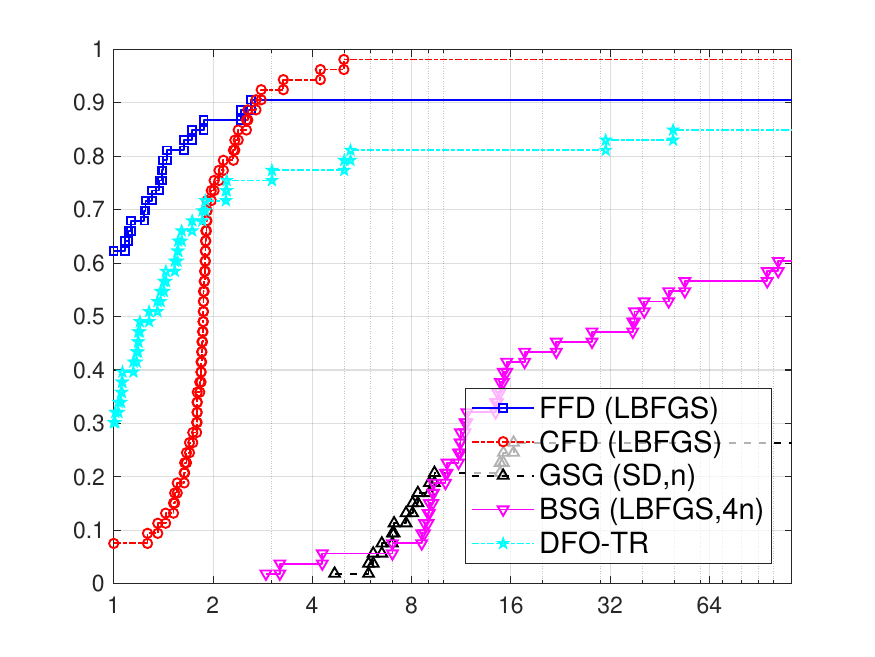}
        \caption{  $\tau = 10^{-5}$}
    \end{subfigure} \\   
    \vspace{0.25cm} 
    \begin{subfigure}{.32\textwidth}
        \centering
\includegraphics[trim=35 20 30 20,clip,width=0.95\linewidth]{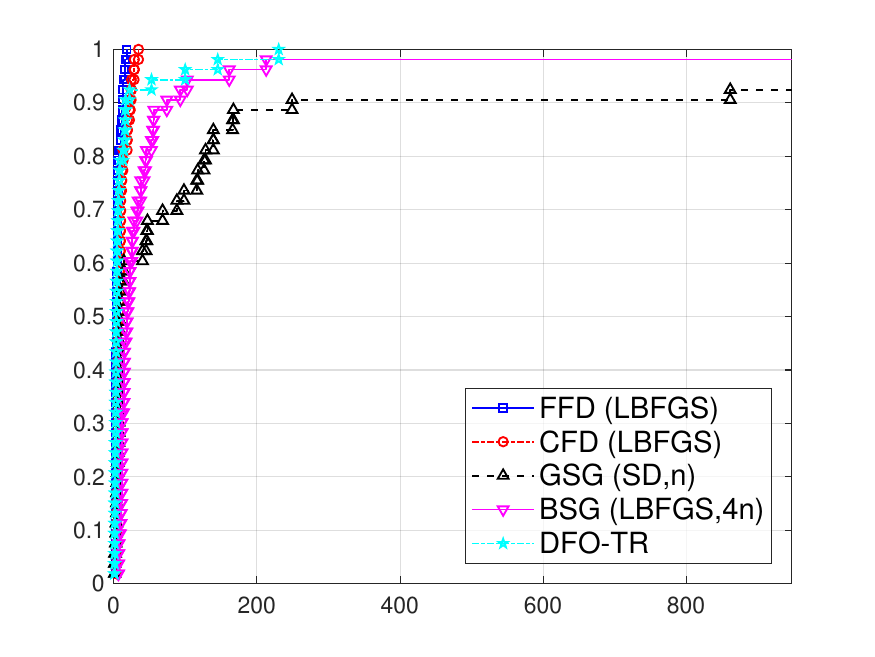}
        \caption{  $\tau = 10^{-1}$}
    \end{subfigure}%
    \begin{subfigure}{.32\textwidth}
        \centering
\includegraphics[trim=35 20 30 20,clip,width=0.95\linewidth]{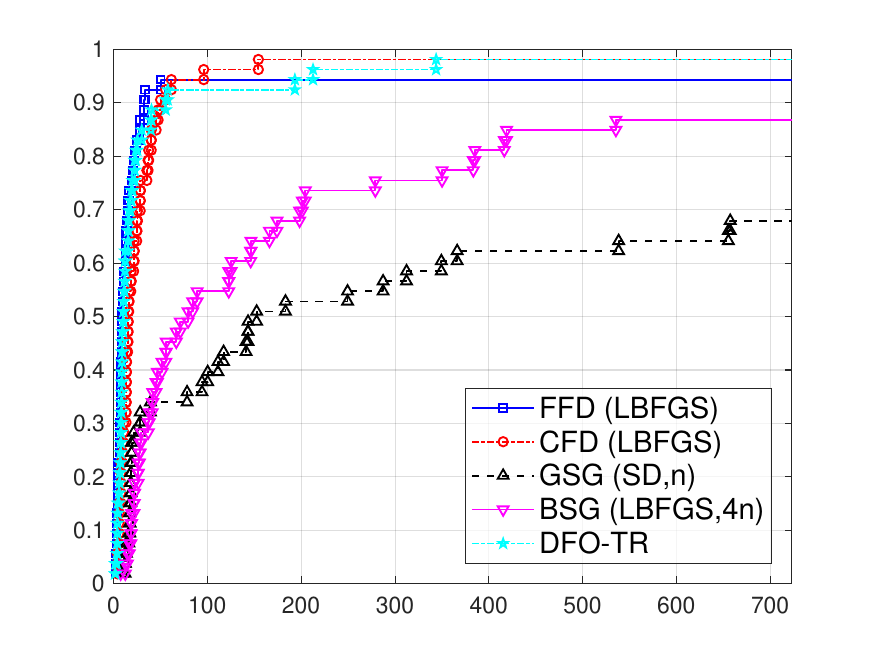}
        \caption{  $\tau = 10^{-3}$}
    \end{subfigure}%
    \begin{subfigure}{.32\textwidth}
        \centering
\includegraphics[trim=35 20 30 20,clip,width=0.95\linewidth]{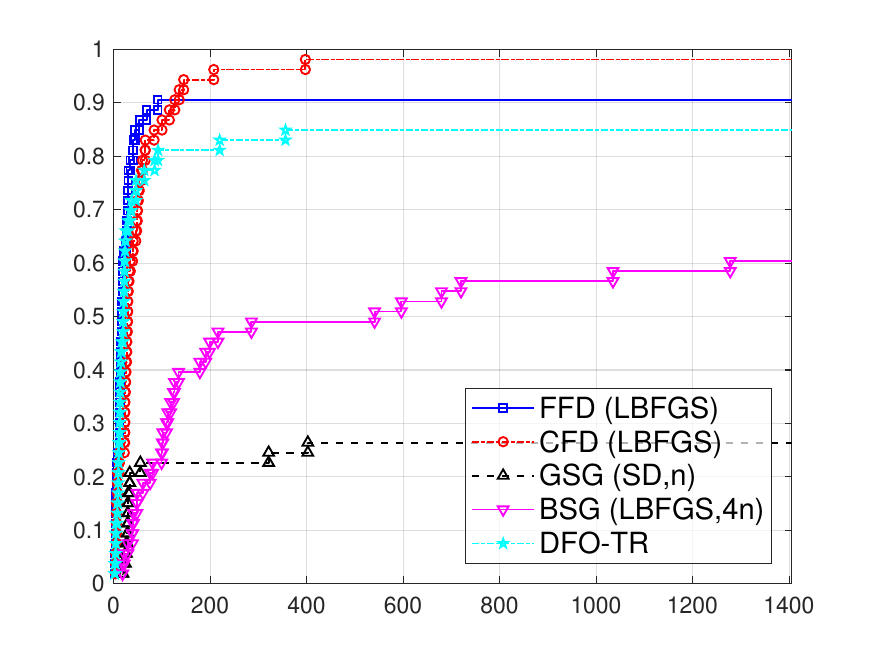}
        \caption{$\tau = 10^{-5}$}
    \end{subfigure} 
    \caption{Performance and data profiles for best variant of each method. Top row: \emph{Performance profiles}, where the x-axis represents \emph{performance ratio}; Bottom row: \emph{Data profiles}, where the x-axis represents the \emph{number of function evaluations divided by ($n+1$)}. See \cite{more2009benchmarking,DolaMore02} for more details about performance and data profiles.}  
    \label{fig: final}
\end{figure}

We tested the algorithms on the problems described in \cite{more2009benchmarking} ($53$ problems), and illustrate the performance of the methods using performance and data profiles \cite{more2009benchmarking,DolaMore02}. 
Each curve in the profile  displayed in Figure \ref{fig: final} corresponds to one algorithm's overall performance on the entire problem set. Roughly speaking, larger area under the curve indicates  better overall performance. We compare the performance of the {\em best variant} of each algorithm for different accuracy levels. For a given accuracy level $\tau \geq 0$ and problem, a method was deemed successful if for some iterate $x_k$, $\frac{f(x_0) - f(x_k)}{f(x_0) - f_L} \geq 1-\tau$ was satisfied, where $f_L$ is the best (lowest) function value achieved by any method; see \cite{more2009benchmarking} for more details.
We selected only the \emph{best performers} amongst different possible variants by first comparing the variants among themselves. For example, for FFD and CFD the LBFGS variant outperformed the steepest descent variant.  
With regards to the smoothing methods, GSG with $N=n$ samples per iteration and steepest descent search directions was the best performer out of all GSG methods, and BSG with $N=4n$ and LBFGS performed best among all BSG variants. For all the types of gradient approximations, the variants that performed the best used an adaptive step length procedure. We omit illustrations of these comparison for brevity. Finally, in Figures \ref{fig: fd_ls} and \ref{fig: fd_smoothed} we compare the adaptive step size methods versus the constant step size variants.

\begin{figure}[ht]
    \centering
    \begin{subfigure}{.32\textwidth}
        \centering
        \includegraphics[trim=35 21 30 20,clip,width=0.95\linewidth]{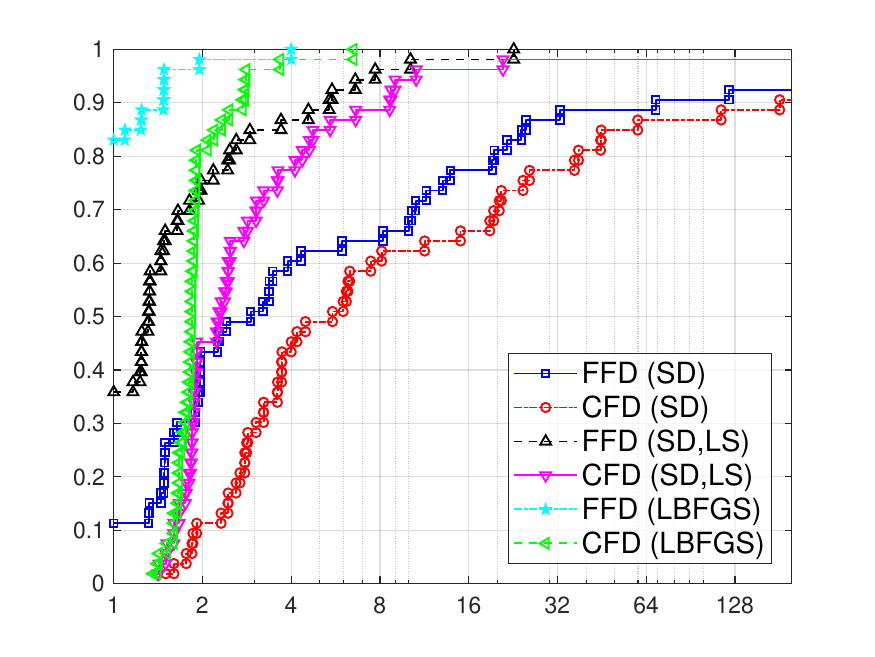}
        \caption{  $\tau = 10^{-1}$}
    \end{subfigure}%
    \begin{subfigure}{.32\textwidth}
        \centering
        \includegraphics[trim=35 21 30 20,clip,width=0.95\linewidth]{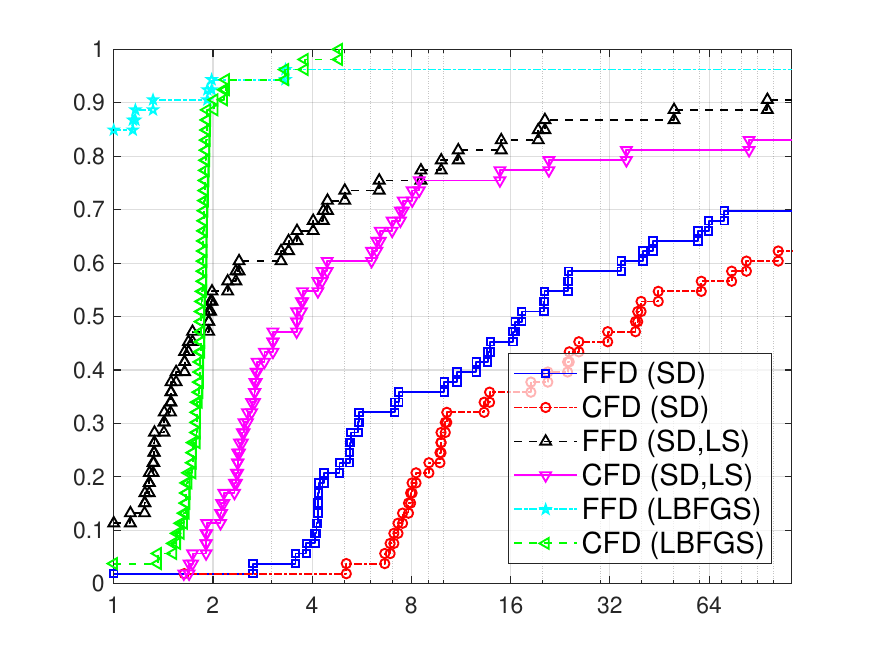}
        \caption{  $\tau = 10^{-3}$}
    \end{subfigure}%
    \begin{subfigure}{.32\textwidth}
        \centering
        \includegraphics[trim=35 21 30 20,clip,width=0.95\linewidth]{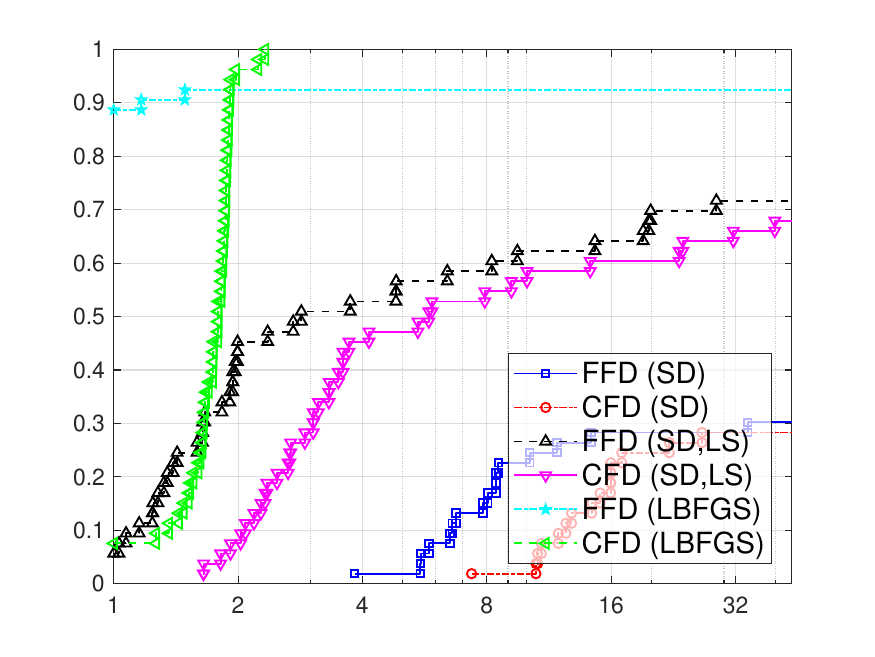}
        \caption{  $\tau = 10^{-5}$}
    \end{subfigure}
    \caption{Performance profiles for Finite Difference variants with steepest descent (SD) and LBFGS search directions; SD with and without a line search (LS). }    
    \label{fig: fd_ls}
\end{figure}

\begin{figure}[ht]
    \centering
    \begin{subfigure}{.32\textwidth}
        \centering
        \includegraphics[trim=35 21 30 20,clip,width=0.95\linewidth]{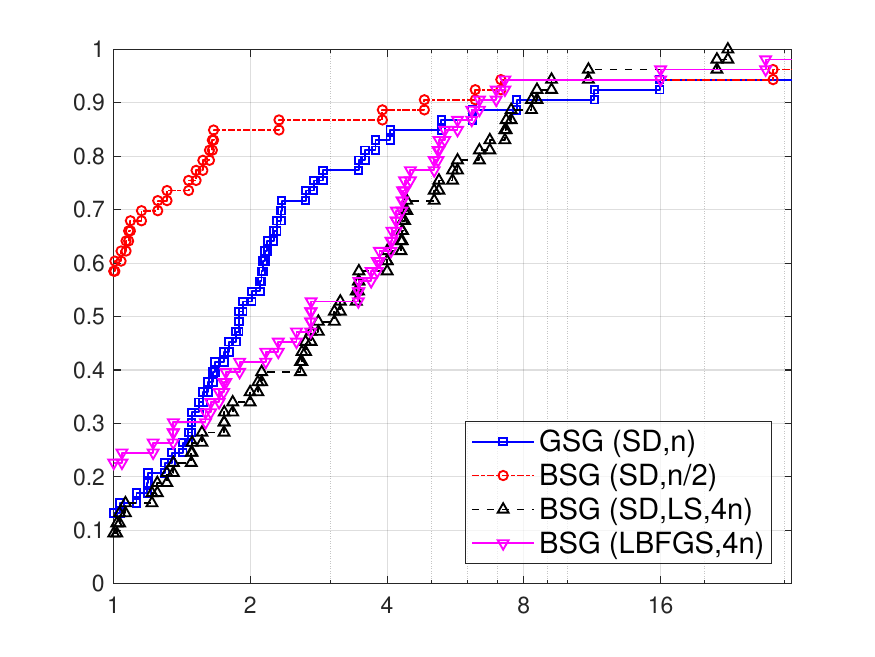}
        \caption{  $\tau = 10^{-1}$}
    \end{subfigure}%
    \begin{subfigure}{.32\textwidth}
        \centering
        \includegraphics[trim=35 21 30 20,clip,width=0.95\linewidth]{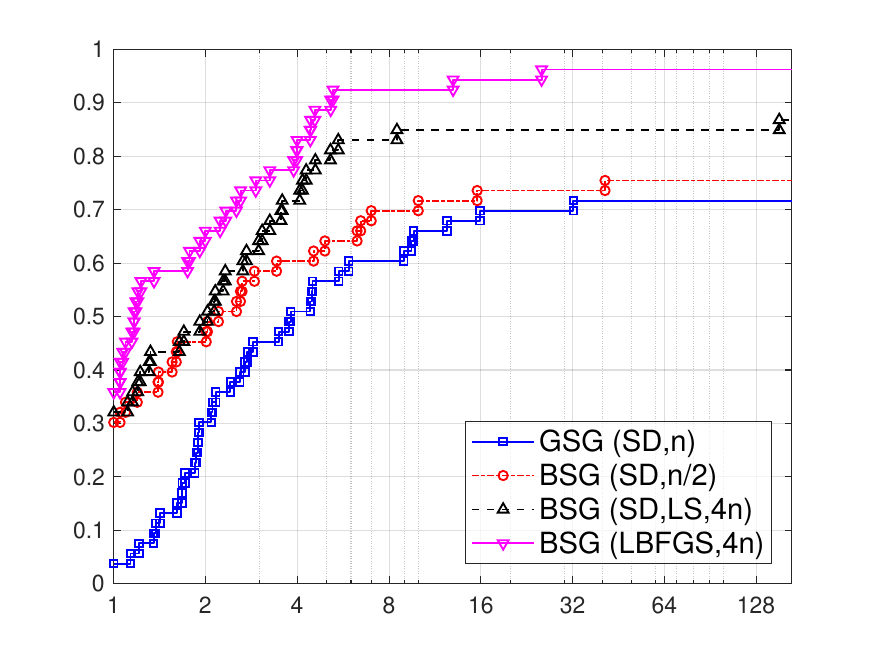}
        \caption{  $\tau = 10^{-3}$}
    \end{subfigure}%
    \begin{subfigure}{.32\textwidth}
        \centering
        \includegraphics[trim=35 21 30 20,clip,width=0.95\linewidth]{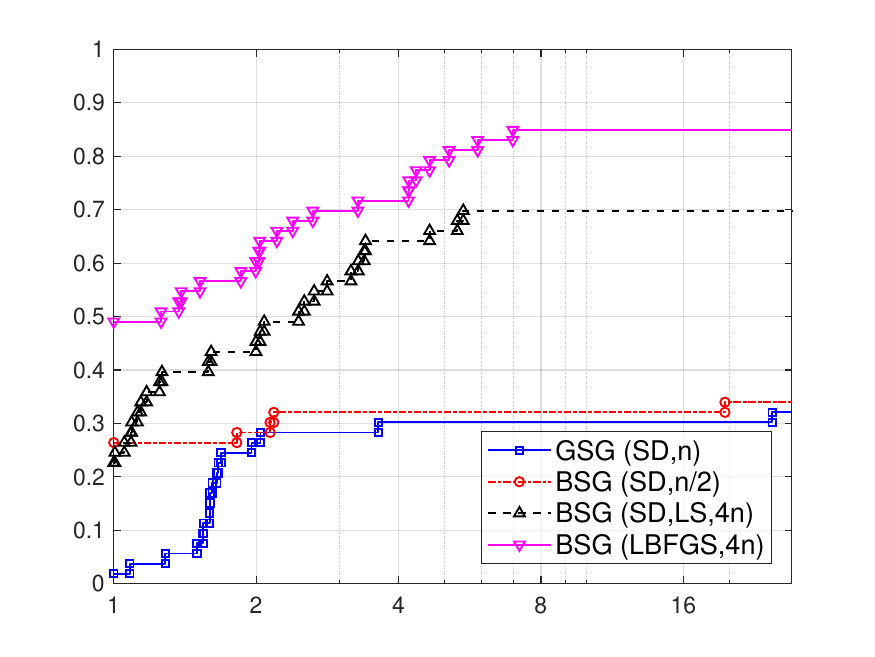}
        \caption{  $\tau = 10^{-5}$}
    \end{subfigure}
    \caption{Performance profiles for best smoothed variants with steepest descent (SD) and LBFGS search directions; SD with and without a line search (LS).}    
    \label{fig: fd_smoothed}
\end{figure}

\paragraph{Reinforcement Learning Tasks \cite{brockman2016openai}} 
In this section, we investigate the performance of the methods on noisy optimization problems. Specifically, we present numerical results for reinforcement learning tasks from $\mathrm{OpenAI}$ $\mathrm{Gym}$ library \cite{brockman2016openai}. We compare gradient based methods, where the gradients are approximated as follows:
\begin{enumerate}
	\item Forward Finite Differences (\texttt{FFD (SD)}),
	\item Linear Interpolation (\texttt{Interpolation (SD)}) and (\texttt{Interpolation (SD, LS)}),
	\item Gaussian Smoothed Gradients (\texttt{GSG (SD)}).
\end{enumerate}
For the methods that use interpolation, we implemented two different step length strategies: $(1)$ fixed step length $\alpha_k = \alpha$, and $(2)$ step length chosen via a line search.

In Figure \ref{fig: RL}, we show the average (solid lines) and max/min (dashed lines) over a number of runs. We can see that in some experiments FD did not perform well compared to other methods. \ks{This happens because FD being deterministic method may get stuck in local minima, while adding some randomness helps to escape those.} While our theory is the same for FD and Interpolation, our experiments show that for these tasks, choosing $u_i$ to be orthonormal but random helps the algorithm to avoid getting stuck in local maxima. We observe that the Interpolation method is superior to the GSG and that line-search provides some improvements over a manually tuned choice of $\alpha_k$. More details are given in the Appendix \ref{sec:rl_details}.
\begin{figure}[ht]
    \centering
    \begin{subfigure}{.32\textwidth}
        \centering
        \includegraphics[trim=5 3 40 5,clip,width=0.95\linewidth]{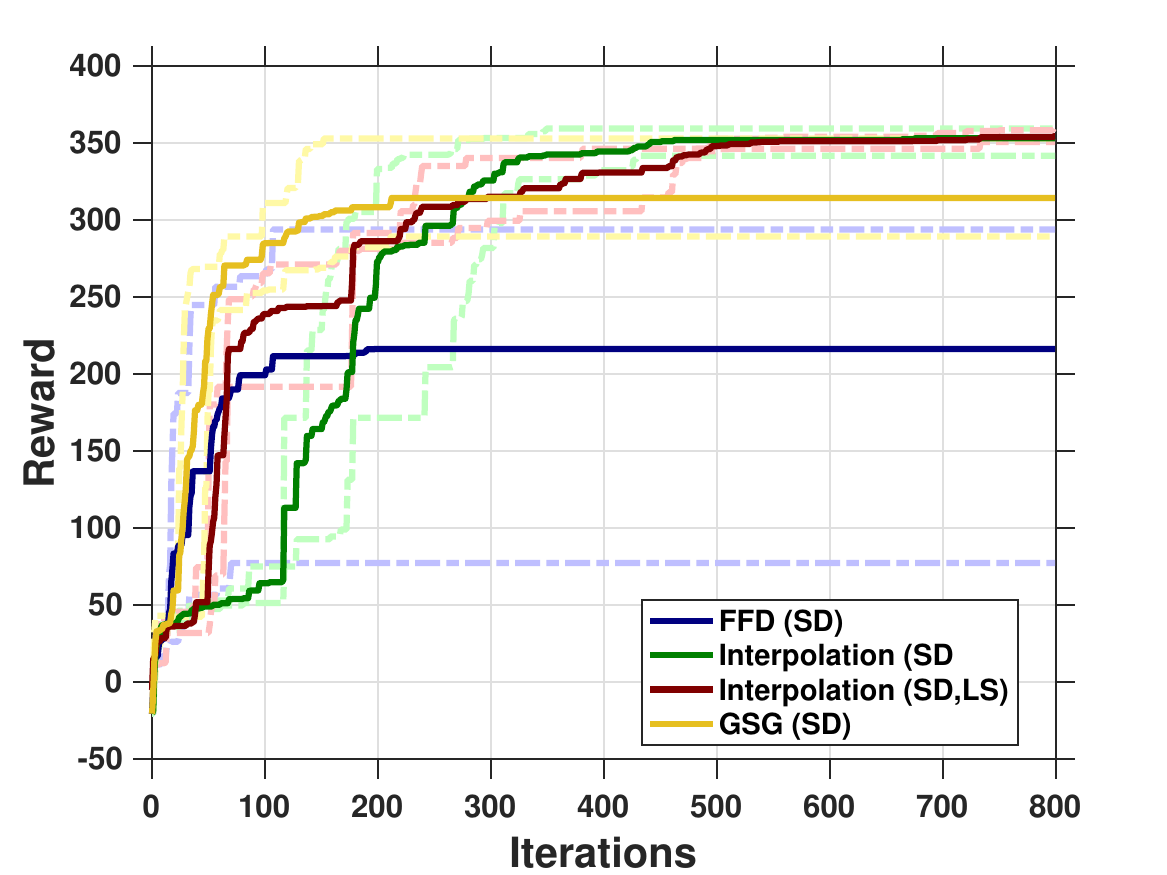}
        \caption{  Swimmer }
    \end{subfigure}%
    \begin{subfigure}{.32\textwidth}
        \centering
        \includegraphics[trim=0 3 35 5,clip,width=0.95\linewidth]{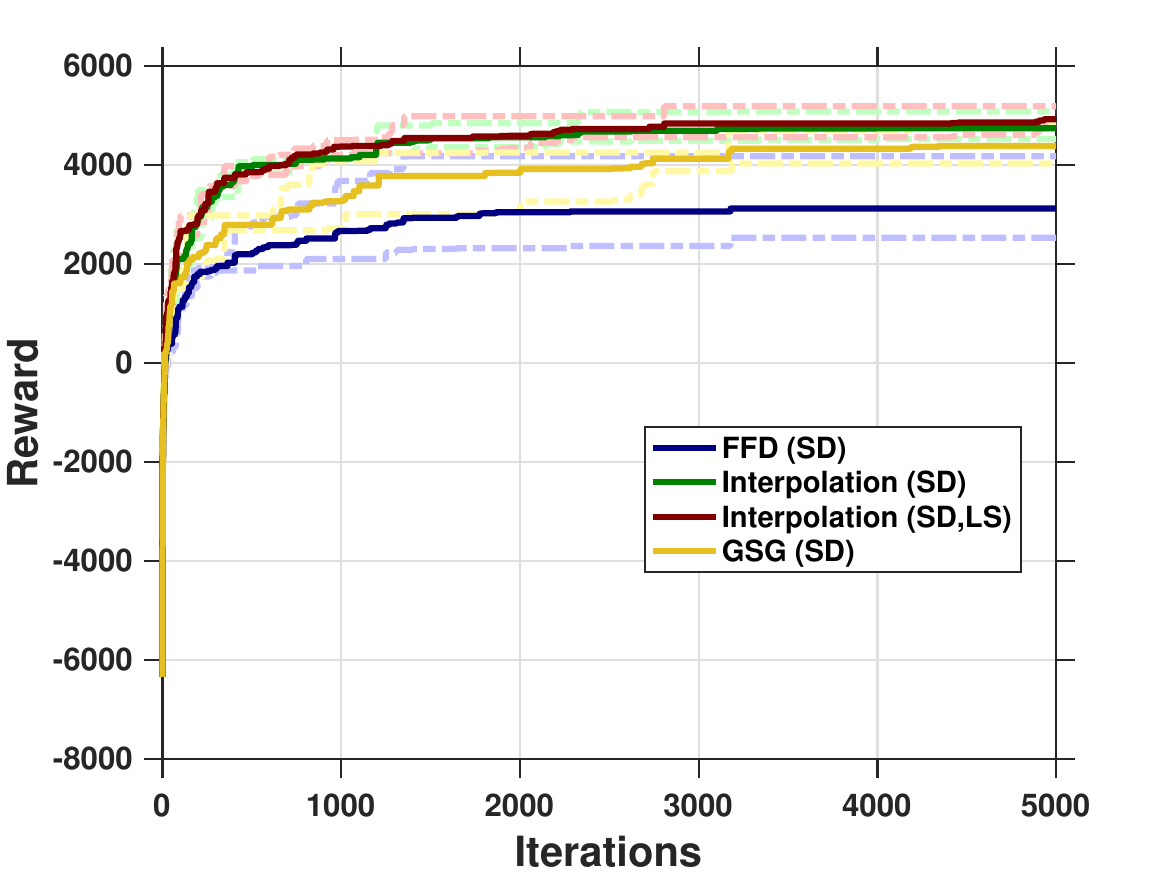}
        \caption{ HalfCheetah }
    \end{subfigure}%
    \begin{subfigure}{.32\textwidth}
        \centering
        \includegraphics[trim=5 3 35 5,clip,width=0.95\linewidth]{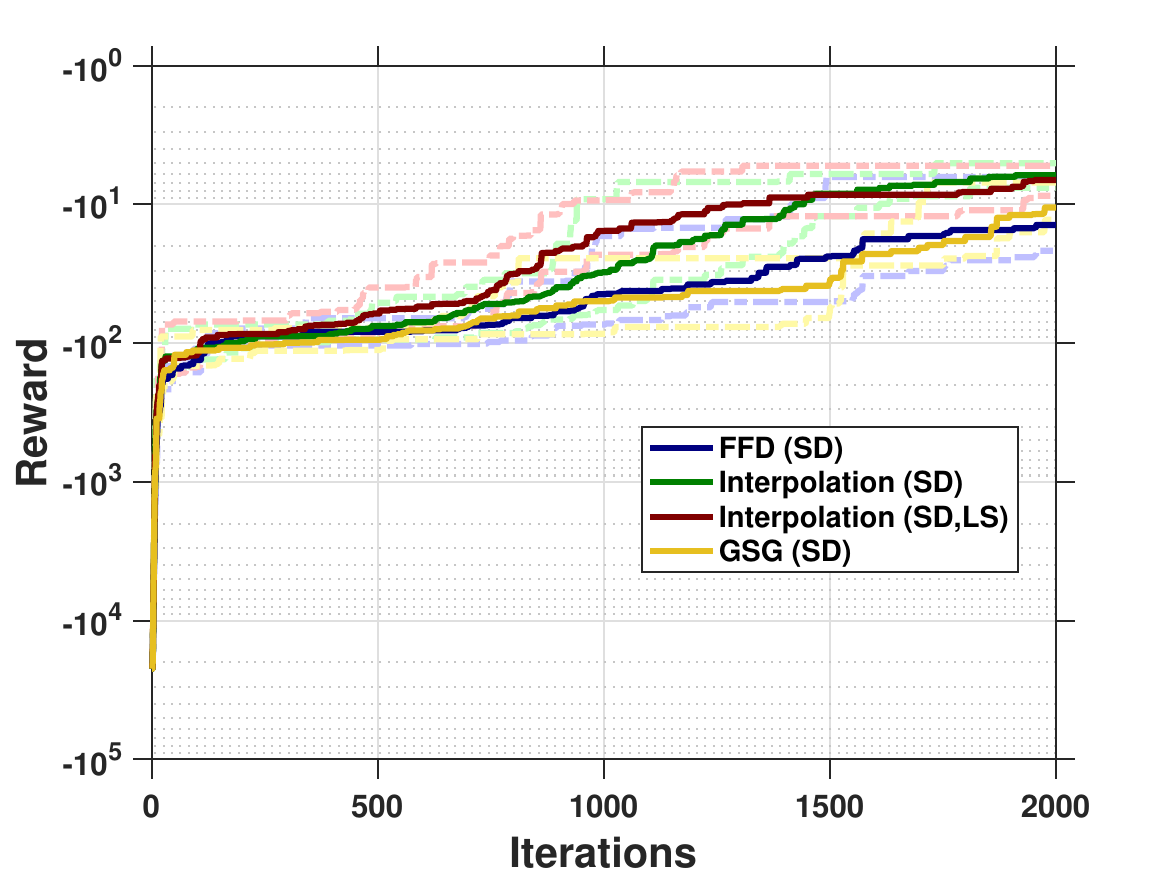}
        \caption{ Reacher }
    \end{subfigure}
    \caption{Performance of Methods on Reinforcement Learning Tasks. }    
    \label{fig: RL}
\end{figure}

\pagebreak
\section{Final Remarks} \label{sec:finrem}
We have shown that several derivative-free techniques for approximating gradients provide comparable estimates under reasonable assumptions. More specifically, we analyzed the gradient approximations constructed via finite differences, linear interpolation, Gaussian smoothing and smoothing on a unit sphere using functions values with bounded noise.  For each method, we derived bounds on the number of samples and the sampling radius which guarantee favorable convergence properties for a line search or fixed step size descent method. These approximations can be used effectively in conjunction with a line search algorithm, possibly with L-BFGS search directions, provided they are sufficiently accurate. Our theoretical results, and related numerical experiments, show that finite difference and interpolation methods are much more efficient than smoothing methods in providing good gradient approximations. The techniques presented in this paper can be extended to other distributions of the random vector $u$, as long as individual components of $u$ are symmetric and independent and identically distributed  random variables; e.g., the distribution used for constructing gradient approximations in \cite{spall2000adaptive}.

\bibliographystyle{plain}
\bibliography{Katya}

\appendix
\clearpage
\section{Derivations}
\subsection{Derivation of \eqref{eq:GSG_bound1}}	\label{app:GSG_bound1}
\begin{align*}
	\| \nabla F(x) - \nabla \phi(x))\| &= \left\| \mathbb{E}_{u \sim \cN(0,I)} \left[ \frac{1}{\sigma} f(x+\sigma u) u \right] - \nabla \phi(x)\right\| \\
	&= \left\| \mathbb{E}_{u \sim \cN(0,I)} \left[ \frac{\phi(x+\sigma u) + \epsilon(x+\sigma u)}{\sigma} u \right] - \nabla \phi(x)\right\| \\
	&= \left\| \mathbb{E}_{u \sim \cN(0,I)} \left[ \nabla \phi(x+\sigma u) - \nabla \phi(x) \right] + \mathbb{E}_{u \sim \cN(0,I)} \left[ \frac{\epsilon(x+\sigma u)}{\sigma} u \right]\right\| \\
	&\le \left\| \mathbb{E}_{u \sim \cN(0,I)} \left[ \nabla \phi(x+\sigma u) - \nabla \phi(x) \right] \right\| + \left\| \mathbb{E}_{u \sim \cN(0,I)} \left[ \frac{\epsilon(x+\sigma u)}{\sigma} u \right]\right\| \\
	&\le \mathbb{E}_{u \sim \cN(0,I)} [\|\nabla \phi(x+\sigma u) - \nabla \phi(x)\| ] + \mathbb{E}_{u \sim \cN(0,I)} \left[ \left\| \frac{\epsilon(x+\sigma u)}{\sigma} u \right\| \right] \\
	&\leq L \sigma \mathbb{E}_{u \sim \cN(0,I)} [\| u \|] + \frac{\epsilon_f}{\sigma} \mathbb{E}_{u \sim \cN(0,I)} [\| u \|] \\
	&= \left( L \sigma + \frac{\epsilon_f}{\sigma} \right) \sqrt{2}\frac{\Gamma(\frac{n+1}{2})}{\Gamma(\frac{n}{2})} \leq \sqrt{n}L \sigma + \frac{\sqrt{n} \epsilon_f}{\sigma}.
\end{align*}

\subsection{Derivation of \eqref{eq:cGSG_bound1}}	\label{app:cGSG_bound1}
\begin{align*}
	& \| \nabla F(x) - \nabla \phi(x))\| \\
	&\qquad ={} \left\| \mathbb{E}_{u \sim \cN(0,I)} \left[ \frac{\phi(x+\sigma u) + \epsilon(x+\sigma u) - \phi(x+\sigma u) - \epsilon(x+\sigma u)}{2\sigma} u \right] - \nabla \phi(x)\right\| \\
	&\qquad ={} \left\| \mathbb{E}_{u \sim \cN(0,I)} \left[\frac{1}{2}\nabla \phi(x+\sigma u) + \frac{1}{2}\nabla \phi(x-\sigma u) - \nabla \phi(x)\right] + \mathbb{E}_{u \sim \cN(0,I)} \left[ \frac{\epsilon(x+\sigma u) - \epsilon(x+\sigma u)}{2\sigma} u \right] \right\| \\
	&\qquad\le{} \frac{1}{2}\mathbb{E}_{u \sim \cN(0,I)} \left[\|\left(\nabla \phi(x+\sigma u) - \nabla \phi(x)\right) - (\nabla \phi(x) - \phi(x-\sigma u)) \| \right]\\
	&\qquad \qquad+ \mathbb{E}_{u \sim \cN(0,I)} \left[\left\| \frac{\epsilon(x+\sigma u) - \epsilon(x+\sigma u)}{2\sigma} u \right\|\right] \\
	&\qquad \le{} \frac{1}{2}\mathbb{E}_{u \sim \cN(0,I)} \left[\|\left(\nabla \phi(x+\sigma u) - \nabla \phi(x)\right) - (\nabla \phi(x) - \phi(x-\sigma u)) \| \right] + \mathbb{E}_{u \sim \cN(0,I)} \left[ \frac{\epsilon_f}{\sigma} \|u\| \right] \\
	&\qquad ={} \frac{1}{2}\mathbb{E}_{u \sim \cN(0,I)} \left[ \|\left(\nabla^2 \phi(x+ \xi_1 u) - \nabla^2 \phi(x - \xi_2 u )\right) \sigma u \| \right] + \mathbb{E}_{u \sim \cN(0,I)} \left[ \frac{\epsilon_f}{\sigma} \|u\| \right], 
\end{align*}
for some $0 \le \xi_1 \le \sigma$ and $0 \le \xi_2 \le \sigma$ by the intermediate value theorem. Then 
\begin{align*}
	\| \nabla F(x) - \nabla \phi(x))\| 
	\le{}& \frac{1}{2} \mathbb{E}_{u \sim \cN(0,I)} \left[\| \nabla^2 \phi(x+ \xi_1 u) - \nabla^2 \phi(x - \xi_2 u ) \| \| \sigma u \| \right] + \mathbb{E}_{u \sim \cN(0,I)} \left[ \frac{\epsilon_f}{\sigma} \|u\| \right] \\
	\le{}& \frac{1}{2} \mathbb{E}_{u \sim \cN(0,I)} \left[M \|\xi_1 u + \xi_2 u\| \cdot \sigma \|u\| \right] + \mathbb{E}_{u \sim \cN(0,I)} \left[ \frac{\epsilon_f}{\sigma} \|u\| \right] \\ 
	={}& \frac{1}{2} \mathbb{E}_{u \sim \cN(0,I)} \left[|\xi_1 + \xi_2| \cdot \|u\|^2 M \sigma \right] + \mathbb{E}_{u \sim \cN(0,I)} \left[ \frac{\epsilon_f}{\sigma} \|u\| \right] \\
	\le{}& n M \sigma^2 + \frac{\sqrt{n} \epsilon_f}{\sigma} .
\end{align*}

\subsection{Derivation of \eqref{eq:mvn}}	\label{app:mvn} 
For the first equality, let $A = \mathrm{E}_{u \sim \mathcal{N}(0,I)} (a^\intercal u)^2 u u ^\intercal$. 
Then for any $(i,j) \in \{1,2,\dots, n\}^2$ with $i \neq j$, we have 
\[ \begin{aligned}
    A_{ij} &= \mathrm{E} \left\{ (a^\intercal u )^2 u _i u _j \right\} \\
    &= \sum_{k=1}^n \sum_{l=1}^n \mathrm{E} \left\{ a_k u _k a_l u _l u _i u _j \right\} \\
    &= \sum_{k=i} \sum_{l=i} \mathrm{E} \left\{ a_k u _k a_l u _l u _i u _j \right\}
    + \sum_{k\neq i} \sum_{l=i} \mathrm{E} \left\{ a_k u _k a_l u _l u _i u _j \right\}
    + \sum_{k=i} \sum_{l\neq i} \mathrm{E} \left\{ a_k u _k a_l u _l u _i u _j \right\} 
    + \sum_{k\neq i} \sum_{l\neq i}  \mathrm{E} \left\{ a_k u _k a_l u _l u _i u _j \right\} \\
    &= \mathrm{E} \left\{ a_i^2 u _i^3 u _j \right\}
    + \sum_{k\neq i} \mathrm{E} \left\{ a_k a_i u _k u _i^2 u _j \right\}
    + \sum_{l\neq i} \mathrm{E} \left\{a_i a_l u _l u _i^2 u _j \right\} 
    + \mathrm{E} \left\{ u _i \right\} \sum_{k\neq i} \sum_{l\neq i}  \mathrm{E} \left\{ a_k u _k a_l u _l u _j \right\} \\
    &= 0 + \sum_{k\neq i} \mathrm{E} \left\{ a_k a_i u _k u _i^2 u _j \right\}
    + \sum_{l\neq i} \mathrm{E} \left\{a_i a_l u _l u _i^2 u _j \right\} + 0 \\
    &= \mathrm{E} \left\{ a_i a_j u _i^2 u _j^2 \right\} + \mathrm{E} \left\{ a_i a_j u _i^2 u _j^2 \right\} \\
    &= 2 a_i a_j.
\end{aligned} \]
For any $i \in \{1,2,\dots,n\}$, 
\[ \begin{aligned}
    A_{ii} &= \mathrm{E} \left\{ (a^\intercal u )^2 u _i^2 \right\} \\
    &= \sum_{k=i} \sum_{l=i} \mathrm{E} \left\{ a_k u _k a_l u _l u _i^2 \right\}
    + \sum_{k\neq i} \sum_{l=i} \mathrm{E} \left\{ a_k u _k a_l u _l u _i^2 \right\}
    + \sum_{k=i} \sum_{l\neq i} \mathrm{E} \left\{ a_k u _k a_l u _l u _i^2 \right\} 
    + \sum_{k\neq i} \sum_{l\neq i}  \mathrm{E} \left\{ a_k u _k a_l u _l u _i^2 \right\} \\
    &= \mathrm{E} \left\{ a_i^2 u _i^4 \right\}
    + \sum_{k\neq i} \mathrm{E} \left\{ a_k a_i u _k u _i^3 \right\}
    + \sum_{l\neq i} \mathrm{E} \left\{a_i a_l u _l u _i^3 \right\} 
    + \mathrm{E} \left\{ u _i^2 \right\} \sum_{k\neq i} \sum_{l\neq i}  \mathrm{E} \left\{ a_k u _k a_l u _l \right\} \\
    &= 3a_i^2 + 0 + 0 + 1 \times \sum_{k=l\neq i} \mathrm{E} \left\{ a_k u _k a_l u _l \right\}
    = 3a_i^2 + \sum_{k\neq i} \mathrm{E} \left\{ a_k^2 u_k^2 \right\} \\
    &= 3a_i^2 + \sum_{k\neq i} a_k^2 = 2 a_i^2 + \sum_{k=1}^n a_k^2.
\end{aligned}\]
Then by writting the result in matrix format, we get $\mathrm{E}_{u \sim \mathcal{N}(0,I)} \left[ (a^\intercal u)^2 u u ^\intercal \right] = a^\intercal a I + 2 a a^\intercal$. 
This result is valid for any distribution for $u$ such that $u_i$, $i \in \{1,2,\dots,n\}$ are i.i.d. and has $\mathbb{E} u_i = 0$ and $\mathbb{E} u_i^2=1$ for all $i \in \{1,2,\dots,n\}$. 

For the second equality, since the possibility density function of $\mathcal{N}(0, I)$ is even while $a^\intercal u \cdot \|u\|^k \cdot u u ^\intercal$ is an odd function, the expectation $\mathbb{E}_{u \sim \mathcal{N}({0},I)} \left[a^T u \cdot \|u\|^k \cdot u u^T \right]$ is zero. 

Because $\mathbb{E}_{u \sim \mathcal{N}({0},I)} \left[\|u\|^k u^\intercal u \right] = \mathbb{E}_{u \sim \mathcal{N}({0},I)} \left[\|u\|^{k+2} \right]$ is the $(k+2)$nd moment of a Chi distributed variable for all $k \in \mathbb{N}$, we have 
\[ \mathbb{E}_{u \sim \mathcal{N}({0},I)} \left[\|u\|^k u^\intercal u \right] = \frac{2^{1 + k/2} \Gamma((n+k+2)/2)}{\Gamma(n/2)}. 
\]
This value is also the trace of the matrix $\mathbb{E}_{u \sim \mathcal{N}({0},I)} \left[\|u\|^k u u^T \right]$. 
Considering all $n$ elements on the diagonal of this matrix are the same, we have 
\[
\mathbb{E}_{u \sim \mathcal{N}({0},I)} \left[\|u\|^k u u^T \right]= \frac{2^{1 + k/2} \Gamma((n+k+2)/2)}{n \Gamma(n/2)} I \text{ for } k = 0,1, 2, \dots .
\]
For even $k$, this quantity is equal to $\prod_{i=1}^{k/2} (n + 2i)$. 
For odd $k$, this quantity is equal to $\left[ \sqrt{2} \Gamma \left( \frac{n+1}{2} \right) \middle/ \Gamma \left( \frac{n}{2} \right) \right] \frac{1}{n} \prod_{i=1}^{(k+1)/2} (n + 2i - 1)$. Use the inequality $\sqrt{2} \left. \Gamma \left( \frac{n+1}{2} \right) \middle/ \Gamma \left( \frac{n}{2} \right) \right. \le \sqrt{n}$ for all $n \in \mathbb{N}$, we have 
\[\mathbb{E}_{u \sim \mathcal{N}({0},I)} \left[\|u\|^k u u^T \right]  \preceq (n+1)(n+3)\cdots(n+k) \cdot n^{-0.5} I \text{ for } k = 1, 3, 5, \dots . 
\]

\subsection{Derivation of \eqref{eq:lower_req_eq1}}	\label{app:lb_1} 

\begin{align*}
\mathbb{E} \left[\|g(x) - \nabla F(x)\|^2\right] 
={}& \mathbb{E} \left[ \left \|\frac{1}{N} \sum_{i=1}^N \frac{f(x+\sigma u_i) - f(x)}{\sigma} u_i - \nabla F(x) \right\|^2 \right] \\ 
={}& \frac{1}{N} \mathbb{E}_{u \sim \cN(0,I)} \left[ \left(\frac{f(x+\sigma u) - f(x)}{\sigma} \right)^2 u^\intercal u \right] - \frac{1}{N} \nabla F(x)^\intercal  \nabla F(x) \\
={}& \frac{1}{N} \mathbb{E}_{u \sim \cN(0,I)} \left[ \left(a^\intercal u\right)^2 u^\intercal u \right] - \frac{1}{N} a^\intercal  a \\ 
={}& \frac{1}{N} (n + 1) a^\intercal a. 
\end{align*}

\subsection{Derivation of \eqref{eq:lower_req_eq2}}	\label{app:lb_2}

The expression for $E\left[\|g(x) - \nabla F(x)\|^4\right]$ is a sum of $N^4$ terms with each term being the product of four vectors: 
\[ 
\mathbb{E} \left[\|g(x) - \nabla F(x)\|^4\right] = \frac{1}{N^4} \mathbb{E} \left[ \sum_{i=1}^N \sum_{j=1}^N \sum_{k=1}^N \sum_{l=1}^N \prod_{w \in \{i,j,k,l\}} \left( \frac{f(x+\sigma u_w) - f(x)}{\sigma} u_w - \nabla F(x) \right) \right], 
\]
where $\prod$ denotes the operation which is a product of the inner products of the two pairs of vectors. Specifically,  given  four vectors $a_1, a_2, a_3, a_4\in \R^n$, $\prod_{i \in \{1,2,3,4\}} a_i = (a_1^\intercal a_2) \cdot (a_3^\intercal a_4)$ and $\prod_{i \in \{1,1,2,2\}} a_i = (a_1^\intercal a_1) \cdot (a_2^\intercal a_2)$. 

We first observe that $\prod_{w \in \{i,j,k,l\}} \left( \frac{f(x+\sigma u_w) - f(x)}{\sigma} u_w - \nabla F(x) \right)=0$
whenever one of the indices $(i,j,k,l)$ is different from all of the other ones. This is because all $u_w$,  for  $w \in \{i,j,k,l\}$ are independent of each other if their indices are different and 
\[
\mathbb{E} \left[ \frac{f(x+\sigma u_w) - f(x)}{\sigma} u_w - \nabla F(x) \right] = 0. 
\]

Thus we need only to consider the terms having one of the following conditions: 
\begin{enumerate}
	\item $i=j=k=l$; 
	\item $i=j \neq k=l$; 
	\item $i=k \neq j=l$; 
	\item $i=l \neq j=k$. 
\end{enumerate}


First we consider the case: $i=j \neq k=l$, which occurs when $N>1$. 
\[ \begin{aligned}
&\mathbb{E} \left[ \sum_{i=1}^N \sum_{k=1, k\neq i}^N \prod_{w \in \{i,i,k,k\}} \left( \frac{f(x+\sigma u_w) - f(x)}{\sigma} u_w - \nabla F(x) \right) \right] \\
&\qquad ={} \sum_{i=1}^N \mathbb{E} \left[ \left\| \frac{f(x+\sigma u_i) - f(x)}{\sigma} u_i - \nabla F(x) \right\|^2 \right] \cdot \sum_{k=1,k\neq i}^N \mathbb{E} \left[ \left\| \frac{f(x+\sigma u_k) - f(x)}{\sigma} u_k - \nabla F(x) \right\|^2 \right] \\
&\qquad ={} N(N-1) \left[ (n + 1) a^\intercal a \right]^2. 
\end{aligned} \]

We now consider  two other cases: $i=k \neq j=l$ and $i=l \neq j=k$ that are essentially the same. We have 
\[ \begin{aligned}
&\mathbb{E} \left[ \sum_{i=1}^N \sum_{k=1, k\neq i}^N \prod_{w \in \{i,k,i,k\}} \left( \frac{f(x+\sigma u_w) - f(x)}{\sigma} u_w - \nabla F(x) \right) \right] \\
&\qquad ={} \sum_{i=1}^N \sum_{k=1, k\neq i}^N \mathbb{E} \left\{ \left[ \left( \frac{f(x+\sigma u_i) - f(x)}{\sigma} u_i - \nabla F(x) \right)^\intercal \left( \frac{f(x+\sigma u_k) - f(x)}{\sigma} u_k - \nabla F(x) \right) \right]^2 \right\} \\
&\qquad ={} \sum_{i=1}^N \sum_{k=1, k\neq i}^N \mathbb{E} \left( \left\{ \left[ (a^\intercal u_i) u_i - a \right]^\intercal \left[ (a^\intercal u_k) u_k - a \right] \right\}^2 \right) \\
&\qquad ={} \sum_{i=1}^N \sum_{k=1, k\neq i}^N \mathbb{E} \left( \left[ (a^\intercal u_i) (a^\intercal u_k) (u_i^\intercal u_k)  - (a^\intercal u_i)^2 - (a^\intercal u_k)^2 + a^\intercal a \right]^2 \right) \\
&\qquad ={} \sum_{i=1}^N \sum_{k=1, k\neq i}^N \mathbb{E} \left[ \begin{array}{rl}
	&(a^\intercal u_i)^2 (a^\intercal u_k)^2 (u_i^\intercal u_k)^2 + (a^\intercal u_i)^4 + (a^\intercal u_k)^4 + (a^\intercal a)^2 \\
	+& 2 (a^\intercal a) (a^\intercal u_i) (a^\intercal u_k) (u_i^\intercal u_k) - 2(a^\intercal a) (a^\intercal u_i)^2 - 2(a^\intercal a) (a^\intercal u_k)^2 \\
	-& 2 (a^\intercal u_i)^3 (a^\intercal u_k) (u_i^\intercal u_k) - 2 (a^\intercal u_i) (a^\intercal u_k)^3 (u_i^\intercal u_k) + 2 (a^\intercal u_i)^2 (a^\intercal u_k)^2
\end{array} \right]  \\
&\qquad ={} \sum_{i=1}^N \sum_{k=1, k\neq i}^N \left[ \begin{array}{rl}
	& (n+8) (a^\intercal a)^2 + 3(a^\intercal a)^2 + 3(a^\intercal a)^2  + (a^\intercal a)^2 \\
	+& 2 (a^\intercal a)^2 - 2(a^\intercal a)^2 - 2(a^\intercal a)^2 \\
	-& 6 (a^\intercal a)^2 - 6 (a^\intercal a)^2 + 2 (a^\intercal a)^2
\end{array} \right] \\
&\qquad ={} \sum_{i=1}^N \sum_{k=1, k\neq i}^N (n+3) (a^\intercal a)^2 = N(N-1) (n+3) (a^\intercal a)^2
\end{aligned} \]

Finally, we have the $i=j=k=l$ case: 
\[ \begin{aligned}
 &\mathbb{E} \left[ \sum_{i=1}^N \prod_{w \in \{i,i,i,i\}} \left( \frac{f(x+\sigma u_w) - f(x)}{\sigma} u_w - \nabla F(x) \right) \right] \\
&\qquad ={} N \mathbb{E}_{u \sim \cN(0,I)} \left[ \left\| \frac{f(x+\sigma u) - f(x)}{\sigma} u - \nabla F(x) \right\|^4 \right] \\
&\qquad ={} N \mathbb{E}_{u \sim \cN(0,I)} \left\{ \left[ \left(\frac{f(x+\sigma u) - f(x)}{\sigma}\right)^2 u^\intercal u - 2 \left(\frac{f(x+\sigma u) - f(x)}{\sigma}\right) u^\intercal \nabla F(x) + \nabla F(x)^\intercal \nabla F(x) \right]^2 \right\} \\
&\qquad ={} N \mathbb{E}_{u \sim \cN(0,I)} \left[ \begin{array}{rl}
&\left(\frac{f(x+\sigma u) - f(x)}{\sigma}\right)^4 (u^\intercal u)^2 
+ 4 \left(\frac{f(x+\sigma u) - f(x)}{\sigma}\right)^2 \left( u^\intercal \nabla F(x) \right)^2 \\
+& \left( \nabla F(x)^\intercal \nabla F(x) \right)^2 
- 4 \left(\frac{f(x+\sigma u) - f(x)}{\sigma}\right)^3  (u^\intercal u) \left( u^\intercal \nabla F(x) \right) \\
-& 4 \left(\frac{f(x+\sigma u) - f(x)}{\sigma}\right) (u^\intercal \nabla F(x)) (\nabla F(x)^\intercal \nabla F(x)) \\ 
+& 2 \left(\frac{f(x+\sigma u) - f(x)}{\sigma}\right)^2 (u^\intercal u) (\nabla F(x)^\intercal \nabla F(x))
\end{array} \right] \\
&\qquad ={} N \mathbb{E}_{u \sim \cN(0,I)} \left[ \begin{array}{rl}
&\left( a^\intercal u \right)^4 (u^\intercal u)^2 
+ 4 \left( a^\intercal u \right)^2 \left( u^\intercal a \right)^2 
+ \left( a^\intercal a \right)^2 
- 4 \left( a^\intercal u \right)^3  (u^\intercal u) \left( u^\intercal a \right) \\
-& 4 \left( a^\intercal u \right) (u^\intercal a) (a^\intercal a) 
+ 2 \left( a^\intercal u \right)^2 (u^\intercal u) (a^\intercal a)
\end{array} \right] \\
&\qquad ={} N \left[ \begin{array}{rl}
& 3 (n+4)(n+6) (a^\intercal a)^2 + 12 (a^\intercal a)^2 + (a^\intercal a)^2 - 12 (n+4) (a^\intercal a)^2 \\
-& 4(a^\intercal a)^2 + 2 (n+2) (a^\intercal a)^2 
\end{array} \right] \\
&\qquad ={} N (3n^2 + 20n + 37) (a^\intercal a)^2
\end{aligned} \]
In summary, we have 
\[
\begin{aligned}
&N^4 \mathbb{E} \left[ \|g(x) - \nabla F(x)\|^4 \right] \\
&\qquad ={} N(N-1) (n+1)^2 (a^\intercal a)^2 + 2N(N-1)(n+3) (a^\intercal a)^2 +N(3n^2 + 20n + 37) (a^\intercal a)^2 \\
&\qquad ={} N(N-1)(n^2 + 4n + 7) (a^\intercal a)^2 + N(3n^2 + 20n + 37) (a^\intercal a)^2. 
\end{aligned}
\]

\subsection{Derivation of \eqref{eq:BSG_bound1}}	\label{app:BSG_bound1}

\begin{align*}
	\| \nabla F(x) - \nabla \phi(x))\| &= \left\| \mathbb{E}_{u \sim \mathcal{U}(\mathcal{S}(0,1))} \left[ \frac{n}{\sigma} f(x+\sigma u) u \right] - \nabla \phi(x)\right\| \\
	&= \left\| \mathbb{E}_{u \sim \mathcal{U}(\mathcal{S}(0,1))} \left[ \frac{n}{\sigma} (\phi(x+\sigma u) + \epsilon(x+\sigma u)) u \right] - \nabla \phi(x)\right\| \\
	&= \left\| \mathbb{E}_{u \sim \mathcal{U}(\mathcal{B}(0,1))} \left[ \nabla \phi(x+\sigma u) - \nabla \phi(x) \right] + \mathbb{E}_{u \sim \mathcal{U}(\mathcal{S}(0,1))} \left[ \frac{n\epsilon(x+\sigma u)}{\sigma} u \right]\right\| \\
	&\le \left\| \mathbb{E}_{u \sim \mathcal{U}(\mathcal{B}(0,1))} \left[ \nabla \phi(x+\sigma u) - \nabla \phi(x) \right] \right\| + \left\| \mathbb{E}_{u \sim \mathcal{U}(\mathcal{S}(0,1))} \left[ \frac{n\epsilon(x+\sigma u)}{\sigma} u \right]\right\| \\
	&\le \mathbb{E}_{u \sim \mathcal{U}(\mathcal{B}(0,1))} [\|\nabla \phi(x+\sigma u) - \nabla \phi(x)\| ] + \mathbb{E}_{u \sim \mathcal{U}(\mathcal{S}(0,1))} \left[ \left\| \frac{n\epsilon(x+\sigma u)}{\sigma} u \right\| \right] \\
	&\leq L \sigma \mathbb{E}_{u \sim \mathcal{U}(\mathcal{B}(0,1))} [\| u \|] + \frac{n \epsilon_f}{\sigma} \mathbb{E}_{u \sim \mathcal{U}(\mathcal{S}(0,1))} [\| u \|] \\
	&= L \sigma \frac{n}{n+1} + \frac{n \epsilon_f}{\sigma} \leq L \sigma + \frac{n \epsilon_f}{\sigma}.
\end{align*}

\subsection{Derivation of \eqref{eq:cBSG_bound1}}	\label{app:cBSG_bound1}

\begin{align*}
	&  \| \nabla F(x) - \nabla \phi(x))\| \\
	&\qquad ={} \left\| \mathbb{E}_{u \sim \mathcal{U}(\mathcal{S}(0, 1))} \left[ \frac{n}{2\sigma} (\phi(x+\sigma u) + \epsilon(x+\sigma u) - \phi(x+\sigma u) - \epsilon(x+\sigma u)) u \right] - \nabla \phi(x)\right\| \\
	&\qquad ={} \left\| \mathbb{E}_{u \sim \mathcal{U}(\mathcal{B}(0, 1))} \left[\frac{1}{2}\nabla \phi(x+\sigma u) + \frac{1}{2}\nabla \phi(x-\sigma u) - \nabla \phi(x)\right] + \mathbb{E}_{u \sim \mathcal{U}(\mathcal{S}(0, 1))} \left[ \frac{n}{2\sigma} (\epsilon(x+\sigma u) - \epsilon(x+\sigma u)) u \right] \right\| \\
	&\qquad \le{} \frac{1}{2}\mathbb{E}_{u \sim \mathcal{U}(\mathcal{B}(0, 1))} \left[\|\left(\nabla \phi(x+\sigma u) - \nabla \phi(x)\right) - (\nabla \phi(x) - \phi(x-\sigma u)) \| \right]\\
	&\qquad \qquad + \mathbb{E}_{u \sim \mathcal{U}(\mathcal{S}(0, 1))} \left[\left\| \frac{n}{2\sigma} (\epsilon(x+\sigma u) - \epsilon(x+\sigma u)) u \right\|\right] \\
	&\qquad \le{} \frac{1}{2}\mathbb{E}_{u \sim \mathcal{U}(\mathcal{B}(0, 1))} \left[\|\left(\nabla \phi(x+\sigma u) - \nabla \phi(x)\right) - (\nabla \phi(x) - \phi(x-\sigma u)) \| \right] + \mathbb{E}_{u \sim \mathcal{U}(\mathcal{S}(0, 1))} \left[ \frac{n \epsilon_f}{\sigma} \|u\| \right] \\
	&\qquad ={} \frac{1}{2}\mathbb{E}_{u \sim \mathcal{U}(\mathcal{B}(0, 1))} \left[ \|\left(\nabla^2 \phi(x+ \xi_1 u) - \nabla^2 \phi(x - \xi_2 u )\right) \sigma u \| \right] + \mathbb{E}_{u \sim \mathcal{U}(\mathcal{S}(0, 1))} \left[ \frac{n \epsilon_f}{\sigma} \|u\| \right], 
\end{align*}
for some $0 \le \xi_1 \le \sigma$ and $0 \le \xi_2 \le \sigma$ by the intermediate value theorem. Then 
\begin{align*}
	\| \nabla F(x) - \nabla \phi(x))\| 
	\le{}& \frac{1}{2} \mathbb{E}_{u \sim \mathcal{U}(\mathcal{B}(0, 1))} \left[\| \nabla^2 \phi(x+ \xi_1 u) - \nabla^2 \phi(x - \xi_2 u ) \| \| \sigma u \| \right] + \mathbb{E}_{u \sim \mathcal{U}(\mathcal{S}(0, 1))} \left[ \frac{n \epsilon_f}{\sigma} \|u\| \right] \\
	\le{}& \frac{1}{2} \mathbb{E}_{u \sim \mathcal{U}(\mathcal{B}(0, 1))} \left[M \|\xi_1 u + \xi_2 u\| \cdot \sigma \|u\| \right] + \mathbb{E}_{u \sim \mathcal{U}(\mathcal{S}(0, 1))} \left[ \frac{n \epsilon_f}{\sigma} \|u\| \right] \\ 
	={}& \frac{1}{2} \mathbb{E}_{u \sim \mathcal{U}(\mathcal{B}(0, 1))} \left[|\xi_1 + \xi_2| \cdot \|u\|^2 M \sigma \right] + \mathbb{E}_{u \sim \mathcal{U}(\mathcal{S}(0, 1))} \left[ \frac{n \epsilon_f}{\sigma} \|u\| \right] \\
	\le{}& M \sigma^2 + \frac{n \epsilon_f}{\sigma} .
\end{align*}

\subsection{Derivation of \eqref{eq:uos}}	\label{app:uos} 
The first and third equalities of \eqref{app:uos} comes from the first and third equalities of \eqref{eq:mvn}. 
Considering any vector of iid Gaussian $v$, dividing by its own norm, can be expressed as $v = \|v\| u$. Moreover, $\|v\|$ and $u$ are independent. Thus any homogeneous polynomial $p$ in the entries of $u$ of degree $k$ has the property that 
\[ \mathbb{E}_{u \sim \mathcal{U}(\mathcal{S}(0,1))} [ p(u) ] = \frac{\mathbb{E}_{v \sim \mathcal{N}(0,I)} [ p(v) ]}{\mathbb{E}_{v \sim \mathcal{N}(0,I)} \|v\|^k}. 
\]
Then 
\[ \begin{aligned}
\mathbb{E}_{u \sim \mathcal{U}(\mathcal{S}(0,1))} \left[(a^T u)^2 u u^T \right] &= \frac{\mathbb{E}_{u \sim \mathcal{N}(0,I))} \left[(a^T u)^2 u u^T \right]}{\mathbb{E}_{u \sim \mathcal{N}(0,I))} \| u \|^4} = \frac{a^T a I + 2 a a^T}{n(n+2)} \\
\mathbb{E}_{u \sim \mathcal{U}(\mathcal{S}(0,1))} \left[ \|u\|^k u u^T \right] &= \frac{\mathbb{E}_{u \sim \mathcal{N}(0,I))} \left[ \|u\|^k u u^T \right]}{\mathbb{E}_{u \sim \mathcal{N}(0,I))} \|u\|^{k+2}} = \frac{1}{n} I. 
\end{aligned} \]

The second equality of \eqref{eq:uos} being 0 follows the same argument as that for the second equality of \eqref{eq:mvn}. 

\newpage
\section{Additional Details: RL Experiments}\label{sec:rl_details}
In all RL experiments the blackbox function $f$ takes as input the parameters of the policy $\pi_{\theta}:\mathcal{S} \rightarrow \mathcal{A}$ which maps states to proposed actions. The output of $f$ is the total reward obtained by an agent applying that particular policy $\pi_{\theta}$ in the given environment.

To encode policies $\pi_{\theta}$, we used fully-connected feedforward neural networks with two hidden layers, each of $h=41$ neurons and with $\mathrm{tanh}$ nonlinearities. The matrices of connections were encoded by low-displacement rank neural networks (see \cite{choro2}), as in several recent papers on applying orthogonal directions in gradient estimation for ES methods in reinforcement learning.
We did not apply any additional techniques such as state/reward renormalization, ranking or filtering, in order to solely focus on the evaluation of the presented proposals.

All experiments were run with hyperparameter $\sigma=0.1$. 
Experiments that did not apply line search were run with the use of $\mathrm{Adam}$ optimizer and $\alpha = 0.01$. For line search experiments, we were using adaptive $\alpha$ that was updated via Armijo condition with Armijo parameter $c_{1}=0.2$ and backtracking factor $\tau=0.3$.

Finally, in order to construct orthogonal samples, at each iteration we were conducting orthogonalization of random Gaussian matrices with entries taken independently at random from $\mathcal{N}(0,1)$ via Gram-Schmidt procedure (see \cite{choro2}). Instead of the orthogonalization of Gaussian matrices, we could take advantage of constructions, where orthogonality is embedded into the structure (such as random Hadamard matrices from \cite{choro2}), introducing extra bias but proven to work well in practice. However in all conducted experiments that was not necessary.

For each environment and each method we run $k=3$ experiments corresponding to different random seeds.


\end{document}